\documentclass[11pt]{amsart}
\usepackage{epsf, amsmath, accents, amsfonts, amscd, amssymb, amsthm, latexsym,verbatim,
dsfont, enumerate,url, graphicx,multicol,times,color,comment}
\usepackage[shortlabels]{enumitem}
\usepackage[normalem]{ulem}
\usepackage[all]{xy}  
\usepackage{hyperref}
\usepackage[usenames,dvipsnames]{xcolor}

\def\N{{\mathbb N}}
\def\R{\mathbb{R}}
\def\Z{{\mathbb Z}}

\def\C{{\mathbb C}}
\def\E{{\mathbb E}}
\def\F{{\mathbb F}}
\def\G{{\mathbb G}}
\def\W{{\mathbb W}}
\def\Q{{\mathbb Q}}

\newcommand{\CP}{\mathbb{C}\mathbb{P}}

\def\b{\beta}
\def\d{\delta}
\def\de{\delta}

\newcommand{\eps}{\epsilon}
\def\s{\sigma}

\def\g{\gamma}

\def\s{\sigma}

\def\io{\iota}

\def\l{\lambda}

\def\G{\Gamma}
\def\Si{\Sigma}

\def\Om{\Omega}

\def\cB{{\mathcal B}}
\def\cC{{\mathcal C}}

\def\cE{{\mathcal E}}
\def\cF{{\mathcal F}}
\def\cG{{\mathcal G}}

\def\cJ{{\mathcal J}}
\def\cK{{\mathcal K}}
\def\cL{{\mathcal L}}
\def\cM{{\mathcal M}}

\def\cO{{\mathcal O}}
\def\cP{{\mathcal P}}

\def\cR{{\mathcal R}}

\def\cU{{\mathcal U}}
\def\cV{{\mathcal V}}

\def\cX{{\mathcal X}}
\def\cY{{\mathcal Y}}

\def\ud{{\underline{\delta}}}

\def\rD{{\rm D}}
\def\rT{{\rm T}}
\def\rd{{\rm d}}

\def\dt{{\rm d}t}

\newcommand{\pbar}{\bar{\partial}_J}

\def\bM{{\overline{\mathcal M}}}
\newcommand{\CM}{\overline{\mathcal{M}}}
\newcommand{\CMp}{\overline{\mathcal{M}}\,\!'}

\newcommand{\ti}{\tilde}
\newcommand{\Ti}{\widetilde}
\newcommand\ul{\underline}

\def\la{\langle\,}
\def\ra{\,\rangle}
\def\st{\: \big| \:}

\newcommand\quotient[2]{
        \mathchoice
            {% \displaystyle
                \text{\raise1ex\hbox{$#1$}\Big/\lower1ex\hbox{$#2$}}%
            }
            {% \textstyle
                #1\,/\,#2
            }
            {% \scriptstyle
                #1\,/\,#2
            }
            {% \scriptscriptstyle
                #1\,/\,#2
            }
    }
    
\newcommand\quo[2]{
                \text{\raise1ex\hbox{$#1\!$}\big/\lower1ex\hbox{$\!#2$}}
  }

\newcommand\qu[2]{
                \text{\raise.8ex\hbox{$\scriptstyle#1$}/\lower.8ex\hbox{$\scriptstyle#2$}}
  }

\DeclareMathOperator{\Hom}{Hom}

\DeclareMathOperator{\pr}{pr}

\DeclareMathOperator{\supp}{supp}

\newcommand{\id}{\operatorname{id}}

\newcommand{\Aut}{\operatorname{Aut}}

\newcommand{\Crit}{\operatorname{Crit}}

\newcommand{\ind}{\operatorname{ind}}

\newcommand{\im}{\operatorname{im}}

\newcounter{qcounter}
\newenvironment{enumlist}
   {
      \begin{list}
         {\bf\arabic{qcounter})}
         {
         \usecounter{qcounter}
                     \setlength{\itemsep}{.5ex}
            \setlength{\leftmargin}{1.5em}
         }
   }
   {
      \end{list}
   }

\newenvironment{itemlist}
   { \begin{list} {$\bullet$}
         {  \setlength{\itemsep}{.5ex}
            \setlength{\leftmargin}{1.5em} } }
   { \end{list} }

\newtheorem{theorem}{Theorem}[subsection]

\newtheorem{example}[theorem]{Example}

\newtheorem{lemma}[theorem]{Lemma}
\newtheorem{proposition}[theorem]{Proposition}

\newtheorem{definition}[theorem]{Definition}
\newtheorem{remark}[theorem]{Remark}

\begin{document}
\title{Polyfolds: A First and Second Look}
\author{Oliver Fabert}
\author{Joel W. Fish}
\author{Roman Golovko}
\author{Katrin Wehrheim}
\begin{abstract}
Polyfold theory was developed by Hofer-Wysocki-Zehnder by finding
  commonalities in the analytic framework for a variety of geometric
  elliptic PDEs, in particular moduli spaces of pseudoholomorphic curves.
It aims to systematically address the common difficulties of
  ``compactification'' and ``transversality'' with a new notion of
  smoothness on Banach spaces, new local models for differential geometry,
  and a nonlinear Fredholm theory in the new context.
We shine meta-mathematical light on the bigger picture and core ideas of
  this theory. In addition, we compiled and condensed the core definitions
  and theorems of polyfold theory into a streamlined exposition, and
  outline their application at the example of Morse theory.

\end{abstract} 

\address{Department of Mathematics, Vrije Universiteit Amsterdam,
1081 HV Amsterdam, Netherlands}
\email{oliver.faber@gmail.com}
\urladdr{www.few.vu.nl/$\sim$fabert}
\vspace{2in}
\address{Department of Mathematics, University of Massachusetts Boston, MA 02125, USA}
\email{joel.fish@umb.edu}
\urladdr{www.joelfish.com}
\address{Departement de Mathematiques, Universite Paris-Sud, Batiment 425
F-91405 Orsay, France}
\email{roman.golovko@math.u-psud.fr}
\address{Department of Mathematics, UC Berkeley, CA 94720, USA} \email{katrin@math.berkeley.edu}
\urladdr{www.math.berkeley.edu/$\sim$katrin/}
\thanks{Research partially supported by NSF grants DMS-0802927 %This is Joel's support
and DMS-0844188, % Katrin's funding
and the Ellentuck Fund % Joel's funding at IAS
}
\subjclass[2000]{Primary 32Q65; Secondary 53D99}

\keywords{non-linear functional analysis, fredholm theory, transversality, polyfolds}

\maketitle 

\vspace{-5mm}

\tableofcontents

\markboth{Fabert, Fish, Golovko, Wehrheim}{Polyfolds: A First and
  Second Look}

%%%%%%%%%%%%%%%%%%%%%%%%%%%%%%%%%%%%%%%%%%%%%%%%%%%%%%%%%%%%%%%%%%%%%%%%%%%%%%%%
%%%%%%%%%%                            SECTION                          %%%%%%%%%
%%%%%%%%%%%%%%%%%%%%%%%%%%%%%%%%%%%%%%%%%%%%%%%%%%%%%%%%%%%%%%%%%%%%%%%%%%%%%%%%
%
\section{Introduction} 
One of the main tools in symplectic topology is the study of moduli
  spaces of pseudo-holomorphic curves.
Roughly speaking, one thinks of such a moduli space $\mathcal{M}$
  as a set of equivalence classes of smooth maps which satisfy the
  Cauchy-Riemann equation, $\bar{\partial}_Ju:=\frac 12 (du + J\circ
  du\circ j)=0$, where two maps \(u\) and \(v\) are equivalent provided
  there exists a holomorphic automorphism \(\phi\) of the domain such
  that \(u=v\circ \phi\).
Additionally, one may wish to consider one or more standard
  modifications, like considering an inhomogeneous Hamiltonian term,
  Lagrangian boundary conditions, point constraints, or punctures with
  specified asymptotics.
In most applications, one would like to associate to such a moduli
  space a ``compact regularization,'' denoted $\CM\,\!'$, that is a
  compact manifold/orbifold, possibly with boundary and corners, and
  that is unique up to the appropriate notion of cobordism.
Indeed, such a rich geometric structure, in which boundary strata are
  related to lower dimensional components of other moduli spaces, is
  precisely what gives rise to the rich algebraic structures appearing
  in applications such as Floer complexes \cite{floer} and Symplectic
  Field Theory \cite{EGH}.

The current constructions of such regularized moduli spaces $\CM\,\!'$
  all use essentially similar ingredients: The Cauchy-Riemann equation
  is cast as a Fredholm problem, a compactness theorem is proven in
  which the description of convergence to a ``broken'' or ``nodal''
  curve is provided, a gluing theorem is proven in which smooth curves
  are constructed from the broken or nodal curves, and the issue of
  transversality is resolved in order to obtain a smooth structure.
Due to the length and technical complications that arise in such a
  program, very few moduli space constructions in the literature are
  technically complete.
In fact, such completeness is often undesirable since it would lead to
  countless repetitions of ``standard techniques'' in slightly different
  settings, which would hide the main ideas.
On the other hand, subtle problems are easily overlooked when proofs
  merely refer to techniques of other papers which are not complete either.

The polyfold theory, developed by H.~Hofer, K.~Wysocki, and E.~Zehnder,
  aims to provide an analytic framework within which technically complete
  proofs can be given in a compact and instructive way.
Additionally, the theory comes with a collection of ``building block''
  results which allow the theory to rapidly extend from a few model cases
  to a large variety of different setups.
The most important pair of features, perhaps, is the abstract
  perturbation scheme and implicit function theorem which together resolve
  the transversality problem at a completely abstract functional-analytic
  level: Any compact moduli space that admits a description as the zero
  set of a ``Fredholm section of a polyfold bundle'' can be perturbed
  within this ambient space as if it was the zero set of a smooth section
  in a finite dimensional bundle.
Such a perturbation scheme then yields a natural representation of
  the moduli space as a cobordism class of smooth, finite dimensional,
  closed manifolds -- in the case of trivial isotropy and empty boundary.
In the case of nontrivial isotropy, which is analogous to perturbing a
  section of an orbi-bundle, a multi-valued perturbation scheme yields
  a cobordism class of weighted branched orbifolds.\footnote
  {
  Weighted branched orbifolds are a mild generalization of closed
    manifolds in the sense that they still have natural fundamental
    classes, just with rational coefficients.
  }
In cases involving boundary (and corners), polyfold theory offers a
  relative perturbation scheme that allows one to restrict the support of
  perturbations to a neighbourhood of the non-transverse part of the
  moduli space (in practice often a complement of the boundary).
This essentially reduces the challenge of constructing ``coherent
  perturbations''\footnote
  {
  Coherence of perturbations with gluing operations is a core requirement
    for all Floer-type theories arising from moduli spaces (except for
    Gromov-Witten theory).
  The reason is that these theories not only construct algebraic
    structures (e.g.\ a Floer differential $\partial$) from
    moduli spaces, but also deduce their algebraic properties (e.g.\
    $\partial\circ\partial =0$) by identifying the boundary of each moduli
    space with fiber products of other moduli spaces.
  }
  to ensuring that the combinatorics of the gluing operations would
  allow for coherent perturbations if all involved moduli spaces were
  cut out by smooth sections of finite dimensional bundles.

Let us briefly sketch the two core analysis issues for achieving such a
  powerful abstract perturbation scheme, and how polyfold theory arises
  naturally as a means to resolve these issues directly rather than
  circumvent them.
Firstly, the reparametrization action $(\phi,u)\mapsto u\circ\phi$
  by nondiscrete\footnote{ Standard examples are the action of \({\rm
  PSL}(2,\mathbb{C})\) on the space of maps \(u:\mathbb{C}P^1\to X\)
  via reparametrization, or the action of \(\mathbb{R}\) on the space
  of maps \(\gamma:\mathbb{R}\to X\) via reparametrization. } families
  of automorphisms $\phi$ on an infinite dimensional space of maps $u$
  is not classically differentiable in any usual Banach topology.
(See e.g.\ Example~\ref{ex:MorseDiff} and \cite{mcduff-wehrheim} for
  discussions of this phenomenon.)
Hence a moduli space of pseudoholomorphic curves is classically
  described by first giving the space of pseudoholomorphic maps a smooth
  structure by finding an equivariant (!) transverse perturbation, and
  then quotienting this finite dimensional space by the -- then smoothly
  acting -- reparametrizations.
Such perturbations exist in many cases, e.g.\ by variation of the almost
  complex structure $J$, but in general transversality and equivariance
  are contradictory requirements.
These requirments can be achieved simultaneously for pseudoholomorphic
  maps only under significant geometric control of the maps -- usually
  some type of injectivity.

The novel approach of polyfold theory to this issue is to replace
  the classical notion of differentiability with a new notion of scale
  differentiability.
This allows one to give a scale smooth structure to the infinite
  dimensional space of reparametrization-equivalence-classes of maps,
  and it also allows one to express the Cauchy-Riemann operator as a
  section over this space in such a way that the zero set of this section
  is precisely the moduli space.
Perturbations of this section then only need to be scale differentiable
  rather than equivariant.
On the other hand, this yields a new notion of smoothness that is
  sufficiently strong for zero sets of transverse scale smooth sections
  to inherit a smooth structure in the classical sense.

Secondly, almost all moduli spaces of pseudoholomorphic curves with
  regular domains\footnote
    {
    Throughout, we will call the domain of a pseudoholomorphic map or
      curve ``regular'' if it is a smooth, connected Riemann surface.
    Here ``curve'' stands for ``map modulo reparametrization of the
      domain'', see Remark~\ref{rmk:term}.
    We will refer to the corresponding curves as ``non-nodal'' /
      ``unbroken'' or ``smooth'' since regularity for maps or curves
      usually refers to surjectivity of the linearized Cauchy-Riemann
      operator.  
    }
  require a compactification by ``nodal'' or ``broken'' curves, which
  are described as pseudoholomorphic maps from singular domains\footnote
    {
    Throughout, we will call the domain of a pseudoholomorphic map or curve
    ``singular'' if it is not regular.
    For example, the domain of a map representing a ``nodal curve''
      consists of several connected Riemann surfaces together with marked
      points which indicate the nodes at which the pseudoholomorphic
      map is required to satisfy incidence conditions between pairs of
      marked points.
    For a ``broken curve'' the underlying domain is of the same kind,
      but the marked points are considered as punctures at which the map
      generally doesn't extend continuously but has a certain asymptotic
      behaviour, with limits that are required to satisfy incidence
      conditions between pairs of punctures.
    }.  
This precludes any description of the compactified moduli space as a
  subset of a single Banach manifold of maps.
Classically, this compactification is constructed by gluing theorems
  after transversality is achieved.
This raises nontrivial difficulties for each new moduli space problem
  -- in particular, when families of curves must be glued to form the
  boundary of moduli spaces of dimension two or more.
Here the novel notion of an sc-retract or splicing core (which formalize
  the pregluing construction) allows polyfold theory to build ambient
  spaces of (equivalence classes of) maps in which maps with singular
  domains have neighborhoods of maps with both singular and regular
  domains.
In fact, nodal curves in Gromov-Witten theory become smooth interior
  points of an ambient space that consists of nodal and non-nodal
  equivalence classes of maps that may or may not satisfy the PDE.
Then part of the gluing analysis is formalized as a Fredholm condition
  on the Cauchy-Riemann operator at nodal curves, and other parts are
  replaced by an abstract implicit function theorem for Fredholm sections
  over sc-retracts.

Together, these two ideas generate a fundamentally new version of
  nonlinear Fredholm theory, which is stronger than the classical theory
  in that it includes an abstract perturbation scheme in addition to an
  implicit function theorem.
Furthermore, it is more flexible in that it is expected to admit a
  description of any compactified moduli space $\CM$ of pseudoholomorphic
  curves as the zero set of a single ``scale smooth Fredholm section''
  $\ti{\sigma}:\Ti{\cB}\to\Ti{\cE}$ in a ``polyfold bundle''
  $\Ti{\cE}\to\Ti{\cB}$.
Once such a description is given, the abstract transversality package
  is a direct generalization of finite dimensional differential geometry.
More specifically, after verifying that $\tilde{\sigma}^{-1}(0)$ is
  compact, one knows that there exist arbitrarily small perturbations
  $p:\Ti\cB\to\Ti\cE$ such that $\ti\s+p$ is transverse to the zero
  section; the zero set of such a perturbed section $\CM\,\!':=(\ti
  s+p)^{-1}(0)$ is a compact, finite dimensional manifold (or orbifold,
  and possibly with boundary and corners); and the zero sets for any two
  such perturbations are cobordant in the appropriate sense.

Hence one benefit of the polyfold approach is that the perturbation
  theorem sketched above does not depend on specific properties of the
  moduli problem under study, but rather it holds abstractly in the
  category of polyfolds.
Consequently, the resolution of the difficult transversality problem
  for moduli spaces is reduced to the simpler task of showing that the
  moduli problem fits into the polyfold framework.
On the other hand, a drawback of the polyfold approach is that one
  must become at least minimally familiar with the language, the new
  differentiable structures, and the basic results of the theory,
  which are dispersed across many articles and hundreds (if not yet
  thousands) of pages written by H.~Hofer, K.~Wysocki and E.~Zehnder
  (\cite{Hofer0}, \cite{Hofer}, 
  \cite{HWZ_lectures},
  \cite{HoferWysockiZehnder1}, \cite{HoferWysockiZehnder2},
  \cite{HoferWysockiZehnder3}, \cite{HoferWysockiZehnder4},
  \cite{HWZscSmooth}, \cite{HWZ_Integration}, \cite{HWZ_DM},
  \cite{HWZI_applications}, \cite{HWZII_applications},  \cite{hwzbook},
  \cite{HWZ_DetBundles}).

As such, the goal of this paper is to distill the theory down to a
  few essential elements, and to present these core ideas and suggested
  applications to any reader who wonders how a moduli space is constructed
  from a differential equation and who knows what a Banach space is.
Furthermore, this should empower such a  reader to evaluate the benefits
  and applicability of polyfold theory, and provide the basics for dealing
  with this theory.
More specifically, those who do not usually touch a differential operator
  themselves should be enabled to make sense of moduli space constructions
  written in polyfold language.
Readers who are considering applying polyfold theory in their own
  work should obtain a road map, which should allow them to efficiently
  compile details from the large body of work of Hofer--Wysocki--Zehnder --
  henceforth abbreviated by HWZ --  with little additional technical work.
For that purpose this article is divided into the following two parts,
  which are mostly independent of each other and may be of interest to
  different readers.

\smallskip
\noindent {\bf I) Meta-mathematics:} 
This section provides some polyfold philosophy.  
We loosely describe the key elements of the theory, and we compare
  the polyfold approach to other currently used approaches (namely
  ``geometric'' and ``virtual'') by providing a road map for each.

\smallskip
\noindent{\bf II) Mathematics:} 
This section provides the core definitions which are presented in a
  streamlined fashion so that we may state the abstract transversality
  result as quickly as possible.
For several key ideas we present companion examples to illustrate either
  the concept or its necessity in the theory.

\smallskip
For the sake of brevity, we restrict our presentation to the theory of
  M-polyfolds, which deals with the case of the automorphism group acting
  freely (i.e.\ the case of trivial isotropy) 
  and yields solution spaces which have the structure of a manifold.
The most essential new concepts of polyfold theory are already contained
  in this part and are best presented without the algebraic distraction
  of additional discrete group actions (i.e.\ nontrivial isotropy).
In cases of nontrivial discrete stablizers, the ambient space can then
  be described as a polyfold -- a groupoid whose spaces of objects
  and morphisms are M-polyfolds -- and transverse multisections of a
  polyfold bundle give the moduli spaces the structure of a branched
  weighted orbifold.
The latter ideas for dealing with discrete symmetries have already been
  well established in the literature.
The crucial new input is the transversality package for M-polyfolds,
  which can be directly applied to polyfolds; 
  see Remark~\ref{rmk:isotropy}.

The approaches and technical ingredients for moduli space problems
  discussed here build on the shoulders of many researchers, in particular
  Donaldson, Floer, Fukaya, Gromov, Hofer, Joyce, Li, Liu, McDuff, Oh,
  Ohta, Ono, Ruan, Salamon, Siebert, Taubes, Tian, Wysocki, Zehnder.
In order to neither offend nor misrepresent, we have decided to not
  attempt to provide systematic citations except for elements of polyfold
  theory.

\medskip
\noindent
{\bf Acknowledgements:}
These notes grew out of a working group organized by the first three
  authors at MSRI in fall 2009.
We would like to thank this working group as well as Helmut Hofer for
  their great help and stimulating discussions.
Further useful comments were provided by Sonja Hohloch, Urs Fuchs,
  Chris Wendl, participants of the 2014 ``ECH \& Polyfolds'' seminar at
  UC Berkeley, and the referee. 

\part{Traversing Transversality Troubles}

In this meta-mathematical part, we will share our insights on the
  approaches to the regularization of moduli spaces that are currently
  present in the literature.
The main goal here is to clarify the origin and novelty of the polyfold
  approach and show how a different ordering of basic ingredients (implicit
  function theorem, quotient, gluing) results in a more organized and
  automated theory of transversality.
While we will not explicitly discuss any concrete constructions, we
  encourage the readers to interpret all general discussions in their
  favorite specific setting and then make appropriate adjustments to our
  vague formulations.
For instance, the discussion that follows can be adapted to Gromov-Witten
  theory, various versions of Floer homology, various versions of contact
  homology, Symplectic Field Theory, and other moduli space problems
  as well.
In order to maximize accessibility of the discussion that follows,
  we will use Morse theory as a common ground.
Of course, polyfolds are not needed to resolve transversality issues that
  arise in Morse theory, however polyfolds do indeed apply to Morse theory,
  and the simplicity of such an analytic setup will help to illuminate
  the core ideas arising in the polyfold theory.

\begin{example}[\bf Compactified Morse moduli space] \label{ex:Morse} \rm
The Morse moduli space $\cM$ consists of trajectories between any pair
  of critical points of the gradient vector field of a Morse function
  \mbox{$f:X\to\R$} on a Riemannian manifold $(X, g)$.
That is, $\cM$ is made up of gradient flow lines, i.e.\
  maps $\gamma:\R\to X$ satisfying the gradient flow equation
  $\frac\rd{\dt}\gamma - \nabla f=0$ modulo the automorphism group $\R$
  which acts by shifts $(s,\gamma)\mapsto \gamma(s+\cdot)$.
The compactification $\CM$ of this moduli space consists of broken
  trajectories, which are tuples $[\gamma_1],\ldots,[\gamma_k]\in\cM$
  of any length $k\geq 1$ with matching limits
  $\lim_{t\to-\infty}\gamma_{i-1}(t) = \lim_{t\to \infty}\gamma_{i}(t)$.
\end{example}

\begin{remark}[Terminology]  \label{rmk:term}
We will use the following terminology: A {\bf trajectory} \([\gamma]\)
  is an equivalence class of maps  \(\gamma:\mathbb{R}\to X\), where
  \([\gamma_1] = [\gamma_2]\) iff \(\gamma_1(\cdot) = \gamma_2(s_0 +
  \cdot)\) for some \(s_0\in \mathbb{R}\); a {\bf gradient trajectory} or
  a {\bf flow line} is a trajectory for which each representative solves
  the gradient equation \(\frac\rd{\dt}\gamma = \nabla f (\gamma) \).
Similarly, a {\bf curve} is an equivalence class of triples $(\Sigma, j, u)$, where \(u:(\Sigma,
  j)\to X\), and \([u_1] = [u_2]\) iff \(u_2\circ \phi = u_1\) for
  a biholomorphism \(\phi:(\Sigma_1,j_1) \to (\Sigma_2,j_2)\); a {\bf
  pseudoholomorphic curve} with respect to some almost complex structure
  $J$ on $X$ is then a curve such that each representative solves the
  Cauchy-Riemann equation \(\bar{\partial}_J u = 0\).
\end{remark}

\begin{remark}[Conventions]
It will be convenient to distinguish between \(\mathbb{N}=\{1, 2, 3,
  \ldots\}\) and \(\mathbb{N}_0:=\{0\}\cup \mathbb{N}\).
More importantly, $\cC^k(\Omega,\R^N)$ for a Riemannian manifold
  $\Omega$ will always denote the Banach space of $k$-fold continuously
  differentiable maps $u:\Omega\to \R^N$.
In particular, if $\Omega$ is noncompact, then we explicitly require
  any $u\in\cC^k(\Omega,\R^N)$ to have bounded derivatives up to order
  $k$, and we equip this space with the $\cC^k$-norm rather than the
  $\cC^k_{\rm loc}$-topology.
Similarly, for a Riemannian manifold $X$, we denote by $\cC^k(\Omega,X)$
  the Banach manifold of maps $u:\Omega\to X$ whose derivatives are
  bounded up to order $k$, and whose image is precompact.
It is equipped with the $\cC^k$-topology and modeled on the Banach
  space $\cC^k(\Omega,\R^N)$ for $N=\dim X$.
\end{remark}

\section{The essence of polyfolds}
In this section we discuss some of the foundational issues that arise
  in attempts to regularize moduli spaces, and we provide a broad picture
  of polyfold theory via comparison to a finite dimensional regularization
  theorem.
In Sections~\ref{ss:sc} and \ref{ss:ret} we then provide an overview of
  the two fundamentally new concepts of scale calculus and sc-retractions
  on which polyfold theory builds.

\subsection{Some broad strokes}\label{sec: broad strokes}
We begin by comparing the analytic framework of a typical moduli space
  problem to a familiar problem in finite dimensions.
In order to obtain an efficient transversality theory for a given
  moduli space $\cM$, we aim to build an ambient space \(\mathcal{B}\),
  a vector bundle over this space \(\mathcal{E}\to \mathcal{B}\), and
  a section \(\sigma:\mathcal{B}\to\mathcal{E}\) so that the zero set
  $\sigma^{-1}(0)\cong \cM$ represents the moduli space as a subset of
  the ambient space $\mathcal{B}$.

Given such a description, we intuitively expect an implicit function
  theorem to equip $\cM$ with a smooth structure whenever the section
  $\sigma$ is transverse to the zero section of $\mathcal{E}$; and we
  hope to achieve this transversality by some dense set of perturbations
  of $\sigma$, with the resulting regularized moduli space essentially
  independent of this choice.
In finite dimensions, this intuition is in fact valid and it can easily
  be made precise:

\begin{theorem}[\bf Finite dimensional regularization] \label{thm:model}
Let $E\to B$ be a smooth finite dimensional vector bundle, and let
  $s:B\to E$ be a smooth section such that $s^{-1}(0)\subset B$ is compact.
Then there exist arbitrarily small, compactly supported, smooth
  perturbation sections $p:B\to E$ such that $s+p$ is transverse to the
  zero section, and hence $(s+p)^{-1}(0)$ is a smooth manifold.
Moreover, the perturbed zero sets $(s+p')^{-1}(0)$ and $(s+p)^{-1}(0)$
  of any two such perturbations $p,p':B\to E$ are cobordant.
\end{theorem}

\begin{remark} \label{rmk:regularize} \rm
At this point we can explain our notions of {\bf regularization} and
  {\bf transversality}.
The latter is a fixed and rigorous mathematical notion, and in this
  case it is the assertion that at any solution $x\in (s+p)^{-1}(0)$
  the image of the differential $\rd_x(s+p)$ projects surjectively to
  the fiber $E_x$.
By the implicit function theorem, this equips $(s+p)^{-1}(0)$ with a
  smooth structure, and it is customary to refer to the existence of a
  class of such transverse perturbations $p$ as transversality.
However, transversality does not yet guarantee compactness of
  $(s+p)^{-1}(0)$ or its uniqueness up to cobordism.
It is this package -- the existence of a class of perturbations
  $p\in\cP$, whose compact smooth zero sets $\CM_p:=(s+p)^{-1}(0)$
  are unique up to cobordism -- which we call the regularization of the
  solution space $s^{-1}(0)$.
More precisely, this allows us to associates to a possibly rather
  singular space $s^{-1}(0)=\CM$ (in practice this is the moduli space)
  the more regular object of a cobordism class $[\CM]:=[\CM_p]$.
This regularization of $\CM$ is independent of the choice of perturbation
  $p\in\cP$ due to the existence of cobordisms $\CM_p \sim \CM_q$ for
  any other $q\in\cP$.
\end{remark}

The aim of this section can then be stated as the discussion of
  possible generalizations of Theorem~\ref{thm:model} that could provide
  an efficient regularization theory for moduli spaces.
Before doing this however, let us highlight two limitations of the
  finite dimensional regularization theorem.

\begin{itemlist}
\item
  Neither Theorem~\ref{thm:model}, nor any direct generalization of it
    provides equivariant transversality.
  That is, if the section $\sigma$ is equivariant under a group action,
    then one generally cannot require the transverse perturbation $p$
    to be equivariant as well.
  One notable exception is the case of a \emph{finite} group action,
    in which case one can generally find transverse equivariant
    multisections.
  For nondiscrete groups, equivariance and transversality are  -- except
    for rather special circumstances -- nearly contradictory requirements.
\item
  While transversality for perturbed sections can still be achieved if
    \(\sigma^{-1}(0)\) is non-compact, one cannot
    expect regularization.  
  More specifically, our notion of regularization demands not just
  transverality, but also \emph{uniqueness} of the cobordism class of
  the zero set of the perturbed moduli space, and such uniqueness is
  not obtained in general if the unperturbed solution set is not compact.
\end{itemlist}

For the application to moduli spaces, it is crucially important that
  the perturbed zero set be regularized in the above manner because
  the topological invariants arising from the moduli spaces are usually
  obtained by \emph{counting}\footnote
  { 
    More generally one seeks to pull back differential forms from a
      target manifold \(X\) and integrate them over the moduli space,
      which need not be well defined if the moduli space is not compact.
  } 
  elements in the perturbed solution space.  
For example, one counts gradient flow lines modulo translation to
  define the differential in Morse homology,  and a count of closed
  pseudoholomorphic curves (i.e. pseudoholomorphic maps modulo
  reparametrization) defines the Gromov-Witten invariants.
This also points to the significance of equivariance: A generalization of
  Theorem~\ref{thm:model} would have to provide equivariant transversality
  if it was to be applicable to the classical description of these moduli
  spaces %K by equivariant Fredholm 
  in terms of equivariant sections.
In order to demonstrate this in an example, we return to Morse theory
  as a common ground and recall its classical equivariant %K Fredholm 
  setup.

\begin{example}[\bf %K Fredholm setup and translation action for Morse theory
Equivariant setup for Morse theory
]\label{ex:MorseFredholm} \rm
Let $X$ be a closed smooth manifold of positive dimension, let $f:X\to
  \R$ be a Morse function, and let $g$ be a Riemannian metric on $X$.
%K
%Then a Banach manifold $\mathcal{B}$, Banach bundle $\mathcal{E}\to
%  \mathcal{B}$, and Fredholm  section $\sigma:\mathcal{B}\to\mathcal{E}$ are given by the following:
%\begin{align*}
%  &\cB=\{\gamma\in\cC^1(\R,X) \st \lim_{s\to\pm\infty}\gamma(s)\in{\rm
%    Crit}(f) 
%    \;\text{ and } \lim_{|s|\to \infty} |\gamma'(s)|=0\} , \\
%  &\cE={\textstyle \bigcup_{\gamma\in\cB}} \,\cE_\gamma, \qquad
%    \cE_\gamma=\{ \eta\in\cC^0(\R,\gamma^*TX)  \st \lim_{|s|\to \infty}
%    |\eta(s)|=0\},\\
%  &\s(\gamma)=\dot\gamma - \nabla f(\gamma).
%\end{align*}
The flow lines of the gradient vector field $\nabla f$ on $X$ are the solutions $\gamma:\R\to X$ of 
$\dot\gamma - \nabla f(\gamma) = 0$. 
Since these solutions are automatically smooth, there are many choices of bundles 
$\mathcal{E}\to\mathcal{B}$ so that $\s(\gamma):=\dot\gamma - \nabla f(\gamma)$ defines a 
section $\sigma:\mathcal{B}\to\mathcal{E}$ whose zeros are the gradient flow lines. 
The regularization approaches discussed in \S\ref{s:roadmaps} all require a Fredholm setup -- i.e.\ a choice of Banach manifold $\mathcal{B}$ and Banach bundle $\mathcal{E}\to\mathcal{B}$ so that the linearizations $D_\gamma\sigma:\rT_\gamma\cB \to \cE_\gamma$ at solutions $\sigma(\gamma)=0$ are Fredholm operators. 
This is generally achieved by working in suitable Sobolev spaces, 
but the expository purposes of this section are better served by considering the simplified setup\footnote{
Note that the section in this simplified setup is generally not Fredholm since e.g.\ for $X=\R/\Z$ and $f=0$ the image of the linearized section 
$\{\xi\in \cC^1(\R,\R) \,|\, \lim_{s\to\pm\infty}\xi(s) = 0\}  \to \cC^0(\R,\R) , \xi \mapsto \dot\xi $ 
does not contain any $\cC^0$-function $\eta:\R\to\R$ with divergent indefinite integral $\int_0^\infty \eta(s) \, \rd s$ or $\int_{-\infty}^0 \eta(s) \, \rd s$.
}
\begin{align*}
 \cB=\{\gamma\in\cC^1(\R,X) \st \lim_{s\to\pm\infty}\gamma(s)\in{\rm
    Crit}(f) \} , \qquad  \cE={\textstyle \bigcup_{\gamma\in\cB}} \,\cE_\gamma, \qquad
    \cE_\gamma= \cC^0(\R,\gamma^*TX) . 
\end{align*}
Observe that if $\gamma\in \mathcal{B}$ and $\sigma(\gamma)=0$, then
  for each $s\in \R$ we also have $\sigma(\tau(s,\gamma))=0$, where $\tau$
  is the {\bf translation action} (often also called {\bf shift map})
  \begin{equation}\label{tau}
    \tau: \R\times \mathcal{C}^1(\R, X)\to \mathcal{C}^1(\R, X) \qquad
    \text{given by}\qquad\tau(s,\gamma):=\gamma(s+\cdot).
  \end{equation}
Since the automorphism group ${\rm Aut}=\R$ is non-compact, we must
  conclude that $\sigma^{-1}(0)$ is non-compact, unless it only consists
  of fixed points of the action, i.e.\ constant maps.
The moduli space of unbroken Morse trajectories is then defined as
  the quotient $\cM:=\s^{-1}(0)/{\rm Aut}$ of the zero set by this
  reparametrization action.
\end{example}

Similar to the above example, most (not yet compactified) moduli
  spaces of pseudoholomorphic curves have a description as quotient
  $\cM:=\s^{-1}(0)/{\rm Aut}$ of an ${\rm Aut}$-equivariant %K Fredholm
  section $\sigma:\mathcal{B}\to\mathcal{E}$ over a Banach manifold
  $\cB$ of maps (and often additional parameters describing a variation
  of domain or equation), on which a Lie group ${\rm Aut}$ acts by
  reparametrizations.
We cannot expect any general regularization theory such as
  Theorem~\ref{thm:model} to apply to this type of setup for two reasons
  related to the limitations discussed above:
\begin{itemlist}
  \item
    We are ultimately interested in the space \(\cM = \sigma^{-1}(0)/{\rm
      Aut}\) of solutions modulo reparametrization, so in order to
      be able to quotient the perturbed zero set by ${\rm Aut}$, the
      perturbation $p$ in Theorem~\ref{thm:model} would have to be ${\rm
      Aut}$-equivariant.
  \item
    The automorphism group ${\rm Aut}$, such as ${\rm Aut}=\mathbb{R}$ in
      the above example, is usually non-compact, and the moduli space does
      not just consist of fixed points of ${\rm Aut}$, hence $\s^{-1}(0)$
      must be non-compact.
    Furthermore, even if ${\rm Aut}$ was compact, then in all nontrivial
      examples the appearance of nodal (or broken) curves (or trajectories)
      is an additional source of non-compactness.
  \end{itemlist}
However, in any general setup, even the finite dimensional theory
  provides neither equivariant transverse perturbations nor a
  regularization of non-compact zero sets.
As such, approaches to regularize moduli spaces split into several
  basic types:
\begin{itemlist}
  \item
    The {\bf geometric approach}, discussed further in
      Section~\ref{ss:geometric}, makes use of special geometric
      properties of a given moduli problem to find transverse equivariant
      perturbations of a section with noncompact zero set.
    However, this \emph{only} yields transversality; that is, one
      still must construct a compactification and prove uniqueness up
      to cobordism, and this additional work may require new ideas and
      substantial effort.
    The only major abstract theorem used in this approach is the
    classical Sard-Smale theorem (where regular points
      yield transversality) applied to manifolds of
      maps of a fixed domain, and hence it cannot regularize moduli spaces
      in which the topology of the domain changes abruptly.
   Consequently, such geometric approaches are not analagous to Theorem ~\ref{thm:model}, since the latter
      simultaneously yields transversality, \emph{compactness}, and
      \emph{uniqueness}.
  \item 
    Any abstract approach via some type of generalization of
      Theorem~\ref{thm:model} must work in a setting where the unperturbed
      solution space is compact and no further nondiscrete symmetry of
      the perturbation is required.
    We roughly classify such approaches by the dimensionality of the
      bundles involved:
  \begin{itemize}
    \item
      Several types of {\bf virtual approaches}, which we discuss further
	in Section~\ref{ss:virtual}, work with a highly generalized version
	of Theorem~\ref{thm:model} for finite dimensional bundles over
	groupoid-like structures or topological spaces with merely local
	smooth structures.
    \item
      The {\bf polyfold approach} works with a direct generalization
	of Theorem~\ref{thm:model} to infinite dimensional bundle-like
	linear structures over infinite dimensional manifold-like spaces
	with a global smooth structure.
  \end{itemize}
\end{itemlist}

Since the polyfold approach aims to be a unified perturbation theory
  for a broad class of moduli problems, it must develop a regularization
  theory that directly applies to sections of a bundle over the space
  \(\mathcal{B}/{\rm Aut}\), and, in so doing, it removes the requirement
  that perturbations must be equivariant (since \({\rm Aut}\) does not
  act on \(\mathcal{B}/{\rm Aut}\)).
This is then one step closer to a setting in which the unperturbed
  solution space $\s^{-1}(0)$ is compact, and hence a full regularization
  theory can be hoped for, however doing analysis directly on the space
  \(\mathcal{B}/{\rm Aut}\) raises a serious difficulty.
We take a moment to highlight the {\bf failure of differentiability of
  the action of reparametrization} in the example of Morse theory.

\begin{example}[\bf Differentiability of translation action]
  \label{ex:MorseDiff} \rm
In the notation of Example~\ref{ex:MorseFredholm}, the development of a
  regularization theory would require some type of smooth structure on the
  space \(\mathcal{B}/{\rm Aut}\) of \(\mathcal{C}^1\)-paths $\g:\R\to X$
  between two critical points, modulo the reparametrization action of
  ${\rm Aut}=\R$.
However, the translation action of $\R$ on $\cC^1(\R,X)$, given by
  \(\tau\) in equation \eqref{tau}, is nowhere differentiable with respect
  to the \(\mathcal{C}^1\)-norms.
At first, one might think that $\tau$ is differentiable at points
  $(s_0,\g_0)\in \R\times \cC^2(\R)$; for example, at $(0,\gamma_0)$
  the differential, were it to exist, would necessarily be given by
  \begin{align*}
    ``\rD_{(s_0,\gamma_0)}\tau``: \; \R \times \cC^1(\R,\gamma_0^*\rT X)
    &\;\longrightarrow\; \cC^1(\R,\gamma_0^*\rT X) \\
    (S, \Gamma) \quad &\;\longmapsto\; S \tfrac \rd \dt \gamma_0 + \Gamma .
    \end{align*}
Note here that the right hand side takes values in $\cC^1$ only if
  $\gamma_0$ is $\cC^2$, so that this linear operator is not even defined
  for $\g_0\in\cC^1(\R)\setminus \cC^2(\R)$.
Moreover, the definition of the directional derivative in a fixed
  direction $(S,\G)\in\R\times \cC^1(\R)$ requires a linear approximation
  estimate, which holds only if $\max_{s\in \R} \bigl| \dot\G(s+h) -
  \dot\G(s) \bigr| \to 0$ as $h\to 0$.
Consequently, directional derivatives only exist in directions $\G$
  whose derivative is uniformly continuous, e.g.\ $\G\in\cC^2(\R)$.
Similarly, the linear estimate required for differentiability,
  $\max_{\|\G\|_{\cC^1}=1} \bigl\| \dot\G(\cdot+h) - \dot\G(\cdot)
  \bigr\|_\infty \to 0$ as $h\to 0$ fails at any $(s_0,\g_0)$, so that
  the above linear operator only provides directional derivatives in
  certain directions, and can never be viewed as differential of $\tau$.
Hence the best that can be said about differentiability of $\tau$ is that
  it is continuously differentiable as map $\R\times \cC^2 \to \cC^1$,
  and generally $k$-fold continuously differentiable as map $\R\times
  \cC^{k+\ell} \to \cC^\ell$.
For more details see Section~\ref{ss:sc}.

Another idea might be to restrict $\tau$ to the space of smooth paths
  and use a different Banach topology.
Note however that the restricted shift map is still not continuously
  differentiable in any standard Banach norm, since, for example, the
  potential differential
  \begin{align*}
    \R \times \cC^\infty(\R,X) &\;\longrightarrow\; \Hom \bigl( \R \times
    \cC^\infty(\R,\gamma_0^*\rT X) ,  \cC^\infty(\R,\gamma_0^*\rT X)
    \bigr) \\
    (s_0, \g_0) \quad &\;\longmapsto\; \quad ``\rD_{(s_0,\gamma_0)}\tau``
    \end{align*}
  is not continuous in the operator topology with respect to any
  fixed H\"older or Sobolev norms on the spaces $\cC^\infty(\R,X)$
  and $\cC^\infty(\R,\gamma_0^*\rT X)$.
In fact, this would in particular require continuity of the map
  $\g_0\mapsto \tfrac \rd \dt \gamma_0$, which -- with the Arzel\`a--Ascoli
  theorem in mind -- is plausible only on finite dimensional subspaces
  of $\cC^\infty(\R,X)$.
\end{example}

Other moduli problems share this same difficulty: The reparametrization
  action of a smooth family of automorphisms on a H\"older or Sobolev
  space of maps is not smooth in a classical sense.
The general consequence of this failure is that one cannot appeal to
  an abstract slice theorem to obtain a Banach manifold structure on
  \(\mathcal{B}/{\rm Aut}\).

\begin{remark}[\bf Local slices for maps modulo reparametrization]
  \label{rmk:slice} \rm
One may argue that, despite its differentiability failure, the
  translation action in Example~\ref{ex:MorseFredholm} of ${\rm Aut}=\R$
  on $\cB\subset\cC^1(\R,X)$ nevertheless has local slices: for any
  hypersurface $H\subset X$, let $\cU_H\subset\cB$ be the open set of
  maps $\g\in\cB$ which intersect \(H\) both transversely and exactly once.
Then $\cB_H=\{\gamma\in\cU_H \st \gamma(0)\in H\}$ is a Banach manifold
  homeomorphic to $\cU_H/{\rm Aut}$.
This yields Banach manifold charts for $\cB/{\rm Aut}$ in the Morse
  theory example\footnote
  {
    Strictly speaking, one has to restrict to a neighborhood of the
      Morse trajectories to ensure unique intersection points, or one can
      use a more subtle slicing for the space of all nonconstant maps.
    Moreover, Banach charts in the strict sense are obtained by
      composition with charts for \(\mathcal{B}_H\).
    See Example~\ref{ex:MorseTrajectorySpaces} for details.
  } 
  and similarly for all other reparametrization actions encountered in
  moduli spaces of holomorphic curves.
However, the transition maps between these charts are generally only
  continuous.
Indeed, for any other hypersurface $H'\subset X$, the transition map
  $\cB_H \cap \cU_{H'} \to \cB_{H'}$ is of the form $\gamma \mapsto
  \tau(s_\gamma, \gamma)$, where $s_\gamma\in\R$ is determined by
  $\gamma(s_\gamma)\in H'$.
Example~\ref{ex:MorseDiff} shows that maps of this type are not
  continuously differentiable unless $s_\gamma$ is constant.

For Morse theory, one can avoid transition maps by reducing $\cB$
  to a small neighborhood of the gradient flow lines.
Then a regular level set of the Morse function can serve as global
  hypersurface, since any map $\cC^1$-close to a gradient flow line will
  have a unique, transverse intersection with it.
In general, however, such global hypersurfaces are rare, and new
  methods would be needed to show that the resulting algebraic invariant
  is independent of their choice.
\end{remark}

We conclude from the preceeding discussion that \(\mathcal{B}/{\rm
  Aut}\) usually has geometrically constructed local slices, but the
  differentiability failure of the reparametrization action of ${\rm Aut}$
  obstructs the construction of a global smooth structure.
The manner in which polyfold theory resolves this difficulty constitutes
  one of the fundamentally new concepts of the theory: A {\bf scale
  calculus} of scale differentiable maps between scale Banach spaces;
  we introduce this notion in more detail in Section~\ref{ss:sc}.
It has several crucial properties:
  \begin{enumerate}
    \item In finite dimensions the scale calculus agrees with the
      classical calculus.
    \item The chain rule holds.
    \item It provides a framework in which reparametrization actions
      on infinite dimensional function spaces, such as the translation
      action \eqref{tau}, are scale smooth.
    \end{enumerate}
At this point polyfold theory gives \(\mathcal{B}/{\rm Aut}\) the
  structure of a scale manifold.
This is essentially achieved in two steps, the first of which is to
  enrich the smooth structure on the local slices $\cB_H$ to a scale
  structure.
Roughly speaking, this scale structure is a sequence of Banach spaces
  (e.g.\ Sobolev or H\"older spaces of increasing regularity) that are
  compactly and densely embedded into nested subspaces of $\cB_H$.
The second step is then to modify the notion of smoothness for
  the transition maps between the local slices by weakening it to
  scale smoothness, which requires only slightly more than $k$-fold
  differentiability between the Banach topologies in the scale sequence
  of distance $k$.
Nevertheless, the resulting scale calculus for scale manifolds is
  rich enough to establish a regularization theorem along the lines of
  Theorem~\ref{thm:model} for suitably defined scale smooth Fredholm
  sections with compact zero set.

We note however, that this scale regularization still does not even
  apply to our Morse theory example.
Indeed, the trouble is that the space of Morse trajectories is
  non-compact due to trajectory breaking.\footnote
  {
    For example, a sequence of trajectories between critical points
      of Morse indices 0 and 2 may converge, in the Gromov-Hausdorff
      topology on the images, to a broken trajectory comprised of one
      trajectory from the index 0 to an index 1 critical point, and
      another trajectory from this index 1 to the index 2 critical point.
  }  
Similarly, most pseudoholomorphic curve moduli spaces are compactified
  by adding nodal or broken curves.
In either case, the ambient space $\cB/{\rm Aut}$ has to be enlarged by
  fiber products of similar spaces in order to obtain an ambient space
  $\Ti\cB$ on which a generalized Cauchy-Riemann operator can provide a
  section $\Ti\s$ whose zero set $\Ti\s^{-1}(0)=\CM$ is the compactified
  moduli space. 
The topology on these enlarged ambient spaces is given by the images
  of open sets under a pregluing map, which roughly has the form
  $$
    \oplus \,:\; (R_0,\infty] \times \cB \times \cB \;\longrightarrow\;
    \Ti\cB \,.
  $$
In the Morse theory example this map joins the two domains $\R\sqcup\R$
  into a single domain $\R$, and it interpolates between shifts of the
  two maps that are determined by the gluing parameter in $(R_0,\infty]$;
  here a gluing parameter equal to $\infty$ corresponds to the broken
  trajectories in $\Ti\cB$.
At this point, the natural expectation is to also use this pregluing map
  (after fixing local slices $\cB_H\subset\cB$ of the $\rm Aut$-action)
  as a chart map for the ambient space $\Ti\cB$ near a broken trajectory.
However, such pregluing maps are never injective.
In fact, their kernel varies with the gluing parameter, and only the
  broken trajectories are parametrized uniquely.
Polyfold theory resolves this issue by the second fundamentally
  new concept of the theory: a differential geometry based on {\bf
  charts from retraction images}, which we introduce in more detail in
  Section~\ref{ss:ret}.
Roughly speaking, this allows one to view the pregluing map as a chart
  map for an M-polyfold by enriching it with a scale smooth retraction
  $\rho$ on its domain so that the pregluing map $\oplus|_{\im\rho}$
  restricted to the retraction image\footnote
  {
    Here the fact that this image of $\rho$ is a topological retract of
      the domain of $\rho$ has no significance; however the retraction
      property $\rho\circ\rho=\rho$ is crucial for the development of
      scale calculus on these images.
  }
  is a homeomorphism to an open subset of $\Ti\cB$. 
Diagrammatically we have
  \[
    \xymatrix{ (R_0,\infty] \times \cB_H \times
    \cB_H \; \ar[r]^{\qquad\qquad{\textstyle\oplus}}
    \ar@{->>}[d]_{\textstyle\rho}  &\; \Ti\cB   \\ \im\rho \;
    \ar@{^{(}->}[ur]_{\textstyle\oplus|_{\im\rho}} &  }
  \]
  where
  \begin{itemize}
    \item \((R_0, \infty]\) is the space in which the gluing parameter
      is allowed to vary;
    \item \(\mathcal{B}_H\) is a local model for the unbroken
      trajectories; i.e. \(\mathcal{B}/{\rm Aut}\);
    \item \(\widetilde{\mathcal{B}}\) is the space of broken and
      unbroken trajectories;
    \item \(\oplus\) is the pregluing map;
    \item \(\rho\) is the sc-smooth retraction mapping to and from
      \((R_0, \infty]\times\mathcal{B}_H\times\mathcal{B}_H\);
    \item \({\rm im} \rho\) is the image of \(\rho\) which is contained
      in \((R_0,\infty]\times\mathcal{B}_H\times\mathcal{B}_H\),
      and it serves as local model for an M-polyfold;
    \item \(\oplus\big|_{{\rm im }\rho}\) is the pregluing map restricted
      to the image of the retraction, and it serves as chart map; in
      other words it is a homeomorphism from the local M-polyfold model
      to an open subset of $\Ti\cB$.
    \end{itemize}

Note that this is a drastically weaker notion of chart than that of a
  Banach manifold chart.
The strength of the M-polyfold notion is in the requirements of
  transition maps, which involve the ambient space of the retraction and
  not just its image.
For example, the compatibility requirement for two charts, as above,
  which arise from different local slices (i.e. $\cB_H$ and $\cB_{H'}$
  of $\cB/\rm Aut$) is that the induced map $\iota_{\rho'} \circ
  (\oplus|_{\im\rho'})^{-1}\circ \oplus|_{\im\rho} \circ\rho$ (shown
  in the following diagram) is scale smooth between open subsets of the
  ambient scale manifolds.
\[
  \xymatrix{ (R_0,\infty] \times \cB_H \times  \cB_H
  \ar@{->>}[d]_{\textstyle \rho}  &  \Ti\cB  & \; (R_0,\infty] \times
  \cB_{H'} \times \cB_{H'} \ar@{->>}[d]_{\textstyle \rho'} \\ \im\rho
  \ar@{^{(}->}[ur]_{\textstyle\oplus|_{\im\rho}} &  & \; \im \rho'
  \ar@{^{(}->}[ul]^{\textstyle\oplus|_{\im\rho'}} \ar@/_1pc/[u]_{\textstyle
  \iota_{\rho'}} }
\]
This provides the notion of an M-polyfold atlas for a topological space
  such as $\Ti\cB$.
Given the notions of scale smoothness and M-polyfolds, HWZ then follow
  a relatively straightfoward path to defining compatible notions of
  bundles and Fredholm sections, and they then establish the following
  M-polyfold regularization theorem, which is a direct generalization of
  the finite dimensional regularization Theorem~\ref{thm:model}.

\begin{theorem}[\bf M-polyfold regularization]
  \label{thm:PolyfoldRegularization}
Let $\Ti{\cE}\to\Ti{\cB}$ be an M-polyfold bundle, and let
  $\ti\s:\Ti{\cB}\to\Ti{\cE}$ be a scale smooth Fredholm section such
  that $\ti\s^{-1}(0)\subset\Ti{\cB}$ is compact.
Then there exists a class of perturbation sections
  $p:\Ti{\cB}\to\Ti{\cE}$ supported near $\ti\s^{-1}(0)$ such that
  $\ti\s+p$ is transverse to the zero section and $(\ti\s+p)^{-1}(0)$
  carries the structure of a smooth compact manifold.
Moreover, for any other such perturbation $p':\Ti{\cB}\to\Ti{\cE}$
  there exists a smooth cobordism between $(\ti\s+p')^{-1}(0)$ and
  $(\ti\s+p)^{-1}(0)$.
\end{theorem}

\begin{remark}[\bf Regularization for moduli spaces with nontrivial
  isotropy] \label{rmk:isotropy} \rm
Beyond Morse theory, almost all moduli spaces that one may want to
  apply an abstract regularization scheme to -- in particular
  those consisting of pseudholomorphic curves in general
  symplectic manifolds -- require a further generalization of
  Theorem~\ref{thm:PolyfoldRegularization} to sections of a polyfold
  bundle.
This is because the analogue of the action in
  Example~\ref{ex:MorseFredholm} (in particular when functions on spheres
  are reparametrized) may have nontrivial discrete stablizers (also called
  isotropy groups).
Then local slices as in Remark~\ref{rmk:slice} still exist, but have
  to be viewed modulo an action of the local isotropy group.
This generalization is achieved by the same principles as the
  generalization of Theorem \ref{thm:model} to sections of orbi-bundles.
In particular, the notion of a polyfold is obtained from the notion of
  an M-polyfold just like the notion of an orbifold is obtained from the
  notion of a manifold.

More precisely, an atlas of a manifold can be described as a groupoid
  (a category with invertible morphisms) whose space of objects is given
  by the disjoint union of the charts, and with morphisms induced by the
  transition maps.
Some further properties are required in order for the realization of
  this category (the space of objects modulo morphisms) to form a manifold,
  in particular the isotropy groups (given by the morphisms from an object
  to itself) must be trivial.
Dropping this last condition yields the notion of an atlas for an
  orbifold; see e.g.\ \cite{moer}.
In complete analogy, a polyfold is the realization of a
  groupoid whose object and morphism spaces are M-polyfolds; see
  \cite[\S3]{HoferWysockiZehnder3}.

Next, a section of an orbi-bundle can be described as a functor between
  groupoids, i.e.\ a section of a vector bundle over the object space
  that is compatible with morphisms.
Since the object space is a manifold, the notions of smoothness and
  transversality directly transfer to sections of orbi-bundles.
It is only in the perturbation of sections that some new considerations
  are needed to achieve compatibility with morphisms.
In fact, transversality is generally only achieved by multi-valued
  perturbations, as described in e.g.\ \cite{cms,FO}.
However, this just adds an algebraic layer of more complicated
  book-keeping (best done in categorical terms) to the analysis of smooth
  sections of vector bundles.
Hence the same categorical constructions can be based on the analysis
  of scale-smooth Fredholm sections of M-polyfold bundles to yield a
  regularization theorem for scale-smooth Fredholm sections of polyfold
  bundles: There exists a class of multiperturbation sections so that
  the perturbed zero sets are compact weighted branched orbifolds,
  and unique up to cobordism.
In particular, they carry fundamental classes (with
  rational coefficients), whose inverse limit (constructed as in
  \cite[Thm.7.5.4]{mcduff-wehrheim}) provides a well defined virtual
  fundamental class on the moduli space.
\end{remark}

With this frame of reference in place, we now introduce the two core
  ideas of polyfold theory in more detail.

\subsection{Scale Calculus}  \label{ss:sc}
In order to motivate sc-Banach spaces and sc-calculus, we begin with a
  crucial observation: in almost all cases, the procedure to regularize
  a moduli space of Morse trajectories or pseudoholmorphic curves will,
  at some point, quotient by an action of a reparametrization group.
Furthermore, unless a geometric perturbation provides a smooth finite
  dimensional space of (smooth) solutions that is invariant under this
  action, the reparametrizations will need to be considered on an infinite
  dimensional space of maps.
However, as discussed in Example~\ref{ex:MorseDiff}, such actions are
  not continuously differentiable in the classical sense.
To explore this failure, we simplify the Morse theoretic example further
  to a compact domain\footnote
  { 
    For compact domains we have compact embeddings $\cC^\ell(S^1)\to
      \cC^k(S^1)$ for $\ell>k$, whereas the Morse setting with noncompact
      domain $\R$ will require the use of weighted Sobolev spaces to obtain
      scale Banach spaces as introduced below; see Lemma~\ref{ex:sobolev}
      for details.
  } 
  $S^1\cong\R/\Z$ and the target $\R$, so that we consider the modified
  shift map
  \begin{equation}\label{taus1}
    \tau: \R\times \mathcal{C}^1(S^1)\to \mathcal{C}^1(S^1) \qquad
    \text{given by}\qquad\tau(s,\gamma):=\gamma(s+\cdot).
  \end{equation}
The original motivation behind the development of scale calculus
  was to find a notion of differentiability in which the map given
  in \eqref{taus1} was smooth, and this was essentially be achieved by
  formalizing the weaker differentiability properties that the map \(\tau\)
  does satisfy.
To see this, we abbreviate $\cC^k:=\cC^k(S^1,\R)$, and note that one
  can verify the following:
  \begin{enumerate}
    \item the map \(\tau:\mathbb{R}\times\mathcal{C}^k\to \mathcal{C}^k\)
      is continuous for each \(k\in\mathbb{N}\);\label{en:sc diff 1}
    \item the map \(\tau:\mathbb{R}\times\mathcal{C}^{k+1}\to
      \mathcal{C}^k\) is differentiable for each \(k\in \mathbb{N}\),
      with differential
      \label{en:sc diff 2}
      \begin{equation*}
	\rD\tau:
	(\mathbb{R}\times\mathcal{C}^{k+1})\times(\mathbb{R}\times\mathcal{C}^{k+1})\to
	\mathcal{C}^k\quad\text{given by}\quad
	\rD_{(s, \gamma)}\tau\;( S, \Gamma) = S\tau(s, \gamma') + \tau(s, \Gamma);
	\end{equation*}
    \item for each \(k\in\mathbb{N}\) and \((s_0, \gamma_0)\in
      \mathbb{R}\times\mathcal{C}^{k+1}\), the differential \(\rD_{(s_0,
      \gamma_0)}\tau\) extends to a bounded linear operator \label{en:sc
      diff 3}
      \begin{equation*}
	\rD_{(s_0, \gamma_0)}\tau:\mathbb{R}\times\mathcal{C}^k\to \mathcal{C}^k ;
	\end{equation*}
    \item the map
      \((\mathbb{R}\times\mathcal{C}^{k+1})\times(\mathbb{R}\times\mathcal{C}^k)\to
      \mathcal{C}^k\), given by \((s, \gamma, S, \Gamma)\mapsto
      \rD_{(s,\gamma)}\tau(S, \Gamma)\) is continuous for each
      \(k\in\mathbb{N}\). \label{en:sc diff 4}
    \end{enumerate}
In particular, note that while the map
  \(\tau:\mathbb{R}\times\mathcal{C}^k\to\mathcal{C}^k\) fails to
  be differentiable for any \(k\in \mathbb{N}\), it nevertheless is
  continuous for each \(k\in \mathbb{N}\), and it gains regularity when
  we lower the regularity of the target space as in (ii).
This suggests that it is undesireable to consider \(\tau\) as a map to
  and from a fixed function space like \(\mathcal{C}^k\).
On the other hand, the various regularity properties of \(\tau\) and
  \(\rD\tau\)  hold for each \(k\in \mathbb{N}\).
This suggests that instead of thinking of \(\tau\) as a map
  \(\mathbb{R}\times\mathcal{C}^k\to \mathcal{C}^k\) for a fixed \(k\in
  \mathbb{N}\), we should instead regard it as a map between \emph{scales}
  of spaces \(\tau:(\mathbb{R}\times\mathcal{C}^k)_{k\in\mathbb{N}} \to
  (\mathcal{C}^k)_{k\in \mathbb{N}}\).

This collection of weaker differentiability properties then
  motivates the precise notion of a {\bf scale Banach space} (see
  Definition~\ref{def:scBanachSpace}) which consists of a nested sequence
  of Banach spaces, such as
  \begin{equation*}
    E_1=\mathcal{C}^1(S^1) \;\supset\; E_2=\mathcal{C}^2(S^1) 
    \;\supset\; E_3=\mathcal{C}^3(S^1) \;\supset\; \cdots  ,
    \end{equation*}
  which satisfy the following two properties:
  \begin{itemize}
    \item the inclusion of higher levels to lower levels
      is compact; e.g. for each \(\ell > k\), the inclusions
      \(E_\ell=\mathcal{C}^\ell(S^1)\to \mathcal{C}^k(S^1)=E_k\) are
      compact  (and hence continuous).
    \item the intersection of all spaces is dense in each level; e.g. the
      space of smooth functions \(E_\infty:=\mathcal{C}^\infty(S^1) =
      \cap_{\ell\in \mathbb{N}} \mathcal{C}^\ell(S^1)=\cap_{\ell\in
      \mathbb{N}} E_\ell\) is dense in each level
      \(\mathcal{C}^k(S^1)=E_k\).
    \end{itemize}

Given two scale Banach spaces, such as
  $(E_k=\mathbb{R}\times\mathcal{C}^k)_{k\in \mathbb{N}}$ and
  $(F_k=\mathcal{C}^k)_{k\in\mathbb{N}}$ as above, the notion of {\bf
  continuous scale differentiability} (sc$^1$) of a map $\tau:\E\to\F$
  is now given by formalizing the properties of the translation
  action \eqref{taus1} above.
More specifically, we require:
\begin{enumerate}
  \item the map \(\tau:E_k\to F_k\) is continuous for each
    \(k\in\mathbb{N}\);
  \item the map \(\tau:E_{k+1}\to F_k\) is differentiable for each
    \(k\in \mathbb{N}\);
  \item for each \(k\in\mathbb{N}\) and \(e\in E_{k+1}\), the
    differential \(\rD_e\tau\) extends to a bounded linear operator
    $\rD_e\tau:E_k\to F_k$;
  \item the map \(E_{k+1}\times E_k\to F_k\), given by \((e,h)\mapsto
    \rD_{e}\tau(h)\) is continuous for each \(k\in\mathbb{N}\).
  \end{enumerate}
In particular, property (i) is used as notion of scale continuity
  (sc$^0$) and properties (iii) and (iv) can be reformulated as scale
  continuity of the differential $\rD \tau$; for further details see
  Definition \ref{def:scTangentBundle}.

Taking the above as definition of sc\(^1\), the notions of higher scale
  regularity, namely sc\(^k\) for $k>1$, can be defined iteratively.
One can furthermore verify that the translation action $\tau$ is scale
  smooth; in other words \(\tau\) is  sc$^k$ for all $k\in\N_0$.
See Example~\ref{ex:transsmooth} for further details.
That \(\tau\) is scale smooth should not be surprising, since
  such regularity was exactly what motivated this new definition of
  differentiability.
A more surprising fact is that the chain rule holds for sc\(^1\) maps.
In other words, the composition of two maps of sc\(^1\)-regularity is
  again sc\(^1\), and the derivative of the composition is the composition
  of derivatives.
We note that this chain rule is not obvious from the above
  definition, and its validity is somewhat surprising since the classical
  differentiability in (ii) is achieved at only at the expense of a shift
  of 1 in scale level, and so it would seem that the composition of two
  such maps should only be classically differentiable with a shift of 2
  in scale level.
Nevertheless, the chain rule does hold; see Theorem \ref{thm:ChainRule}
  for further details.

Based on this new notion of differentiability which satisfies the
  chain rule, the further notions of calculus and differential geometry
  generalize more or less naturally to a scale calculus and scale
  differential geometry.
The next remark spells out why in finite dimensions these coincide with
  the classical notions and why they cannot coincide with Banach space
  notions in infinite dimensions.
\begin{remark} \hspace{2mm} \rm
\begin{enumlist}
  \item[(i)]
    The general definition of a scale Banach space requires
      compactness of the inclusions $E_{k+1} \subset E_k$ such as
      $\cC^{k+1}(S^1)\subset \cC^k(S^1)$, and this axiom is crucial for
      the proof of the chain rule.
  \item[(ii)] 
    Due to the compactness requirement, the only scale Banach spaces
      of the form \(E_0\supset E_0\supset \cdots \supset E_\infty=E_0\)
      (i.e.\ all levels are identical) are those for which \(E_0\) is a
      \emph{finite} dimensional vector space. 
    In such a case, all norms on $E_0$ are equivalent. 
    Hence the notion of scale differentiability differs from the notion
      of classical differentiability on any infinite dimensional Banach
      space.
  \item[(iii)]
    Due to the density requirement, the only scale structure on a finite
      dimensional vector space $E_0$ is the trivial sequence $E_0\supset
      E_0\supset \ldots \supset E_\infty=E_0$, and thus scale calculus in
      finite dimensions coincides with classical calculus; e.g. functions
      are $sc^k$ iff they are $\cC^k$.
  \item[(iv)]
    The density condition requires that the intersection of all scales
      (i.e. the infinity level $E_\infty$) is dense in each $E_k$.
    This means in particular that one can often make arguments on
      $E_\infty$ and use continuous extension to the completions $E_k$
      with respect to different norms.
    Moreover, this reflects the philosophy that we ultimately study the
      ``smooth'' points in $E_\infty$, whose topology is defined by a
      sequence of norms. 
    The scales $E_k$ then arise as completions in these norms.
  \end{enumlist}
\end{remark}

As previously noted, scale calculus is still insufficient to describe
  spaces of trajectories in which a sequence of unbroken gradient
  trajectories is allowed to converge to a broken gradient trajectory.
However, before moving on to the notion of sc-retracts and M-polyfolds,
  which deal with these issues,  we will first discuss how (uncompactified)
  moduli spaces of flow lines -- i.e.\ solutions of a flow ODE modulo
  reparametrizations -- can be described as the zero set of a scale smooth
  Fredholm section.
This will also exhibit the fact that the notions of scale Banach spaces
  and scale continuity are natural from yet another point of view, namely
  that of elliptic operators.
(In fact, scale structures did appear before in this context, e.g.\
  in \cite{Tr}, though not involving a new notion of differentiability.)

In the above simplification of the Morse example from paths to
  loops, let $\mathcal{C}^1(S^1,\R^n)^*$ be the subset of $\cC^1$-loops
  $\g:S^1\to\R^n$ such that $\g(s+\cdot)\neq \g$ for all $s\neq 0$; i.e.\
  $S^1$ acts freely on $\mathcal{C}^1(S^1,\R^n)^*$.
Then one can give $\mathcal{C}^1(S^1,\R^n)^*/S^1$, which is the space
  of non-constant loops in $\R^n$ modulo the reparametrization given in
  equation\eqref{taus1}, a scale smooth structure even though this action
  was not even classically differentiable.
Furthermore, given a vector field, denoted  $X:\R^n\to\R^n$, the flow
  lines (more precisely, the unparametrized orbits of period $1$) are
  the zeros of the scale smooth map
  \begin{equation} \label{sigma}
    \s \, :\;  \quo{\mathcal{C}^1(S^1,\R^n)^*}{S^1}
    \;\longrightarrow\; \quo{\mathcal{C}^1(S^1,\mathbb{R}^n)^*\times
    \mathcal{C}^0(S^1,\R^n)}{S^1} , \qquad \gamma \;\longmapsto\; \big(
    \gamma, \tfrac\rd{\rd t}\g - X(\g)\big) .
    \end{equation}

In the Morse theory case, we study \(\mathcal{C}^1(\mathbb{R},
  \mathbb{R}^n)^*/\mathbb{R}\) rather than \(\mathcal{C}^1(S^1,
  \mathbb{R}^n)^*/S^1\), and we consider a gradient vector field \(X =
  \nabla f\) induced by a Morse function \(f\) and metric on $\R^n$.
%K
%It is also necessary to restrict to a space of paths $\g:\R\to\R^n$
%  that \emph{exponentially} converge to critical points of $f$ as
%  $s\to\pm\infty$.
It is also necessary to restrict to a space of paths $\g:\R\to\R^n$ that 
converge to critical points of $f$ as $s\to\pm\infty$, and this necessitates 
a Fredholm setup in terms of Sobolev spaces. 
Also note that, strictly speaking, the map \(\sigma\) specified above
  should actually be regarded as a section of a bundle, which here we
  have canonically trivialized by $\g^*\rT\R^n\cong S^1\times\R^n$.
In either case, to discuss the analytic properties of this differential
  equation, we should now work in a local slice of the $S^1$-action;
  that is, we work in a codimension $1$ subspace of $\mathcal{C}^1(S^1,\R^n)$.
We will suppress this here since a finite dimensional condition does
  not affect the analytic behavior substantially; for example, it does
  not affect the whether or not operator is Fredholm.
In classical functional analysis, one would call $\s$ a Fredholm section
  if its linearizations are Fredholm operators\footnote
  {
    A linear map between vector spaces is called Fredholm if it has
      finite dimensional kernel and cokernel.
  }.
Indeed, the linearized operator at $\gamma\in \mathcal{C}^1(S^1,\R^n)$
  is $\tfrac\rd{\rd t} - \rD_\g X :  \mathcal{C}^1(S^1,\R^n) \to
  \mathcal{C}^0(S^1,\R^n)$, which is well known to be both Fredholm
  and elliptic.\footnote{
%K new footnote
In the present setup, the Fredholm property crucially relies on compactness of the domain $S^1$. 
To obtain a Fredholm setup for Morse theory one has to work with Sobolev spaces on the noncompact domain $\R$; see Example~\ref{ex:MorseTrajectorySpaces}.  
}
The corresponding elliptic estimates and elliptic regularity are
  easily phrased in scale calculus terms by saying that $\tfrac\rd{\rd
  t} - \rD_\g X : \bigl(\mathcal{C}^{1+k}(S^1,\R^n)\bigr)_{k\in\N_0} \to
  \bigl(\mathcal{C}^{0+k}(S^1,\R^n)\bigr)_{k\in\N_0}$  is a regularizing
  scale operator, which is equivalent to the following properties:
  \begin{enumerate}
    \item 
      $\tfrac\rd{\rd t} - \rD_\g X: \mathcal{C}^{1+k}(S^1,\R^n) \to
      \mathcal{C}^{0+k}(S^1,\R^n)$ is a bounded operator for each
      ${k\in\N_0}$;
    \item 
      if $\tfrac\rd{\rd t}\xi - \rD_\g X \xi \in
      \mathcal{C}^{0+k}(S^1,\R^n)$ for any $k\in\N_0$ then
      $\xi\in\mathcal{C}^{1+k}(S^1,\R^n)$.
    \end{enumerate}
Moreover, the Fredholm property of $\tfrac\rd{\rd t} - \rD_\g X :
  \mathcal{C}^1(S^1,\R^n) \to \mathcal{C}^0(S^1,\R^n)$
  together with these scale regularity properties now abstractly
  imply the Fredholm property on every scale $k\in\N$ of
  $\tfrac\rd{\rd t} - \rD_\g X :  \mathcal{C}^{1+k}(S^1,\R^n) \to
  \mathcal{C}^{0+k}(S^1,\R^n)$; further details can be found in
  Lemma~\ref{le:scfredlin}.
We note, however, that this does not provide a satisfactory Fredholm property for
  the nonlinear section \eqref{sigma}, since the listed properties
  are not sufficient to establish an implicit function theorem --
  even assuming surjective linearizations.
Indeed, the difficulty is that such a theorem is proved by means of a
  contraction property of the section in a suitable reduction.
Since the contraction will be iterated to obtain convergence, it needs
  to act on a fixed Banach space like \(\mathcal{C}^k(S^1, \mathbb{R}^n)\)
  for a fixed \(k\in \mathbb{N}\), rather than between different scales.
HWZ solve this issue by making the contraction property a part of the
  definition of a Fredholm section, and thereby they effectively build
  an implicit function theorem into the definition of a scale Fredholm
  section.

In light of this somewhat contrived definition, the miraculous feature
  then is that standard differential equations are in fact scale Fredholm.
In practice, the desired contraction property can be proven by
  establishing the classical Fredholm property of the linearized section,
  a nonlinear version of the regularizing property (ii) above for the
  section itself, classical differentiability of the section in all but
  finitely many directions, and certain weak continuity properties of
  these partial derivatives (details are provided via
  Lemma~\ref{le:scfred}).
These differentiability properties hold in applications to Morse theory
  and pseudoholomorphic curve moduli spaces since differentiability fails
  only in the directions of the finitely many gluing parameters.

\subsection{Retractions, splicings, and M-polyfolds}
\label{ss:ret}

To discuss the second core idea of polyfold theory in more detail,
  we return to the Morse theory case.
For simplicity let us consider the manifold \(X=\mathbb{R}^n\)
  and assume that the Morse function \(f:\R^n\to \mathbb{R}\) has
  precisely three critical points, denoted \(\Crit f = \{a, b, c\}\),
  which satisfy \(f(c)>f(b)>f(a)\), so that \(b=0\in \mathbb{R}^n\).
Let \(\mathcal{B}_a^c\), \(\mathcal{B}_a^b\), and \(\mathcal{B}_b^c\)
  respectively be the spaces of parametrized paths $\g:\R\to\R^n$ from
  \(a\) to \(c\), from \(a\) to \(b\), and from \(b\) to \(c\).
As in Example \ref{ex:MorseFredholm}, these spaces are invariant under
  the translation action $\tau$ given in (\ref{tau}).
Letting \(\mathbb{R}={\rm Aut}\) denote the automorphism group that
  acts via \(\tau\), we then define the spaces of trajectories (but
  not necessarily gradient trajectories) between critical points to be
  \(\mathcal{B}_a^c/{\rm Aut}\),  \(\mathcal{B}_a^b/{\rm Aut}\), and
  \(\mathcal{B}_b^c/{\rm Aut}\).
These are topological spaces equipped with the quotient topologies
  induced from the $\mathcal{C}^1$-topology on the parametrized paths.

In order to describe the compactified moduli space $\CM$
  of broken and unbroken Morse trajectories from $a$ to $c$
  as the zero set $\CM=\tilde{\sigma}^{-1}(0)$ of a section
  $\tilde{\sigma}:\widetilde{\mathcal{B}}\to\widetilde{\mathcal{E}}$, we
  need to construct a topological space $\widetilde{\mathcal{B}}$ of broken
  and unbroken trajectories which contains $\CM$ as a compact subset.
Furthermore, we wish that a suitable notion of smooth structure
  on \(\widetilde{\mathcal{B}}\) induces a smooth structure
  on $\ti\sigma^{-1}(0)$ whenever the section is transverse in the
  appropriate sense.
In the following, the construction of local models for such a space
  near broken trajectories will naturally give rise to sc-retractions.

To begin, we equip the unbroken trajectory spaces with sc-structures
  by using local slices as in Remark \ref{rmk:slice}.
For example, for the pair $a,c$ we have Banach manifold charts
  \(\Phi: \cV_a^c\to \mathcal{B}_a^c/{\rm Aut}\) of the form \(u\mapsto
  [\phi_a^c + u]\), where \(\phi_a^c:\R\to\R^n\) is a fixed smooth path
  from $a$ to $c$ for which \(\frac{\rd}{\rd t}{\phi}_a^c(0)\neq 0\), and
  ${\cV_a^c\subset\{u\in\cC^1(\R,\R^n) \,|\, \la u(0) , \frac{\rd}{\rd
  t}\phi_a^c(0) \ra = 0 \}}$ is neighborhood of $u\equiv 0$.
This is a local slice because $\phi_a^c + \rT_0\cV_a^c\subset
  \rT_{\phi_a^c}\cB_a^c$ is a complement to tangent space of the ${\rm
  Aut}$-orbit through $\phi_a^c$, which is spanned by $\frac{\rd}{\rd
  t}\phi_a^c$.

While the transition maps between such charts are not differentiable in
  any known Banach norm,  they are scale smooth when $\cV_a^c\subset E_0$
  is considered as open subset of an appropriate scale Banach space.
Due to the noncompact domain, this needs a more complicated scale than
  just $E_k=\cC^{1+k}(\R,\R^n)$; indeed, one should use exponentially
  weighted Sobolev spaces as in Example \ref{ex:sobolev}.
However, to simplify the exposition here let us pretend that
  \((E_k=\cC^{1+k}(\R,\R^n))_{k\in \mathbb{N}_0}\) is an sc-Banach space.
Then a cover by  charts of the above type gives \(\mathcal{B}_a^c/{\rm
  Aut}\) the structure of a scale manifold.
By only varying the reference path $\phi_a^b$ (or  $\phi_b^c$), we
  can obtain an analogous scale structure on \(\mathcal{B}_a^b/{\rm Aut}\)
  (or\(\mathcal{B}_b^c/{\rm Aut}\)).
Now the set of unbroken and broken trajectories, without yet a topology,
  is given by
  \begin{equation*} 
    \widetilde{\mathcal{B}}\;\;=\;\;\quo{\mathcal{B}_a^c}{{\rm
    Aut}}\quad \sqcup \quad \quo{\mathcal{B}_a^b}{{\rm Aut}}\times
    \quo{\mathcal{B}_b^c}{{\rm Aut}},
  \end{equation*}  
and our first goal is to equip this set with a topology which allows
  unbroken paths in $\qu{\mathcal{B}_a^c}{{\rm Aut}}$ to converge to broken
  paths in $\qu{\mathcal{B}_a^b}{{\rm Aut}}\times \qu{\mathcal{B}_b^c}{{\rm
  Aut}}$.
Polyfold theory accomplishes this by building on the well-known pregluing
  construction, which constructs unbroken trajectories near a broken
  trajectory.
More precisely, we fix representatives $\g_a, \g_b$ for a broken
  trajectory
  \begin{equation*}
    ([\gamma_a], [\gamma_b])\in \; \quo{\mathcal{B}_a^b}{{\rm Aut}}\times
    \quo{\mathcal{B}_b^c}{{\rm Aut}} ,
  \end{equation*}
  and choose charts for \(\mathcal{B}_a^b/{\rm Aut}\) and
  \(\mathcal{B}_b^c/{\rm Aut}\) given by local slices: scale
  smooth submanifolds \(\mathcal{H}_a^b = \phi_a^b + \cV_a^b\subset
  \mathcal{B}_a^b\) and \(\mathcal{H}_b^c = \phi_b^c + \cV_b^c
  \subset\mathcal{B}_b^c\) that contain $\g_a$ and $\g_b$ respectively.
Then for all sufficiently large \(R>0\) we define the {\bf pregluing
  map} by
  \begin{align}
    \oplus: (R_0, \infty)\times \mathcal{H}_a^b\times \mathcal{H}_b^c
    &\;\to\; \mathcal{B}_a^c\label{eq:oplusSec23} \\
    (R,u_a,u_b) &\;\mapsto\; \oplus_R (u_a, u_b) := \beta u_a(\cdot
    + {\textstyle \frac{R}{2}}) + (1-\beta)u_b( \cdot - {\textstyle
    \frac{R}{2}}), \notag
  \end{align} 
  where $\beta:\R\to[0,1]$ is a smooth cutoff function with
  $\beta|_{(-\infty,-1]}\equiv 1$ and $\beta|_{[1,\infty)}\equiv 0$.
See Figure~\ref{fig:plus minus gluing} for an illustration of the
  pregluing (and anti-gluing) map.

\begin{figure}
  \includegraphics[scale=0.8]{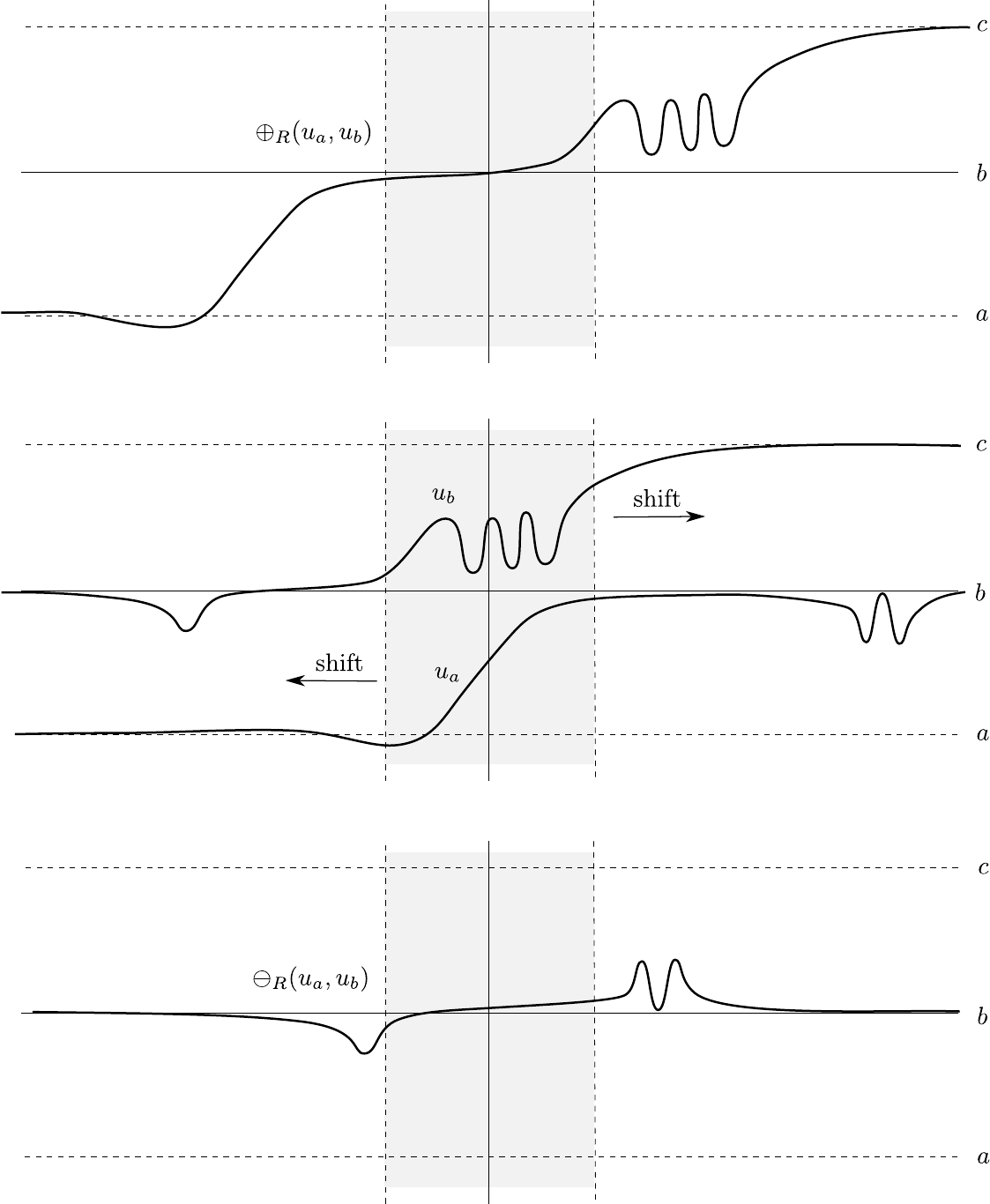}
  \caption{An example of plus gluing (i.e.\ pregluing) and minus gluing
    (i.e.\ anti-gluing) of two smooth paths $u_a,u_b$ from $a$ to $b$
    and  from $b$ to $c$.}
  \label{fig:plus minus gluing}
\end{figure}

The topology on the space of broken and unbroken trajectories $\Ti\cB$
  is now constructed by viewing the pregluing map as map to the quotient
  $\qu{\mathcal{B}_a^c}{ {\rm Aut}}$, extending this map to gluing
  parameter $R=\infty$ by $(\infty,u_a,u_b)\mapsto\bigl([u_a],[u_b]\bigr)$,
  and requiring this extended pregluing map to be open.
In other words, a basis of open sets in $\Ti\cB$ is given by images
  under the extended pregluing map of open subsets of product type
  $$
    \cU:= (R_0, \infty]\times \mathcal{H}_a^b\times \mathcal{H}_b^c
    \; \subset\;  (0, \infty]\times \bigl(\phi_a^b +
    \cC^1(\R,\R^n)\bigr)\times \bigl(\phi_b^c + \cC^1(\R,\R^n)\bigr) .
  $$
Here the ambient space on the right can be equipped with a scale
  smooth structure (with boundary) by replacing $\cC^1(\R,\R^n)$ with a
  scale of weighted Sobolev spaces, as mentioned above, and by fixing a
  homeomorphism $[0,1) \cong (0,\infty]$ that identifies the boundaries
  $0$ and $\infty$.
The latter is the notion of a {\bf gluing profile}, which in polyfold
  theory is usually chosen as the {\bf exponential profile} 
  \begin{equation}\label{expprofile}
    [0,1) \;\to\; (0,\infty], \qquad \tau \; \mapsto\; e^{1/\tau} - e.
    \end{equation}
The choice of the exponential gluing profile in particular ensures that
  the following constructions extend scale smoothly to the boundary.
One could also hope to obtain a chart for $\Ti\cB$ near the broken
  path \(([\gamma_a], [\gamma_b])\) from the map
  \begin{equation} \label{wishchart}
    \Phi:\mathcal{U}\to \Ti\cB \qquad(R, u_a, u_b)\mapsto \begin{cases}
    [\oplus_R(u_a, u_b)]  &; R <\infty ,\\
    ([u_a], [u_b]) &; R=\infty .
    \end{cases}
    \end{equation}
Although \(\Phi|_{\{R<\infty\}}\) is an sc-smooth map to
  $\qu{\cB_a^c}{\Aut}$, it is far from being a local homeomorphism since
  it is not even a bijection except for its restriction to $\{R=\infty\}$.
To see this, observe that for fixed \(R<\infty\), the two maps
  \(\oplus_R(u_a, u_b)\) and \(\oplus_R(u_a+v_+, u_b+v_-)\) are equal
  whenever \(v_\pm\) have support in a sufficiently small neighborhood
  of \(\pm \infty\).
At this point the core idea of polyfold theory arises: obtain a chart
  by restricting \(\Phi\) to an appropriate subset of $\cU$, which is
  then used as a local model for the scale smooth structure on $\Ti\cB$.
In other words, we aim to achieve the following: 
  \begin{enumerate}
    \item Find a subset \(\mathcal{K}\subset\mathcal{U}\) for
      which \(\Phi\big|_{\mathcal{K}}\) is a homeomorphism to its
      image.\label{en:retraction_req1}
    \item Equip sets \(\mathcal{K}\) of this type with a notion of
      scale smooth structure. \label{en:retraction_req2}
  \end{enumerate}
We will see that this can be achieved by describing $\cK$ as the image
  of a retraction on $\cU$.
Moreover, this retraction will appear naturally from the idea of 
  keeping track of the information lost during pregluing for $R<\infty$.
This is accomplishved via the so-called {\bf anti-gluing map}
  \(\ominus_R\), which is given by a complementary interpolation of the
  same shifts as in the pregluing map $\oplus_R$.
More specifically, the combination of both maps is given by a pair of
  reparametrizations together with multiplication by an invertible matrix
  of cutoff functions:
  \begin{equation*}
    \left(
    \begin{matrix}
    \oplus_{R}(u_a, u_b)\\
    \ominus_{R}(u_a, u_b)
    \end{matrix}
    \right)
    =
    \left(
    \begin{matrix}
    \beta & 1-\beta\\
    \beta-1 & \beta
    \end{matrix}
    \right)
    \left(
    \begin{matrix}
    u_a(\cdot +{\textstyle \frac{R}{2}})\\
    u_b(\cdot -{\textstyle \frac{R}{2}}) 
    \end{matrix}
    \right) .
    \end{equation*}
For each fixed \(R<\infty\), this is a bijection by invertibility of
  the matrix at every $t\in\R$.
In fact, one can check that it gives rise to an sc-smooth diffeomorphism
  \begin{align*}
    \boxplus: \{(R,u_a,u_b)\in \mathcal{U} \,|\, R<\infty\} &\;\to\;
    \quo{\mathcal{B}_a^c}{{\rm Aut}}\times \cC^1(\mathbb{R},
    \mathbb{R}^n) \\
    (R, u_a, u_b) &\;\mapsto\; \boxplus_R(u_a, u_b) :=\big([\oplus_R(u_a,
    u_b)], \ominus_R(u_a, u_b)\big) .
  \end{align*}
Moreover, in appropriate charts for domain and target, each $\boxplus_R$
  can be viewed as linear isomorphism $\rT_0\cV_a^b\times \rT_0\cV_b^c\to
  \rT_0\cV_a^c \times \cC^1(\mathbb{R}, \mathbb{R}^n)$,
  which shows that $\ker\ominus_R$ is a complement to $\ker\oplus_R$.
This achieves the first aim and gives an approach to the second:

\begin{enumerate}
  \item
  The map $\Phi|_\cK$ in \eqref{wishchart} restricts to a bijection on
    \begin{equation*}
      \mathcal{K}:=\{(R, u_a, u_b)\subset\mathcal{U} \,|\, \ominus_R(u_a,
      u_b)=0 \;\text{or}\; R=\infty\} .
    \end{equation*}
  To check that $\Phi|_\cK$ is a homeomorphism, one can use the
    observation that $(R, u_a, u_b)=\Phi^{-1}([v])$ is the unique solution
    of $\boxplus_R(u_a,u_b) = \bigl( [v] , 0\bigr)$.  \label{en:1}
  \item
  After possibly shrinking \(\mathcal{U}\), the latter gives rise to
    a description of the set \(\mathcal{K}\) as fixed point set of the
    sc-smooth map
    \[
    r: \cU\to \cU ,\qquad r(R,u_a,u_b) =\begin{cases}
    \boxplus_R^{-1}\bigl( [\oplus_R(u_a, u_b)] , 0\bigr)   &; R <\infty ,\\
    (R,u_a, u_b) &; R=\infty .
    \end{cases}
    \]
  In fact, this map satisfies the retraction property $r\circ r =
    r$ since for $R<\infty$ it is of the form $\boxplus_R^{-1}\circ\pr
    \circ\boxplus_R$, with $\pr(u,v):=(u,0)$ satisfying $\pr\circ\pr=\pr$.
In particular, \(\mathcal{K}=r(\mathcal{U})\) is an sc-retract; that is,
  it is the image of an sc-smooth retraction. \label{en:2}
\end{enumerate} 

To accomplish our aims, it remains to show that \(\mathcal{K}\) carries
  a meaningful notion of scale smoothness.
In other words, we need a notion of scale-differentiability for maps
  \(\Psi:\mathcal{K}\to \mathbb{F}\) to some other sc-Banach space
  \(\mathbb{F}\).
The notion of sc-continuity for such maps is naturally given since
  \(\mathcal{K}\) carries an sc-topology induced from \(\mathcal{U}\).
The notion of sc\(^1\) from scale calculus is also well defined if
  \(\mathcal{K}\) is an open subset of an sc-Banach space.
However, in our Morse theory example \(\mathcal{K}\) has empty interior.
Since  \(r\big|_{\mathcal{K}}=\id_\cK\), a natural extension of $\Psi$
  to a map from an open subset of an sc-Banach space is $\Psi\circ r :
  \cU \to \F$.
We can then define the map \(\Psi:\mathcal{K}\to \mathbb{F}\) to be
  sc\(^k\) if and only if the map \(\Psi\circ r:\mathcal{U}\to \mathbb{F}\)
  is sc\(^k\).
Similarly, we define the tangent spaces $\rT_k\cK$ as fixed point set
  of the linearized retraction $\rd_k r$.
These definitions makes sense (e.g.\ satisfy the chain rule and depend
  only on $\cK$, not the choice of $r$) due to the retraction property
  $r\circ r = r$.
In particular, the latter implies that the differential $\rd r=\rd r
  \circ \rd r$ is a retraction as well, so that the tangent bundle of an
  sc-retract is an sc-retract itself.
This establishes a notion of scale smooth structure on $\cK$, as aimed
  for in (\ref{en:retraction_req2}).
Further details can be found in Example \ref{ex:pregluing_retraction}.

From this Morse theory example, we see the utility of an sc-smooth
  retraction \(r:\mathcal{U}\to \mathcal{U}\), which both characterizes
  the subset \(\mathcal{K}=r(\cU)\) on which a homeomorphic chart map
  \(\Phi\) is defined, and provides a means to establish the notion of
  sc-differentiability on this subset.
Such sc-smooth maps satisfying the retraction property \(r\circ r=r\)
  are called {\bf sc-smooth retractions}, and their images are called
  {\bf sc-retracts}.
These sc-retracts, together with a homeomorphism \(\Phi:\cK \to\Ti\cB\),
  form the local models of M-polyfolds.
That is, an {\bf M-polyfold} is a topological space $\Ti\cB$ that is
  locally homeomorphic to sc-retracts, such that the transition maps
  $\Phi^{-1}\circ\Phi':\cK' \to \cK$ are sc-smooth in the sc-retract
  sense that $\Phi^{-1}\circ\Phi'\circ r': \cU' \to \cU$ is sc-smooth.
The above outline can be fleshed out to prove that $\Ti\cB$ is an
  M-polyfold.

In suitable coordinates, the sc-smooth retraction for Morse theory
  introduced above, and in fact all sc-retractions arising in applications
  to date, have a rather specific form, namely
  \begin{equation*}
    r:[0,1)^k\times \mathbb{E}\to [0, 1)^k\times
    \mathbb{E}\qquad\text{given by}\qquad r(v, e) = (v, \pi_v e),
    \end{equation*}
  where \(\mathbb{E}\) is an sc-Banach space, \(v\) is thought of as a
  gluing parameter, and \(\pi_v: \mathbb{E}\to \mathbb{E}\) is a family
  of linear projections.
Note that the sc-smoothness conditions on $r$ do not require
  $v\mapsto\pi_v$ to be continuous in the operator topology, but just
  ``pointwise'' as map $(v,e)\mapsto \pi_v e$. 
This allows the image $\pi_v\E$ to jump in dimension as $v$ varies.
Such retractions (given by a family of projections) are called
  {\bf splicings}; the induced sc-retracts are called {\bf splicing
  cores}; and they were used as local models for M-polyfolds in
  the early polyfold literature; c.f.\ \cite{HoferWysockiZehnder1,
  HoferWysockiZehnder2,HoferWysockiZehnder3}.

In order to achieve the ultimate goal of describing the
  compactified Morse moduli space $\CM$ as the zero set of a section
  $\ti\s:\Ti\cB\to\Ti\cE$ in a bundle that is sufficiently rich for a
  regularization theorem similar to Theorem~\ref{thm:model}, it remains
  to find a suitable notion of Fredholm sections in M-polyfold bundles.
Here a notion of finite dimensional kernels and cokernels with constant
  index is necessary in order to have any hope for the zero set of a
  transverse section to be a finite dimensional manifold.
However, note that in the Morse theory example, based on our expectation
  of what its zero set should be, the section in the pregluing chart must
  roughly have the form
  \[
    \ti\s(R,u_a,u_b) = 
    \begin{cases}
    \bigl( \frac\rd{\rd t} \bigl( \oplus_R(u_a,u_b) \bigr) - \nabla f\bigl(
    \oplus_R(u_a,u_b) \bigr)  \bigr) &; R<\infty \\
    \bigl( \frac\rd{\rd t} u_a - \nabla f(u_a)  , \frac\rd{\rd t} u_b -
    \nabla f(u_b) \bigr) &; R=\infty .
    \end{cases}
  \]
More specifically, the bundle \(\widetilde{\mathcal{E}}\to
  \Ti\cB\) must have fibers isomorphic to $\cC^0(\R,\R^n)$ over
  points such as $[\oplus_R(u_a,u_b)]$ in the interior of $\Ti\cB$, and
  it must have fibers isomorphic to $\cC^0(\R,\R^n)\times \cC^0(\R,\R^n)$
  over broken trajectories such as $\bigl([u_a],[u_b]\bigr)$.
This can be achieved by constructing $\Ti\cE$ from pregluing maps along
  the same lines as for $\Ti\cB$.
An important feature of this construction is that, roughly speaking, the
  fibers of $\Ti\cE$ jump in the same way as the tangent space $\rT\Ti\cB$.
In turn, this will allow for a meaningful Fredholm theory.

To define the notion of a {\bf scale Fredholm section}, one could try to
  proceed along the lines of the construction of a scale smooth structure
  on an sc-retract $\cK=r(\cU)$.
Note however that the linearization of $\ti\s\circ r$ has infinite
  dimensional kernel as soon as $\rd r$ does, which in the Morse theory
  example is the case whenever $R<\infty$.
At the same time, if $\cK'$ is the sc-retract modeling the bundle
  $\cE$, then $\rT\cK'$ has infinite codimension in each
  fiber over $R<\infty$.
Polyfold theory obtains a Fredholm theory by introducing the notion
  of a filled section, which in local charts is given as an sc-smooth
  extension $\overline\s:\cU\to\cU'$ of the section $\ti\s:\cK\to\cK'$
  to open subsets of sc-Banach spaces.
The filled section is required to have the same zero set
  $\overline\s^{-1}(0)=\ti\s^{-1}(0)$ as the original section, and to
  not contribute to the Fredholm index.
In the setting of splicings, this means that the bundle splicing has
  the form
  $$
    \rho:[0,1)^k\times \mathbb{E}\times \F \to [0, 1)^k\times \mathbb{E}
    \times \F , \qquad\rho (v, e, f) = (v, \pi_v e, \Pi_v f),
  $$
  so that the fibers of $\Ti\cE\to\Ti\cB$ are given by $\im\Pi_v$ over
  $\{v\}\times \im\pi_v$, and there is an sc-smooth family of isomorphisms
  $\ker\Pi_v\cong\ker\pi_v$ between the kernels of the two splicings,
  as the gluing parameter $v$ varies.
Such fillers can typically be constructed via the full gluing map
  \(\widetilde{\boxplus}=([\widetilde{\oplus}], \widetilde{\ominus})\),
  where the nonlinear PDE must naturally be applied in the first factor,
  and a linearized PDE provides an isomorphism that acts on the second
  factor.

Based on these Fredholm notions in the context of scale-calculus and
  sc-retracts, one can then develop a perturbation and stability theory
  for scale Fredholm sections, which culminates in the Regularization
  Theorem \ref{thm:PolyfoldRegularization} stated above.

%%%%%%%%%%%%%%%%%%%%%%%%%%%%%%%%%%%%%%%%%%%%%%%%%%%%%%%%%%%%%%%%%%%%%%%%%%%%%%%%
%%%%%%%%%%                            SECTION                          %%%%%%%%%
%%%%%%%%%%%%%%%%%%%%%%%%%%%%%%%%%%%%%%%%%%%%%%%%%%%%%%%%%%%%%%%%%%%%%%%%%%%%%%%%
%
\section{Road maps for regularization approaches} 
%K
\label{s:roadmaps}

In this section we compare the polyfold approach to regularizing
  moduli spaces to the geometric and virtual approaches in order to
  exhibit how the classical ingredients (compactness, quotienting by
  reparametrizations, Fredholm theory, gluing, etc.) are present in each
  of the approaches but with changing order and significance.
We will outline the basic steps in each of these approaches via the
  example of Morse theory, and we do so using the setup from
  Examples~\ref{ex:Morse} and ~\ref{ex:MorseFredholm}.
In more general abstract terms, we are discussing the regularization of
  a compactification $\CM$ of a moduli space $\cM$, given by the solutions
  to a PDE modulo the reparametrization action of an automorphism group
  $\Aut$.
Here and throughout, we will assume that \(\Aut\) acts freely on the
  space of solutions, which we recall is always the case in Morse theory.

For a more detailed account of Morse theory along these lines, see
  \cite{schwarz:Morse,AD2010}.
Note however that the regularization of the Morse moduli spaces does
  not actually require their study as moduli spaces of a PDE.
Rather, an entirely finite dimensional setup as spaces of trajectories
  under a smooth flow map yields the regularization as manifolds with
  boundaries and corners most effectively, e.g.\ \cite{w:Morse}.

%%%%%%%%%%%%%%%%%%%%%%%%%%%%%%%%%%%%%%%%%%%%%%%%%%%%%%%%%%%%%%%%%%%%%%%%%%%%%%%%
%%%%%%%%%%                       SUB-SECTION                           %%%%%%%%%
%%%%%%%%%%%%%%%%%%%%%%%%%%%%%%%%%%%%%%%%%%%%%%%%%%%%%%%%%%%%%%%%%%%%%%%%%%%%%%%%
\subsection{The geometric approach} \label{ss:geometric}
In this section we describe techniques that obtain transversality by
  perturbing (or exploiting) geometric structures in the moduli problem;
  we call such techniques the ``geometric approach.''
In the case of Morse theory, the given moduli problem is the compactified
  Morse moduli space $\CM$ for a fixed Morse function $f:X\to\R$ and any
  Riemannian metric $g$ on $X$.
This moduli space decomposes into (not necessarily connected) components
  $\CM(x_-,x_+)$ of (possibly broken) Morse trajectories between pairs
  of critical points $x_\pm\in\Crit f$.
The goal of regularization is to replace $\CM$ by a regularized space
  $\CM\,\!'$, which is a manifold with boundary and corners with components
  $\CM\,\!'(x_-,x_+)$, whose first boundary stratum (excluding the higher
  corner strata) is a fiber product of its interior $\cM'\subset \CM\,\!'$
  with itself.  In particular, it should have the form:
  $$
    \partial\CM\,\!' \;=\;\cM' \underset{\scriptscriptstyle \Crit
    f}{\times}\cM'\;=\;{\textstyle \bigcup_{x_-,x,x_+\in\Crit f}}
    \cM'(x_-,x)\times \cM'(x,x_+).
  $$
The signed count of the $0$-dimensional component of $\CM\,\!'$ then
  defines the Morse differential $\partial$, and the boundary structure
  of the $1$-dimensional component establishes $\partial\circ\partial = 0$.
An additional step is then needed to prove independence of the induced
  Morse homology from both the choice of regularization $\CM\,\!'$
  and the choice of $(f,g)$.
For other moduli problems, we write $\cM' \tilde\times\cM'$ for analogous
  fiber products, even if we expect the regularized moduli space to have
  no boundary (which is the case in Gromov-Witten).
The basic order of constructions in geometric approaches is: {\bf 1)
  transversality, 2) quotient, 3) gluing}; where reduction to finite
  dimensions occurs after transversality is achieved.
Such constructions can be roughly broken down into the following eight
  steps -- with adjustments in the case of ``codimension 2 gluing''
  discussed later.
\begin{enumlist}
  \item {\bf Fredholm setup:}
    Set up the PDE (e.g.\ gradient flow equation $\frac\rd\dt \g - \nabla f
      = 0$) as smooth section $\s:\cB\to\cE$ of a Banach space bundle $\cE
      \to \cB$ over a Banach manifold $\cB$ of maps (e.g.\ $\g:\R\to X$
      with suitable convergence to critical points).
    This section should be Fredholm in the sense that the linearizations
      $\rD_b \s : \rT_b\cB \to \cE_b$ at zeros $b\in \s^{-1}(0)$ are Fredholm
      operators.
    Moreover, the section $\s$ will be equivariant under the action
      of the automorphism group ${\rm Aut}$ on $\cE\to\cB$, so that the
      uncompactified moduli space is given as quotient of the zero set $\cM=
      \s^{-1}(0)/{\rm Aut}$.
  \item {\bf Geometric perturbations:}
    Find a family of smooth sections $(p:\cB\to\cE)_{p\in\cP}$
      parametrized by a Banach manifold $\cP$, with the following
      properties.
    \begin{enumerate}
      \item[(P1)] 
	For each $p\in\cP$ the perturbed solution space
	  $(\sigma+p)^{-1}(0)$ is invariant under the action of $\Aut$.
	(Usually this is achieved by using $\Aut$-equivariant sections
	  $p$.)
      \item[(P2)]  
	For each $p\in\cP$ the perturbed solution space
	  $(\sigma+p)^{-1}(0)$ has the same compactification properties
	  as the unperturbed space $\sigma^{-1}(0)$.
      \item[(P3)] 
	The ``universal moduli space'' $\widehat\cM := \bigl\{ (b,p)
	  \in \cB\times\cP \,\big|\, s(b) + p(b) = 0 \bigr\}$ is cut out
	  transversely and has the structure of a Banach manifold.
	That is, for each $(b,p)\in\widehat\cM$ we have a surjective
	  linearized operator $\rT_b\cB \times \rT_p\cP \to \cE_b$,
	  given by $( \xi , \eta ) \mapsto \rD_b (s+p) (\xi)  + \eta(b)$.
      \end{enumerate}
    (For Morse theory, the perturbations could be $p(\g)=\nabla f(\g)
      - \nabla' f(\g)$, where $\nabla'$ is the gradient with respect to
      another metric $g'$ on $X$.)
  \item {\bf Sard-Smale Theorem (automatic):}
    Given a family of perturbations $\cP$ as described, the Sard-Smale
      theorem guarantees a comeagre\footnote
      {
	A subset of a topological space is said to be comeagre if it is
	  the countable intersection of sets with dense interior.
	In a Baire space (such as any complete metric space), this implies
	  density.
	Alternatively, the complement of a comeagre set is meagre, i.e.\
	  the countable union of sets that are nowhere dense.
	Note however, that the commonly used term ``second category''
	  only refers to sets that are not meagre, hence may fail to be dense.
      } 
  set $\cP^{\rm reg}\subset\cP$ of regular values of the canonical
  projection $\pr: \widehat{\mathcal{M}}\to \mathcal{P}$.
Moreover, a little functional analysis (see \cite[Lemma A.3.6]{MDSa})
  shows that for $p\in\cP^{\rm reg}$ the perturbed section $\s_p:=\s+p$
  is transverse to the zero section, yet it is still ${\rm
  Aut}$-equivariant.
Hence, by the implicit function theorem, $\s_p^{-1}(0)\subset\cB$ is a
  smooth submanifold of finite dimension on which $\Aut$ acts, and the
  dimension is given by the Fredholm index.
(For Morse theory, this would pick out the metrics that satisfies the
  Morse-Smale condition:  transversal intersection of stable and unstable
  manifolds.)

\item {\bf Quotient:}
Check that the action of ${\rm Aut}$ on $\s_p^{-1}(0)$ is smooth, free,
  and properly discontinuous.
Then the moduli space $\cM_p := \s_p^{-1}(0) / {\rm Aut}$ is a smooth
  manifold.

\item {\bf Gluing:}
Construct a gluing map $\tilde\oplus:  (R_0,\infty) \times \cM_p
  \tilde\times \cM_p  \hookrightarrow \cM_p$ that is an embedding
  (e.g.\ for fixed critical points it should map $\cM_p(x_-,x)
  \times \cM_p(x,x_+)$ to paths parametrized by the gluing parameter
  $(R_0,\infty)$ in $\cM_p(x_0,x_+)$).
The construction of $\tilde\oplus$ involves a pregluing map $\oplus:
  (R_0,\infty) \times \s_p^{-1}(0) \tilde{\times} \s_p^{-1}(0)  \to \cB$
  similar to \eqref{eq:oplusSec23}, and an implicit function theorem
  to determine exact solutions.

\noindent
{\small {\bf Small print on corners:}
This technique is usually only applied to glue $0$-dimensional components
  or compact subsets of the fiber product.
More generally, to give a higher dimensional moduli space
  the structure of a manifold with boundary \emph{and corners} one would
  have to construct higher gluing maps $\tilde\oplus: (R_0,\infty)^\ell
  \times \tilde{\times}^{\ell+1} \cM_p  \hookrightarrow \cM_p$ which cover
  the overlap of the basic gluing maps, and one would also need to check
  smoothness of transition maps and verify a cocycle condition.
  \footnote
  {
  An abstract manifold (without underlying topological space)
    can be constructed from a tuple of open subsets $U_i\subset\R^n$
    by specifying transition maps $\phi_{ij}:U_{ij}\to U_j$ on open
    subsets $U_{ij}\subset U_i$ that satisfy the cocycle conditions
    $\phi_{jk}\circ\phi_{ij}=\phi_{ik}$ on appropriate domains.
  Alternatively, rather than require cocycle conditions, one could
    instead work with a given compact space $\CM\,\!'$ and simply construct
    the gluing maps as embeddings into this.
  Then cocycle conditions for the transition maps hold automatically.
  Otherwise, this issue is known as constructing ``associative gluing
    maps.''
  }
}
\item {\bf Coherence:}
Ensure that the choice of perturbation $p$ can be made ``coherently'';
  that is, check that perturbations are compatible with the gluing map.
Consequently steps  2) - 5) are interwoven, and they are potentially
  organized by a hierarchy of connected components of $\CM$, such as
  by the difference in Morse indices of $x_\pm$ for the components
  $\CM(x_-,x_+)$.

\item {\bf Compactness:}
Check that the complement of the gluing image, $\cM_p \setminus
  \im\tilde\oplus$, is compact.
Then construct a  compactification of the perturbed moduli space as
  $\CM_p = \bigl( \cM_p \sqcup  (R_0,\infty] \times \cM_p \tilde{\times}
  \cM_p \bigr) / \tilde\oplus$.
After choosing a homeomorphism $(R_0,\infty]\cong [0,1)$, this yields
  a smooth manifold with boundary $\{\infty\}\times\cM_p \tilde{\times}
  \cM_p$.

\noindent
{\small {\bf Small print on corners:}
If the gluing maps have overlaps, e.g.\ due to higher gluing maps,
  then one would have to add their domains $(R_0,\infty]^\ell \times
  \tilde{\times}^{\ell+1} \cM_p$ to $\CM_p$ and take the quotient by all
  gluing maps.
However, this requires the cocycle condition.
If this can be satisfied, then $\{\infty\}^\ell \times
  \tilde{\times}^{\ell+1} \cM_p$ forms the $\ell$-th corner stratum
  of $\CM_p$.
}
 
\item {\bf Invariance:}
Prove that the algebraic structures (e.g.\ the Morse chain complex)
  arising from different choices in the previous steps, in particular the
  choice of perturbation, are equivalent in an appropriate sense (e.g.\
  chain homotopic).
This usually involves the construction of a cobordism from a moduli
  space involving a homotopy of choices.
\end{enumlist}

When applied to a moduli space of pseudoholomorphic curves, Steps 1 -- 3
  remain unchanged, with $\cB$ consisting of maps from a fixed Riemann
  surface $\Si$, possibly with additional marked points and possibly
  varying complex structure on $\Si$.
(Note that we cannot work with a Deligne-Mumford type space of Riemann
  surfaces modulo biholomorphisms, since the corresponding space of
  maps and surfaces does not have a natural Banach manifold structure;
  see \cite[\S 3.2]{mcduff-wehrheim}.)
Then the section $\s$ is given by the Cauchy-Riemann operator -- but
  possibly with further conditions on (for example) the evaluation map
  at the marked points.
Finally, $\Aut$ is the group of holomorphic automorphisms of the
  underlying complex curve $\Si$.
(In the case of varying complex structures, one usually reduces the
  space of complex structures so that there are no further automorphisms.)
Here the requirement that $\Aut$ acts freely on $\s_p^{-1}(0)$
  is rather restrictive since it excludes perturbed solutions $u$ with
  nontrivial isotropy, that is $\phi\ne\id_\Si$ such that $u\circ\phi=u$.
{\it If} $\Aut$-equivariant transversality can be achieved in Step 2,
  then nontrivial {\it finite} isotropy groups could be allowed in Step
  4 with the result of $\s_p^{-1}(0) / {\rm Aut}$ being an orbifold.
However, holomorphic curves with nontrivial isotropy also lack the
  injectivity properties that are needed for the common approaches to
  achieving transversality (P3); see Remark~\ref{rmk:injective}.
In special cases, it might be possible to overcome this transversality
  issue by enriching the geometric approach with ``groupoid'' or
  ``multivalued perturbation''\footnote
  {
    A sketch can e.g.\ be found in \cite[\S 5]{Salamon}, but note that
      the proof of the local slice theorem there requires more geometric
      methods -- e.g.\ slicing conditions -- rather than an implicit
      function theorem for the action.
  }
  methods.

More abstractly,  the existence of perturbations as required in Step
  2 is {\it not} a general fact for equivariant Fredholm sections,
  since many useful classes of perturbations (like the class of compact
  perturbations) need not preserve the compactness properties of the
  solution set for general non-linear Fredholm problem.
Furthermore, the equivariance and transversality properties (P1) and
 (P3) are often mutually exclusive requirements.

For the Cauchy-Riemann operator $\pbar$, the natural geometric structure
  to perturb is the given almost complex structure $J$.
This means that the perturbations $p\in\cP$ are of the form $p(u) = \frac
  12 (J-J') \rd u \circ j$ for some other almost complex structure $J'$.
From the abstract functional analytic point of view, this is a
  perturbation of the same order as the differential operator, so the
  Fredholm property is preserved only by a homotopy of semi-Fredholm
  operators (using the elliptic estimates for each Cauchy-Riemann
  operator together with the connectedness of the space of compatible
  almost complex structures).
For the compactness property (P2) we need to use our geometric
  understanding of $J$-holomorphic curves for any compatible $J$ to see
  that Gromov compactness persists.
However, comparing the requirements for equivariance (P1) and
  transversality (P3), as in the following remark, one sees that almost
  complex structures only provide the required set of perturbations if,
  roughly speaking, the pseudoholomorphic maps are somewhere injective
  along any orbit of a point in the domain $\Si$ under the automorphism
  action.
This follows from the invariance of $J$ along $\Aut$-orbits in $\Si$.
Further common geometric perturbations are Hamiltonian vector fields.
These are lower order (compact) perturbations, which otherwise are used
  in close analogy to the perturbations in the almost complex structure.

{\small
\begin{remark}[{\bf Small print on injectivity requirements}] \rm
  \label{rmk:injective}
Let us semi-formally unravel the equivariance property (P1) and the
  universal transversality property (P3) when we perturb by a space
  $\cJ$ of possibly domain dependent compatible almost complex structures
  $J:\Si\to \cJ(M,\omega)$.

\begin{enumerate}
  \item[(P1)]
  Invariance of the solution set $\{u:\Si\to M \,|\, \pbar u = 0\}$
    under reparametrization by an automorphism $\phi:\Si\to\Si$ requires
    $J:\Si \to \cJ(M,\omega)$ to satisfy $J\circ\phi = J$.
  In particular, $J(z)$ must be constant along orbits $z\in
    \{\phi(z_0)\,|\, \phi\in\Aut\}$ of the automorphism group, and the
    same holds for infinitesimal variations $Y\in \rT_J\cJ$.
  \item[(P3)] 
  Transversality of the universal moduli space at $\pbar u = 0$
    requires, roughly speaking, that the only element $\eta \in \ker
    (\rD_u\pbar)^*$ in the kernel of the dual linearized Cauchy-Riemann
    operator that satisfies $\int_\Si \la \eta(z)\circ j , Y(z,u(z))
    \rd_z u  \ra = 0$ for all $Y \in \rT_J \cJ$ is $\eta = 0$.
  \end{enumerate}
  Assuming $\eta(z_0)\neq 0$ in contradiction to (P3), linear algebra
    guarantees the existence of $Y\in \rT_J \cJ$ such that $\la \eta(z_0)
    \circ j , Y(z_0,u(z_0)) \rd_{z_0} u \ra >0$, as long as $\rd_{z_0}
    u \neq 0$.
We then wish to cut off $Y$ near $(z_0,u(z_0))\in \Si\times M$ so that
  the integrand $\la \eta(z) \circ j , Y(z,u(z)) \rd_z u  \ra$ remains
  positive for all $z\in\Si$.
However, $Y$ is forced by (P1) to be constant along the $\Aut$-orbit
  through $z_0$, so that we need to use cutoff in $M$ near $u(z_0)$.
The latter can only be guaranteed if we have $u(\phi(z_0))\neq u(z_0)$
  for all $\phi(z_0)\neq z_0$; in other words, the $J$-holomorphic map $u$ needs
  to be injective along the orbit through $z_0$, and additionally,
  $z_0$ can not be a singular point of $u$.

On the other hand, we usually have unique continuation for the
  Cauchy-Riemann equation along $\Aut$-orbits, due to the invariance of
  $J$ along these $\Aut$-orbits.
For the dual linearized operator this means that for $(\rD_u\pbar)^*\eta
  =0$ and $\eta|_V\equiv 0$ on some open subset $V\subset\Si$ we obtain
  $\eta_{{\rm Aut}\cdot V}\equiv 0$ on the orbit of $V$.
Hence it suffices to have injectivity of $u$ and nonvanishing of $\rd u$
  somewhere along almost every $\Aut$-orbit in $\Si$.
The most important cases are the following.
\begin{itemize}
  \item 
  For pseudoholomorphic spheres with zero, one, or two fixed marked
    points, the automorphism group acts transitively on $\Si=S^2$, so
    that it suffices to find some $z_0\in S^2$ with $\rd_{z_0} u \neq 0$
    and $u^{-1}(u(z_0))=u(z_0)$.
  In fact, by \cite{McD,MDSa} the set of such ``injective points''
    is dense unless $u$ is multiply covered.
  This is equivalent to the existence of some nontrivial M\"obius
    transformation $\phi:S^2\to S^2$ for which $u\circ\phi=u$, which can
    be stated more elegantly by saying $u$ has a nontrivial isotropy
    group.
  \item
  For pseudoholomorphic disks with zero or one marked points on the
    boundary, it similarly suffices to have one ``injective point''.
  However, there now exist nowhere injective disks that are not multiply
    covered, i.e.\ have trivial isotropy group.
  An example is the ``lantern'': a disc mapping to $M=S^2$ with boundary
    on the equator that wraps two and a half times around the sphere.
  \item
  For Floer trajectories, i.e.\ pseudoholomorphic strips (disks with
    two marked points) or cylinders (spheres with two marked points,
    but with a Hamiltonian perturbation that breaks the $S^1$-symmetry),
    the automorphism group is $\R$.
  So it suffices to find for almost every $t_0\in [0,1]$ (or 
    $t_0\in S^1$) a point $s_0\in\R$ with $\rd_{(s_0,t_0)} u \neq 0$
    and $u(s,t_0)\neq u(s_0,t_0)$ for all $s\neq s_0$.
  In fact, unless the trajectory is constant (i.e.\  $\partial_s
    u \equiv 0$), the set of such points $(s_0,t_0)$ is dense
    by \cite{floer-hofer-salamon}.
  \end{itemize}
\end{remark}
}

Further injectivity requirements for the transversality of
  pseudholomorphic maps arise, for example in SFT, from invariance
  conditions for the almost complex structures on the target $M$.
Apart from such cases, transversality can be obtained by this geometric
  Sard-Smale method for any stable domain $\Si$.
(This excludes tori and spheres or disks with less than 3 marked points;
  where points in the interior of a disk count double.)
However, any bubbling in a space of pseudholomorphic curves (i.e.\
  blow-up of the gradient) leads to unstable sphere or disk components,
  so that this basic version of the geometric regularization approach is
  firmly restricted to cases in which bubbling can be a priori excluded --
  or at least the dimension of spaces of nowhere injective bubbles is
  controlled by underlying injective curves.
The first prominent case considered aspherical symplectic manifolds,
  in which Floer \cite{floer} excluded bubbles by their nonzero energy.
This argument has a direct generalization to monotone settings \cite{oh},
  where a proportionality between energy and Fredholm index allows one to
  exclude sphere or disk bubbling in moduli spaces of small dimension.
Finally, in semi-positive symplectic manifolds, the multiply covered
  spheres have to be localized on simple spheres, whose codimension in
  the moduli space is at least $2$, so that, for example, Gromov-Witten
  moduli spaces can be regularized to pseudo-cycles; see \cite{MDSa}.

Moving on to the compactness properties of spaces of pseudoholomorphic
  maps, the common singularity formations are ``bubbling'', where energy
  concentrates, ``breaking'', where energy escapes into noncompact ends
  of the domain or target, and the formation of ``nodes'' which might be
  allowed in the underlying space of Riemannian surfaces.
With the exception of sphere bubbles and interior nodes, these can
  be compactified along the lines of Steps 4 - 6, leading to boundaries
  and corners, and thus invariance of solution counts only up to some
  algebraic equivalence as in Step 7.
Sphere bubbling and interior nodes can also be treated analogously,
  although they give rise to interior points (or codimension $2$ points
  that do not contribute to the pseudo-cycle) of the compactified moduli
  space as follows.

\begin{enumlist}
  \item[\bf 5')] {\bf Gluing:}
  Due to an extra rotation parameter at the node, the gluing map (for
    a single node) is of the form $\tilde\oplus:  (R_0,\infty) \times
    S^1 \times \cM_p \tilde\times \cM_p  \hookrightarrow \cM_p$.
  \item[\bf 7')] {\bf Compactness:}
  By choosing a homeomorphism from $\bigl((R_0,\infty)\times S^1\bigr)
    \cup \{\infty\}$ to the open unit disk, one could construct a smooth
    manifold in which sphere bubbles (or interior nodes) are interior
    points.
  However, smooth compatibility of the gluing maps is generally hard to
    achieve, so that this technique is mostly used to deduce compactness
    up to codimension $2$ singularities.
  \item[\bf 8')] {\bf Invariance:}
  With the perturbed and compactified moduli spaces being closed
    (or pseudo-cycles), one obtains \emph{well defined} counts of
    solutions (or more generally one obtains a well defined integral
    over the solution set) by regularizing moduli spaces that involve
    an interpolating $1$-parameter family of perturbations; in turn,
    such spaces provide a cobordism between a pair moduli spaces
    obtained from different choices of perturbation, and such a cobrodism
    then guarantees counts (or integrals) are independent of initial
    perturbation.  
\end{enumlist}

Finally, let us mention two more special cases of the geometric
  regularization approach.
The simplest is the case of pseudoholomorphic curves of small genus
  with positive index in a four dimensional symplectic manifold, for
  which \emph{automatic transversality} guarantees sujrectivity of the
  linearized Cauchy-Riemann operator for \emph{every} choice of almost
  complex structure.
This approach has been used successfully in a variety applications;
  see \cite{G,HLS97, W}.

An example with more general perturbations is the construction of
  spherical Gromov-Witten invariants developed in \cite{CM}.
(This approach was also used in \cite{fabert} and recently generalized
  to the positive genus case in \cite{Gerst}; Ionel lays the foundations
  for a similar approach in \cite{Io}.)
Here the idea is to fix a Donaldson hypersurface in such a way that the
  marked points given by intersections with the hypersurface stabilize
  every pseudoholomorphic map in a given homology class.
Letting $\cB$ be a sufficiently small neighborhood of the
  pseudoholomorphic maps, one then obtains an $\Aut$-invariant map to a
  Deligne-Mumford space of marked Riemann surfaces.
One can then work with a space of perturbations $\cP$ that is given by
  families of almost complex structures over the Deligne-Mumford space.
In other words, the almost complex structure $J(u)$ is no longer defined
  pointwise, but may depend on the position of the intersections of $u$
  with the Donaldson hypersurface.
This approach then yields regularizations in the form of pseudo-cycles,
  unique up to rational cobordism, and hence rational Gromov-Witten
  invariants.

\subsection{The virtual approach}
\label{ss:virtual}

The analytic starting point of the ``virtual approach'' is the
  observation that the solution set of the Cauchy-Riemann operator
  restricted to a local slice of the ${\rm Aut}$-action (as in
  Remark~\ref{rmk:slice}) is homeomorphic to an open subset of the
  moduli space.
Since this is a Fredholm section, one can find a finite dimensional
  reduction; in other words, one can find a section of a finite
  dimensional bundle and a homeomorphism from its zero set to an open
  subset of the moduli space.
Alternatively, one could view this procedure as finding a finite
  dimensional obstruction bundle over an open subset of $\cB/\Aut$ that
  covers the cokernel of the linearized Cauchy-Riemann operators.
Both versions of this approach then aim to work in a finite dimensional
  category (either just for the fibers of the obstruction bundle or for
  both fibers and base in finite dimensional reductions) to associate a
  ``virtual fundamental class'' to the compactified moduli space $\CM$;
  for example, sometimes one aims to find  a \v{C}ech homology class
  $[\CM]_\kappa\in \check{H}(\CM;\Q)$ induced by a special type of
  Kuranishi atlas $\kappa$ on $\CM$ as in \cite{mcduff-wehrheim}.
We base this exposition on the latter, and consequently we do not
  explicitly discuss obstruction bundle techniques, which would proceed
  along similar lines.

The overall structure of the virtual approach reorders the basic
  ingredients of the geometric approach from {\it 1) transversality, 2)
  quotient, 3) gluing} to {\bf 1) quotient, 2) transversality, 3) gluing},
  and it aims to reduce the problem to a finite dimensional setting as
  quickly as possible.
A main feature of this approach is that it provides a natural setting
  for dealing with nonfree actions.
Let us only note here that this introduces an additional finite group
  action, or groupoid structure, in the second of the following steps,
  and it additionally requires equivariance in the further steps.

\begin{enumlist}
  \item 
  {\bf Compactness:}
  Construct the compactified moduli space $\CM$ as a compact (usually
    metrizable) topological space containing $\cM$ as well as $\cM
    \tilde{\times} \cM$, and which possibly contains higher fiber products.
  \item 
  {\bf Quotient (local):}
  View the uncompactified moduli space as subset $\cM\subset\cB/{\rm
    Aut}$ of the quotient space of maps as in the geometric approach,
    and for any $[u]\in\cM$ find a local slice.
  That is, find a Banach submanifold $\cB_H\subset\cB$ such that
    $\Aut\times\cB_H \to \cB$ is a homeomorphism to an open subset.
  Since $\Aut$ generally does not act differentiably on infinite
    dimensional spaces of maps like $\cB$, this requires a geometric
    construction as in Remark~\ref{rmk:slice}, for example.
  \item 
  {\bf Fredholm setup and almost Transversality (local):}
  Set up the PDE as a smooth Fredholm section $\s:\cB_H\to\cE|_{\cB_H}$
    of a Banach space bundle such that $\s^{-1}(0)$ is homeomorphic to
    an open neighborhood of the center, $[u]\in\cM$, of the local slice.
  From this, and a choice of finite dimensional {\bf obstruction bundle}
    $\widehat E\to \cB_H$ that covers the cokernels of the linearized PDE,
    construct a finite dimensional reduction, which by definition is 
    a smooth section $s: B \to E$ of a finite dimensional $E\to B$
    over a manifold $B$ such that $s^{-1}(0)$ is homeomorphic to a
    neighborhood of $[u]\in\cM$.  
  \item 
    {\bf Gluing (local):} Construct finite dimensional reductions
    for the higher strata of $\CM$ from a gluing construction.
  The standard gluing analysis does not provide smooth sections $s:B\to
    E$ in this case, but an appropriate notion of stratified smoothness
    should suffice.
  \item 
  {\bf Semi-local Transversality and Quotient compatibility (transition
    data):}
  Establish compatibility of the local finite dimensional reductions by
    forming direct sums of the obstruction bundles near overlaps in $\CM$.
  This requires one to refine the choice of obstruction bundles in
    steps 3 and 4 such that they are transverse on the overlaps.
  The direct sum construction also involves pullbacks of the obstruction
    bundles by an action of ${\rm Aut}$, due to the changing local slices
    in step 2.
  To ensure smoothness and differentiability of the pullback bundles,
    specific geometric constructions of the obstruction bundles are needed.
  \item 
  {\bf Kuranishi regularization (automatic):}
  A general abstract theory associates a virtual fundamental class
    $[\CM]^{\rm vir}$ to any covering of $\CM$ by finite dimensional
    reductions that are suitably compatible.
  Roughly speaking, the Kuranishi charts and transition data form
    categories $\Ti B, \Ti E$ and a functor $\Ti s :\Ti B \to \Ti E$
    so that $\CM$ is identified with the realization of the subcategory
    $|\Ti s^{-1}(0)|$ (which by definition is the subspace of objects
    at which the section vanishes, modulo the equivalence relation
    generated by the morphisms).
  The abstract theory then aims\footnote
    {
    As stated, there exists no such general result in the literature.
    All current approaches struggle with ensuring the Hausdorff
      and compactness properties of the zero set, so at best they find
      the required perturbations in a smaller category whose realization
      still contains $\CM$.
    }
    to provide a class of perturbation functors $\Ti p :\Ti B \to \Ti
    E$ such that $|(\Ti s+\Ti p)^{-1}(0)|$ inherits the structure of a
    compact manifold, and that up to some type of cobordism is
    independent of $p$.
  \item 
  {\bf Coherence:}
  If $\CM$ consists of several components and an identification of
    the boundary $\partial [\CM]_\cK$ with a fiber product $[\CM]_\cK
    \tilde{\times} [\CM]_\cK$ is desired, then steps 2 - 6 need to choose
    the local slices, obstruction bundles, and abstract perturbations
    ``coherently''; in other words they much be chosen to be compatible
    with the gluing maps.
    (We note that, potentially, these interwoven steps can be organized
      by a hierarchy of connected components of $\CM$.)
  \item 
  {\bf Invariance:}
  Prove that the algebraic structures arising from different choices
    in the previous steps, in particular the choice of local slices and
    obstruction bundles, are equivalent. This involves the construction
    of a virtual fundamental chain on $[0,1]\times\CM$ from local
    finite dimensional reductions which reduce to two given choices on
    $\{0\}\times\CM$ and $\{1\}\times\CM$.
  \end{enumlist}

At present, the applicability of the virtual approach to pseudholomorphic
  curve spaces is being revisited.
The recent \cite{mcduff-wehrheim} discusses a number of fundamental
  analytic and topological issues in \cite{FO,FOOO,LiT,LiuT} (one of
  which is discussed in Example~\ref{ex:trivial_splicing}), while itself
  only providing a theory for severely limited cases in which geometric
  methods are known to apply.
Our hope is that a nontrivial convex span of all these publications
  should lead to a theory that is not only solid but also understood by
  more people than just the respective authors.

Assuming that a functional theory for the abstract regularization step
  6 is established, the virtual approach does allow one to regularize
  more moduli problems, yet does not seem to eliminate repetitive work
  in the other steps.
In particular, any application to a new moduli problem still requires
  some new geometric insight to find appropriate local slices in step 2
  and obstruction bundles in step 3 that transform appropriately under
  the automorphism action; this is similar to finding a special set of
  perturbations in the geometric approach.
The Fredholm setup in step 3 is also somewhat more complicated than
  in the geometric approach, since the local slice condition must be
  incorporated.
Next, the gluing analysis in step 4 is exactly the same as that in
  the geometric approach, but the smoothness requirements on the finite
  dimensional reductions in fact require a more refined analysis than in
  some geometric regularizations, which merely construct a pseudocycle.
Moreover, some additional technical work is required to obtain the
  transversality of obstruction bundles needed in step 5.
Finally, coherence and invariance in steps 7 and 8 again require the
  same amount of work and sometimes nontrivial ideas as in the geometric
  approach.

{\small
\begin{remark}[{\bf Relation between Kuranishi atlases and polyfold
  Fredholm sections}] \rm
A description of $\bM$ as the zero set of a Fredholm section in a
  polyfold bundle, as outlined in Section~\ref{ss:polyfold} below,
  is expected to induce an equivalence class of Kuranishi-type atlases
  for $\bM$.
The rough idea is that this setting allows one to perform steps 1--5 of
  the virtual regularization scheme abstractly (or rather, they are
  already part of the polyfold setup).
This requires choosing obstruction spaces that locally cover the
  cokernel; with different choices yielding equivalent atlases in the
  sense of having a common refinement (for details see e.g.\
  \cite{mcduff-wehrheim}).

In the case of trivial isotropy, the stabilization construction for
  classical Fredholm sections can be obtained, see for example in
  \cite{cms}, and it yields a global finite dimensional reduction
  (also referred to as an atlas with one chart); in other words, $\bM$
  can be described as the zero set of a single section of a finite
  dimensional bundle.  
In the case of nontrivial isotropy, and under the additional assumption that at
  every solution there exists a choice of obstruction space on which the
  isotropy group of the given point acts trivially, the above stabilization
  construction generalizes, and it yields a global finite dimensional
  reduction to a section of an orbi-bundle.
To the best of the last author's understanding, this assumption coincides with
  the notion of ``semi-effective d-orbifold"  for which regularization is
  discussed in \cite{joyce}.
In this case, however, the regularization in step 6 can be quoted directly
  from Theorem~\ref{thm:model} or its generalization to orbibundles.
Similarly, the Euler class approach of \cite{Siebert} seems to go
  through under this additional assumption, and it yields a similar global
  finite dimensional reduction.

In the general case, \cite{Dingyu} proposes that finite dimensional
  reduction of Fredholm sections of polyfold bundles yields more general
  Kuranishi atlases with several charts.
This resulting class of Kuranishi atlases should have well controlled
  transition maps which allow for a converse construction, which would
  yeild an equivalence between polyfold Fredholm sections and certain
  Kuranishi atlases.
\end{remark}
}

\subsection{The polyfold approach} \label{ss:polyfold}

The polyfold approach, just like the geometric one, aims to associate to
  a compactified moduli space $\CM$, a smooth compact manifold $\CMp$,
  possibly with boundary $\partial\CMp = \cM' \tilde{\times} \cM'$,
  which is unique up to the appropriate notion of cobordism.
In order to achieve this, and eliminate a lot of the repetitive work
  in the applications, this approach fundamentally changes the basic
  order of ingredients from {\it 1) transversality, 2) quotient, 3) gluing}
  in the geometric approach and  {\it 1) quotient, 2) transversality, 3)
  gluing} in the virtual approach to the order {\bf 1) quotient, 2) gluing, 3)
  transversality}, and remains in an infinite dimensional setting
  until transversality is achieved.
The following eight steps provide an outline of the regularization
  procedure for a given moduli problem offered by the polyfold approach.
{\it [Additionally, in italics, we will compare each step to related
  constructions in the other approaches to demonstrate how significant
  amounts of technical work are automatized  in the polyfold approach.]}

\begin{enumlist}
  \item 
  {\bf Compactness:}
  Construct a (metrizable) topological space $\Ti\cB$ that contains
    the compactified moduli space $\CM$ as compact subset.
  Roughly speaking, $\Ti\cB$ can be obtained from the quotient space
    $\Ti\cB[0]:=\cB/\Aut$, which contains the moduli space $\cM$ of
    smooth (i.e.\ non-nodal or unbroken) solutions of the PDE, by adding
    strata $\Ti\cB[\ell]$ of singular maps (e.g.\ $\ell$-fold broken,
    or with $\ell$ nodes) that need not satisfy the PDE, in the same
    way as $\CM$ is obtained from $\cM$ by adding strata of singular
    solutions.\footnote
      {
      More precisely, the pregluing map of step 3 defines the
	neighborhoods of broken or nodal maps.
      } 
  These higher strata consist of large function spaces which will not,
    in general, solve the given PDE but will contain the compactification
    points of the moduli space, \(\overline{\mathcal{M}}\setminus
    \mathcal{\mathcal{M}}\).
  {\it [This is the same starting point as in the obstruction bundle
    version of the virtual approach.
  It is only slightly more complicated than topologizing the compactified
    moduli space $\CM$ in step 1 of the virtual approach and step 7 of the
    geometric approach.] }
  \item 
  {\bf Quotient (global):}
  Give $\Ti\cB[0]=\cB/\Aut$ a scale smooth structure as a ``scale Banach
    manifold'' by finding local slices as in Remark~\ref{rmk:slice}.
  That is, find Banach submanifolds $\cB_H\subset\cB$ such that
    $\Aut\times\cB_H \to \cB$ is a homeomorphism to an open subset,
    and check that the transition maps are scale smooth.
  Do the same with each singular stratum \(\Ti\cB[\ell]\), which is
    given by a fiber product of two or more copies of the main stratum,
    e.g.\ $\Ti\cB[1]\cong\Ti\cB[0] \tilde{\times}\Ti\cB[0]$.
  {\it [The local slices are the same as those required in step 2 of
    the virtual approach.
  Their existence and scale smoothness follow from triviality of
    isotropy groups (which we assume throughout) and similar basic analytic
    properties of the action as those used to establish
    step 4 of the geometric approach.] }

  \item 
  {\bf Pregluing:} 
  Give the main stratum $\Ti\cB$ a generalized smooth structure near
    the strata of singular maps.
  In order to construct charts centered at once broken or nodal map in
    $\Ti\cB[1]$, use a pregluing map of the form
    $\oplus: \cG^* \times \cU_0 \tilde{\times} \cU_1  \to  \Ti\cB[0]$
    for open sets $\cU_i\subset\Ti\cB[0]$ (realized as local slices
    $\cU_i\hookrightarrow\cB$).
  Here the space of gluing parameters is $\cG^*=(R_0,\infty)$ in the case
    of a broken or boundary nodal map, whereas $\cG^*=(R_0,\infty)\times
    S^1$ for the case of an interior node.
  In either case, the pregluing map is extended by $\{\infty\} \times
    \cU_0 \tilde{\times} \cU_1$ mapping to the corresponding broken or
    nodal maps in the singular stratum $\Ti\cB[1]\subset\Ti\cB$.
  We then give $\cG:=\cG^* \cup\{\infty\}$ a smooth structure by ``a
    choice of gluing profile,'' which is a choice of identification with
    an interval $[0,1)\cong \cG$ with $\{0\}\cong\{\infty\}$ in the
    boundary case, and an open disk with $\infty$ at the center in the
    interior case.

  To make up for the lack of injectivity of these pregluing maps, follow
    a ``gluing and antigluing'' procedure outlined in section~\ref{ss:ret},
    to form an sc-retract $\cR\subset \cG \times \cU_0 \tilde{\times}
    \cU_1$, on which the restriction of $\oplus$ is a homeomorphism to
    an open subset of $\Ti\cB$.
  Analogously, construct such M-polyfold charts near the higher strata
    $\Ti\cB[\ell]$ of multiply broken or nodal maps in $\Ti\cB$ from
    pregluing maps $\oplus: \cG^\ell \times \tilde{\times}^{i=\ell}_{i=0}
    \cU_i \to \Ti\cB$ on multiple fiber products of local slices.
  In order to obtain scale smooth transition maps between these charts
    as well as the local slice charts arising from step 2, the safe choice
    is an exponential gluing profile as in \eqref{expprofile}.
  {\it [This is a mild extension of the pregluing construction that
    provides the basis for an intricate Newton iteration in the gluing
    analysis of step 5 (or 5') in the geometric approach and 
    step 4 in the virtual approach.
  The novelty is in the interpretation as chart maps.
  The construction of these charts and scale smoothness of transition
    maps should usually be obtained by combining basic local building
    blocks\footnote
    {
    At present, only the building blocks for smooth domains and interior
      nodes are readily available in \cite{HWZI_applications}.
    Work on the cases of breaking, Lagrangian boundary problems, and
      boundary nodes is in progress and discussed below.
    } 
    in the literature with a Deligne-Mumford theory for the space of
    underlying domains.] }
   
  \item 
  {\bf Fredholm setup:}
  After gathering the compatible charts constructed in steps 2 and 3
    to an M-polyfold structure on $\Ti{\cB}$, analogously construct an
    M-polyfold bundle $\Ti{\cE} \to \Ti{\cB}$ such that the PDE (e.g.\
    the gradient flow or Cauchy-Riemann operator) forms a section $\s:
    \Ti{\cB} \to \Ti{\cE}$ with $\s^{-1}(0)=\CM$.
  Check that the section $\s$ is a scale smooth polyfold Fredholm
    section.
  {\it [The bundle $\Ti\cE$ could be constructed in one stroke with the
    ambient space $\Ti\cB$ by adding fibers that are essentially given
    by the requirement of the PDE forming a section.
  This bundle as well as the regularity and Fredholm property of the
    section should again usually be obtained from patching together local
    building blocks for which Fredholm properties are established in
    the literature.
  For regular domains (smooth, connected Riemann surfaces), the Fredholm
    property is essentially the same as in step 1 of the geometric
    approach and step 3 of the virtual approach.
  For nodal or broken domains, the polyfold Fredholm property
    formalizes part of the gluing analysis, namely it essentially follows
    from the quadratic estimates that are required in the gluing analysis
    of the other approaches.] }
  \item 
  {\bf Transversality (automatic):}
  At this point the general transversality and implicit function
    theorem for M-polyfolds provides a class of perturbations
    $p:\Ti{\cB} \to \Ti{\cE}$ with the property that
    $\CM_p:=(\s+p)^{-1}(0)\subset\Ti{\cB}$ is a smooth {\bf finite
    dimensional} submanifold with boundary and corners, and for any
    other choice $p'$ in this class there is a suitable cobordism
    between $\CM_{p'}$ and $\CM_p$.
  The interior / boundary / corners of the perturbed moduli space $\CM_p$
    are given by its intersection with the interior $\partial^0\Ti\cB$
    / boundary $\partial^1\Ti\cB$ / corners $\partial^{k\ge 2}\Ti\cB$
    of the ambient space $\Ti\cB$.
  If there are no interior nodes, then each breaking or boundary node
    contributes $1$ to the corner index $k$; in other words, the $k$-th
    corner stratum is given by the fiber products $\partial^k\Ti\cB =
    \Ti\cB[k] \cong \ti\times^k \bigl(\qu\cB\Aut\bigr)$.
  If all nodes are interior, then $\Ti\cB$ has no boundary or corner
    strata, since the gluing parameters $S^1\times(R_0,\infty)$ are
    compactified to an open disk; here a gluing parameter equal to
    $\infty$ corresponds to a nodal map, which is still an interior
    point.
  In the case of mixed types of breakings and nodes, only those with
    gluing parameters $(R_0,\infty)$ (not those with an extra $S^1$ factor)
    affect the boundary and corner stratification (i.e.\ contribute to
    the corner index $k$).
  {\it [Contrary to step 2 of the geometric and steps 3 and 5 of the
    virtual approach, no special geometric class of perturbations or a
    priori transversality of obstruction bundles is required for this
    entirely abstract perturbation scheme.] }
  \item 
  {\bf Coherence (mostly automatic):}
  If the regularized moduli space is expected to have boundary given by
    fiber products of its connected components, then the corresponding
    coherent perturbations can be obtained from an extension of the
    polyfold transversality theorem to ``polyfold Fredholm sections with
    operations'' as outlined in \cite{hwzbook}.
  In this case the expected boundary stratification is reflected
    in the fact that the boundary of the M-polyfold $\Ti\cB$ can be
    identified with a fiber product $\partial^0\Ti{\cB} \tilde{\times}
    \partial^0\Ti{\cB}\cong \partial^1\Ti\cB$ of its interior.
  An ``operation'' is essentially a continuous extension of this
    identification to a (not necessarily injective or single valued)
    map $\Ti{\cB} \tilde{\times} \Ti{\cB} \to \Ti{\cB}\setminus
    \partial^0\Ti{\cB}=:\partial\Ti\cB$ with which the section $\s$
    is compatible -- roughly $\s|_{\partial\Ti\cB} = \s \tilde{\times} \s$.
  If one can now establish combinatorial properties, essentially
    amounting to a prime decomposition, for the operation on the level of
    connected components $\pi_0(\Ti{\cB}) \tilde{\times} \pi_0(\Ti{\cB})
    \to \pi_0(\Ti{\cB})$, then a refined abstract construction of the
    perturbations in step 5 yields a class of transverse perturbations that
    are additionally compatible with the operation on $\Ti{\cB}$.\footnote
      {
      This simple formulation holds in the absence of ``diagonal relators''
	-- connected components of $\Ti\cB$ that can be glued to themselves.
      Such ``self-gluing'' does occur in several instances of e.g.\
	general SFT.
      It can be dealt with by allowing a more general transversality to
	the boundary strata which still yields smooth perturbed moduli spaces
	with boundary and corners.
      However, it no longer ensures that the corner stratification is induced
	from the ambient one -- thus e.g.\ allowing boundaries of the moduli
	space to lie in corners of the ambient M-polyfold.
      The counts of such moduli spaces then yield more involved algebraic
	structures than the ``master equation'' mentioned here.
      }
  As direct consequence, the boundary (not including corners)
    $\partial^1\CM_p= \s|_{\partial^1\Ti{\cB}}^{-1}(0)
    = \s|_{\partial^0\Ti{\cB}}^{-1}(0) \tilde{\times}
    \s|_{\partial^0\Ti{\cB}}^{-1}(0)  = \cM_p \tilde{\times} \cM_p$
    is given by the fiber product of the interior.
  Algebraic structures induced by such perturbed moduli spaces then
    automatically satisfy a ``master equation'' of the type $\partial
    \underline m_p = \underline m_p \tilde{\times} \underline m_p$.
  {\it [This abstract coherent perturbation scheme is essentially just a
    formalization of iterative schemes that exist in various applications.
  The polyfold approach allows one to formulate this scheme abstractly
    since pullback to fiber products and extension to the interior
    automatically provides further abstract scale smooth perturbations,
    whereas in the geometric and virtual approach some care is required
    to preserve a specific geometric type of perturbations in such
    constructions.] }
  \item 
  {\bf Invariance (partially automatic):}
  The algebraic structures arising from different choices of
    perturbations in step 5 are automatically equivalent due to the
    cobordisms between different perturbations.
  Invariance for different choices in the setup of $\s$ still has to
    be proven independently, however this is accomplished via a similar
    M-polyfold setup for a family of sections.
  In particular, the variation of the almost complex structure has to
    be treated this way, since it does not fit into the class of abstract
    perturbations in step 5.
  {\it [Though formally similar to the essential invariance questions in
    the geometric and virtual approach, the polyfold approach has several
    readily available tools to obtain the required cobordisms with much
    less effort than the corresponding steps 8 of the other approaches.
  These are discussed further below.  ] }
  \end{enumlist}

The last step of this road map highlights two particular strengths of
  the polyfold approach.
Firstly, independence from the choice of perturbations is simply
  automatic, whereas it needs to be proven separately in the geometric
  approach.
Compared with the virtual approach, the abstract regularization step
  in the latter also provides some automatic invariance -- though at best
  for a fixed cobordism class of Kuranishi structures.
Here it is worth noting that the ambient M-polyfold for a given moduli
  problem can essentially be constructed naturally; in other words, the
  construction only depends on a few explicit choices such as the Sobolev
  completion and a ``gluing profile'' $[0,r_0)\cong (R_0,\infty]$.\footnote
  {
  At present, all known polyfold constructions use the same ``exponential
    gluing profile'' $v\mapsto e^{1/v}-e$, and the only choice in
    the $W^{2,3+k}_{\delta_k}$ Sobolev completions is a sequence
    $(\delta_k)_{k\in\N_0}$ of exponential decay parameters as in
    Lemma~\ref{ex:sobolev}.
  The question of comparing invariants resulting from different choices
    of such global data has not been addressed at this time.
  However, note that it could be reduced to cases in which there is a
    smooth injection from one into the other polyfold bundle.
  We then expect that this embedding could be used to pull back a large
    class of admissible perturbations (e.g.\ those supported away from
    nodes) that should suffice for achieving transversality and thus
    identifying the invariants.
  }
Stated differently, M-polyfold charts that arise from different choices
  of local slices or local coordinates in the pregluing are compatible.

On the other hand, a Kuranishi structure a priori depends more
  substantially on the inexplicit choice of local slices and obstruction
  bundles, so the virtual approach requires a nontrivial proof of cobordism
  between the Kuranishi structures arising from different sets of choices.

Secondly, the polyfold approach even provides a framework for proving
  invariance under further variations of the PDE.
Namely, if this variation can be described as scale smooth family
  of polyfold Fredholm sections $(\s_\l)_{\l\in[0,1]}$ of a fixed
  M-polyfold bundle $\Ti\cE\to\Ti\cB$, then $[0,1]\times\Ti\cB\to\Ti\cE$,
  $(\l,b)\mapsto \s_\l(b)$ is a polyfold Fredholm section whose
  abstractly given transverse perturbations provide cobordisms between
  the regularizations for $\l=0$ and $\l=1$.

Finally, the greatest benefit of polyfold theory is its ability to
  provide regularizations of a wide variety of moduli problems based on
  a relatively small amount of technical work that moreover is easily
  transferrable to related moduli problems.
The presently developing applications are all closely related to
  pseudoholomorphic curves, but further applications to gauge theoretic
  elliptic PDEs are easily imaginable.
For the moment, we restrict our attention to pseudoholomorphic curve
  moduli problems, and we briefly list those theories for which a polyfold
  framework has been developed, is under development, is expected to result
  from the same techniques, or is hoped for as nontrivial extension of
  existing techniques.

\begin{description}[leftmargin=1.5em, labelindent=0em, itemsep=0em]
  \item 
  [Morse theory] 
  An example in \cite{HWZ_lectures} sketches out the construction of
    a Fredholm section in an M-polyfold bundle whose zero set is the
    moduli spaces of (unbroken, broken, and multiply broken) gradient
    trajectories in a closed Riemannian manifold with Morse function.
  A more thorough construction is being developed in \cite{AW}.
  A description of Morse trajectory spaces as moduli spaces of solutions
    of a PDE (though really an ODE) and a geometric regularization of low
    index moduli spaces from this point of view is available in textbook
    format in \cite{schwarz:Morse}.
  \item 
  [Gromov-Witten theory]  
  Moduli spaces of closed (possibly nodal) pseudoholomorphic curves
    of arbitrary genus in any closed symplectic manifold are described
    as the zero set of a polyfold Fredholm section (in an orbifold type
    bundle modeled on M-polyfolds) in \cite{HWZI_applications}.
  Introductory material on genus zero Gromov-Witten moduli spaces and
    a geometric regularization in semipositive symplectic manifolds  is
    available in textbook format in \cite{MDSa}.
  \item 
  [Symplectic Field Theory]
  The primary motivation for the development of polyfold theory was
    the regularization issue for moduli spaces of pseudoholomorphic
    buildings in non-compact symplectic cobordims -- specifically curves
    in cylindrically-ended cobordisms between manifolds with non-degenerate
    stable Hamiltonian structures.
  These SFT moduli spaces were introduced in \cite{EGH}, and their
    description as the zero set of a polyfold Fredholm section is expected
    as the next publication in the program of Hofer-Wysocki-Zehnder
    \cite{HWZII_applications}.
  \item 
  [Hamiltonian Floer theory] 
  Moduli spaces of (possibly broken) Floer trajectories between periodic
    orbits of a nondegenerate Hamiltonian vector field in any closed
    symplectic manifold $M$ are special cases of SFT moduli spaces for
    the cobordism $\R\times S^1\times M$.
  Thus a description as the zero set of a Fredholm section in a polyfold
    bundle will arise from \cite{HWZII_applications}.
  Partial results on the Fredholm property near broken trajectories
    are available in \cite{w:fred}.
  This polyfold setup will specialize to a Fredholm section in an
    M-polyfold bundle if sphere bubbling can be excluded a priori.
  Hamiltonian Floer theory was first developed by Floer \cite{floer},
    and further introductory material can be found in \cite{Salamon}.
  \item 
  [Arnold conjecture via ${\mathbf S^1}$-equivariance] 
  Floer proved the Arnold conjecture for monotone symplectic manifolds in
    \cite{floer:Arnold} by constructing a moduli space cobordism between
    Hamiltonian Floer moduli spaces and Morse trajectory spaces.
  This proof was generalized to a variety of settings, with the main
    obstacle being the need for an $S^1$-equivariant regularization.
  In the polyfold framework, this approach to the Arnold conjecture would
    require a setup in which a transverse perturbation can be pulled back
    from a quotient by a scale smooth $S^1$-action.
  The analogous finite dimensional quotient theorems are expected to
    generalize to actions on polyfolds under suitable analytic conditions.
  A first rigorous study in a Morse theoretic model case is intended
    to follow after \cite{AW}.
  \item 
  [PSS morphism] 
  An alternative approach to proving the Arnold conjecture was proposed
    in \cite{PSS} based on a moduli space of pseudoholomorphic spheres with
    one Hamiltonian end and one marked point coupled to a Morse flow line.
  The direct approach again required an $S^1$-equivariant regularization
    and was not published in technical detail.
  However, this approach can be algebraically refined so that the
    regularization issues reduce to obtaining a polyfold Fredholm
    description for trees of pseudoholomorphic spheres with one or two
    Hamiltonian ends; these are again special cases of SFT moduli spaces
    (see \cite{AFW}  for further details).
  Given a polyfold setup for the latter and a manifold with boundary
    and corner structure on compactified spaces of finite or half infinite
    Morse trajectories from \cite{w:Morse}, a fiber product construction
    provides a polyfold Fredholm description for compactifications of all
    relevant PSS moduli spaces; these spaces involve a finite or half
    infinite Morse trajectory coupled to one or two trees of spheres
    with a Hamiltonian end.
  \item 
  [Pseudoholomorphic disks] 
  Moduli spaces of pseudoholmorphic disks with Lagrangian boundary
    condition can be compactified in different ways.
  One of the first such compactifications, involving nodal
    disks, was introduced in \cite{FOOO} with the aim of constructing an
    $A_\infty$-algebra on a certain completion of singular chains on the
    Lagrangian.
  Closely related moduli spaces, which in addition allow for Morse
    trajectories between the disks, was introduced in \cite{F,FOh,CL}
    and further developed in \cite{w:msri-talk} with the aim
    of constructing an $A_\infty$-algebra on the Morse complex of the
    Lagrangian.
  The corresponding building blocks of pseudoholomorphic curves with
    Lagrangian boundary conditions and boundary marked points connected
    by Morse trajectories are in the process of being described by an
    M-polyfold Fredholm section in \cite{jiayong-w}.
  Under the assumption of pseudoholomorphic spheres being a priori
    excluded, this should yield an $A_\infty$-algebra over $\Z$ or $\Z_2$.
  In the presence of pseudoholomorphic spheres these building blocks
    are expected to combine with the existing building blocks of
    pseudoholomorphic curves with interior nodes via a general patching
    technique that is being developed in \cite{HWZII_applications}.
  The combined Fredholm setup is expected to yield an $A_\infty$-algebra
    over $\Q$.
  \item 
  [Lagrangian Floer theory and Fukaya category] 
  By adding building blocks of striplike ends with Lagrangian boundary
    condition, one should obtain a polyfold setup for Lagrangian Floer
    theory, which was introduced in \cite{floer:Lagrangian}.
  By lifting this setup to domains given by more involved
    Deligne-Mumford-type spaces of punctured disks, one should moreover
    obtain a polyfold setup allowing one to define Fukaya categories as
    introduced in \cite{FOOO,Seidel}.
  \item 
  [Relative SFT] 
  Finally, the previous moduli spaces can be generalized from domains
    with striplike ends and Lagrangian boundary conditions to SFT-type
    holomorphic curves with boundary in cylindrically-ended symplectic
    cobordisms and boundary values on Lagrangian cobordisms between
    Legendrian submanifolds.
  While the general algebraic structure of such theories is unclear, the
    moduli spaces should have a relatively straight forward description
    as the zero sets of polyfold Fredholm sections, and the boundary
    stratifications are expected to govern the induced algebra.
  A special case of this setup would provide a polyfold framework for
    {\bf Legendrian contact homology}, which originated in \cite{Chekanov}
    and was generalized in \cite{EES}.
  \item 
  [Morse-Bott degeneracies] 
  The scope of \cite{HWZII_applications} is to provide a regularization
    of the moduli space of non-compact curves in cylindrically-ended
    cobordisms such as \(\mathbb{R}\times V\) where \((V, \xi = {\rm
    ker}\lambda) \) is a contact manifold.
  A crucial requirement here is a choice of contact form $\lambda$
    for which all Reeb orbits are non-degenerate.
  Similar nondegeneracy conditions are necessary in all previously
    mentioned moduli space setups.
  Though technically much more involved, it seems possible that analysis
    in \cite{HWZII_applications} may generalize and be applicable to the
    case in which the orbits are Morse-Bott degenerate.
  Morse-Bott contact homology would be a special case of such a theory;
    for introductory material see  \cite{Bourgeois}.
  \item 
  [Pseudoholomorphic quilts]
  The building blocks for Gromov-Witten, Lagrangian Floer theory,
    and pseudoholomorphic disks should also combine to give a polyfold
    setup for the moduli spaces of pseudoholomorphic quilts introduced
    in \cite{WW}. 
  Indeed, this is expected since seam conditions are locally equivalent
    to Lagrangian boundary conditions in a product.
  The novel figure eight bubble, however, has no description in terms of
    previous Cauchy-Riemann-type PDE's, since it involves tangential seams.
  The basic analysis towards a polyfold Fredholm description was established in \cite{Nate}.
  \end{description}

\part{Presenting Palatable Polyfolds}

In this mathematical part, we present the core definitions of polyfold
  theory in a streamlined fashion so that we may state a precise version
  of the abstract regularization result as quickly as possible.
For each of the new key concepts we present examples of their application
  to Morse trajectory spaces as in Example~\ref{ex:Morse}.

\section{Scale Calculus}\label{sec:scCat}

\subsection{Scale Topology and Scale Banach spaces}\label{sec:scTop}
We begin by introducing sc-topological spaces.
While this notion is not explicitly defined by HWZ, it is implicitly
  present in much of the theory.
(For instance, sc-Banach spaces, relatively open subsets in partial
  quadrants, sc-smooth retracts, (M-)polyfolds, and strong polyfold
  bundles all carry sc-topologies.)

\begin{definition}\label{def:scTopology}
Let $X$ be a metrizable topological space.
An {sc-topology} on $X$ consists of a sequence of subsets $(X_k\subset
  X)_{k\in\N_0}$, each equipped with a metrizable topology, such that
  the following hold.
  \begin{enumerate}
    \item $X=X_{0}$ as topological spaces.
    \item For each $k> j$ there is an inclusion of sets $X_k\subset X_j$, and
      the inclusion map $X_k\to X_j$ is continuous with respect to the $X_k$
      and $X_j$ topologies.
    \end{enumerate}
We will refer to $(X_k)_{k\in\N_0}$, or \(\mathbb{X}\), or sometimes
  simply to $X$, as an {\bf sc-topological space}.

An sc-topology $(X_k)_{k\in\N_0}$ is called {\bf dense}  if it has the
  following property.
\begin{enumerate}
  \item[(iii)] 
    The subset $X_\infty:=\bigcap_{k\in\N_0}X_k$ is dense in each $X_j$.
  \end{enumerate}

An sc-topology $(X_k)_{k\in\N_0}$ is called {\bf precompact} if it has
  the following property.
\begin{enumerate}
  \item[(iv)] 
    For each $p\in X_k$ and $j<k$, there exists a neighborhood
    $O_{jk}\subset X_k$ of \(p\), whose closure in $X_{j}$ is compact.
  \end{enumerate}
\end{definition}

Note that an sc-topological space $X$ is related to a multitude of
  topologies -- namely for every $k\in\N_0$ the $X_k$-topology is defined
  on the subset $X_k\subset X$, or any of its subsets.
So by standard topological terms, such as openness or compactness,
  we will always refer to the $X_0$ topology, which makes sense for
  all subsets of $X$ -- unless a different ambient space $X_k$ and its
  topology are specified.

\begin{remark} \label{rmk:sc-top} \rm \hspace{2mm} \\ \vspace{-5mm}

\begin{enumerate}
  \item
  Any topological space $X$ carries the {\bf trivial sc-topology}
    $(X_k=X)_{k\in\N_0}$.
  This is a dense sc-topology and satisfies the compactness property
    if and only if $X$ is locally compact\footnote
    {
      Recall, \(X\) is locally compact if for each point \(p\in X\)
	there exists a neighborhood of \(p\) which is compact.
    }.
  \item
  If $\mathbb{X}=(X_k)_{k\in\N_0}$ is an sc-topological space and
    $Y\subset X_0$ a subset, then $Y$ inherits an sc-topology $(Y_k:=
    Y\cap X_k)_{k\in\N_0}$.
  In general, if \(\mathbb{X}\) is dense and precompact, then
    \(\mathbb{Y}:=(Y_k)_{k\in\mathbb{N}_0}\) need not inherit either of
    these properties.
  However, open subsets $Y\subset X$ do inherit density
    and precompactness from \(\mathbb{X}\) by Lemma
    \ref{lem:sc_topOnOpenSubset} below.
  \end{enumerate}
\end{remark}

\begin{example}\label{ex:Ck} \rm
The collection of {\bf $\mathbf k$ times continuously differentiable
  functions} on the line, all of whose derivatives are bounded, forms an
  sc-topological space $(X_k:=\cC^k(\R, \R))_{k\in\N_0}$, where each $X_k$
  is equipped with the topology induced by the $\cC^k$-norm.
It satisfies the density axiom since $X_\infty=\cC^\infty(\R, \R)$.
However, it does not satisfy the precompactness property, due to the
  noncompactness of the domain $\R$.
Indeed, if \(f\in \mathcal{C}^\infty(\R,\R)\) has compact support,
  then the sequence \((f_n(\cdot):=f(\cdot +n) )_{n\in \mathbb{N}}\) is
  bounded on every scale, but does not contain a convergent subsequence
  on any scale.
Hence for $g_0\in X_k$ any $X_k$-neighborhood $\{g \,|\,
  \|g-g_0\|_{\cC^k}\leq \eps\}$ still contains a sequence
  $g_n=g_0+\frac{\eps}{\|f\|_{\cC^k}} f_n$ that has no $X_j$-convergent
  subsequence for $j<k$.

If we use the compact domain $S^1=\R/\mathbb{Z}$, the sc-topological
  space $(X_k:=C^k(S^1, \R))_{k\in\N_0}$
  is dense and satisfies the precompactness property by the
  Arzel\`{a}-Ascoli Theorem.
Due to its linear structure, this is also the first example of an
  sc-Banach space as discussed in Section \ref{ss:sc} and rigorously
  defined below.
\end{example}

\begin{lemma}\label{lem:sc_topOnOpenSubset}
Let \(\mathbb{X}=(X_k)_{k\in \mathbb{N}}\) be a dense, precompact
  sc-topological space.
Let \(Y\subset X_0\) be an open subset.
Then for \(Y_k:=Y\cap X_k\), with the relative topology induced by $X_k$,
  the collection \((Y_k)_{k\in \mathbb{N}}\) forms a dense, precompact
  sc-topological space.
\end{lemma}
\begin{proof}
The axioms for the sc-topology $\mathbb{X}$ transfer directly to
  \((Y_k)_{k\in \mathbb{N}}\), so it remains to verify the density and
  precompact conditions.
For that purpose first note that $Y_j\subset X_j$ is open for all
  $j\in\N$, since it is the preimage of the open set $Y\subset X_0$ under
  the continuous inclusion $X_j\to X_0$.

Density of $Y_\infty = X_\infty \cap Y$ in a fixed $Y_j$ then follows
  from the density of $X_\infty\subset X_j$, since any $X_j$-convergent
  sequence $X_\infty \ni x_n \to y\in Y_j$ has its tail contained in the
  open subset $Y_j$, so that the tail is a $Y_j$-convergent sequence in
  $X_\infty\cap Y=Y_\infty$.

To prove precompactness of $\mathbb Y$, we fix \(j<k\) and \(p\in Y_k\).
Then precompactness of $\mathbb{X}$ provides a neighborhood
  $O_{jk}\subset X_k$ of $p$, whose closure in $X_j$ is compact.
On the other hand, $p\in X_j$ has a closed neighborhood basis by
  metrizability of the $X_j$-topology.
In particular, we can find a closed $X_j$-neighborhood $B_j$ of $p$
  that is contained $B_j\subset Y_j$ in the open subset $Y_j\subset X_j$.
Since the inclusion $X_k\to X_j$ is continuous, the preimage $B_j\cap
  X_k$ is also a neighborhood of $p\in X_k$.
Now $B_j \cap O_{jk}\subset Y \cap X_k = Y_k$ is the required $X_k$ --
  and hence $Y_k$ -- neighborhood of $p$.
Indeed, it is an intersection of the $X_k$-neighborhoods $B_j\cap X_k$
  and $O_{jk}$.
Its $Y_j$-closure $B_j\cap {\rm cl}_{X_j}(O_{jk})$ is compact since
  it is a closed subset of the compact subset ${\rm cl}_{X_j}(O_{jk})$
  of a Hausdorff space.
\end{proof}

After this gentle introduction to the basic idea of 'scales' providing
  different topologies on dense subsets of the same space, we introduce
  the ambient spaces of scale calculus, which have a linear structure as
  well as a dense precompact sc-topology.

\begin{definition}\label{def:scBanachSpace}
An {\bf sc-Banach space} (sc-Hilbert\footnote
  {
  We will develop all of scale calculus in the general setting of scale
    Banach spaces.
  The regularization theorems (c.f. Theorems
    \ref{thm:PolyfoldRegularization2} and
    \ref{thm:PolyfoldRegularization3}) as stated below will require all
    scale structures to be sc-Hilbert spaces, since this guarantees the
    existence of smooth cutoff functions.
  Note however that it should be sufficient to have scale-smooth cutoff
    functions, which may well exist on the sc-Banach spaces arising from
    Sobolev spaces with $p\ne 2$ introduced in Example~\ref{ex:sc-Sob}
    below (while the existence of classically smooth cutoff functions is
    a highly nontrivial question).
  } 
  space) \(\mathbb{E}\) consists of a Banach (Hilbert) space \(E\),
  together with a {\bf linear scale structure}.
The latter is a sequence of linear subspaces \(E=E_0\supset E_1\supset
  E_2 \supset \cdots\), each equipped with a Banach norm \(\|\cdot\|_k\)
  (Hilbert inner product \(\langle\cdot, \cdot \rangle_k\)), so that
  the induced sequence of topological spaces forms a dense precompact
  sc-topology.
\end{definition}

In the context of scale manifolds and M-polyfolds, we will also use
  the notion of {\bf scale smooth structures}, which are given by
  local models in scale Banach spaces; in other words, the models
  only \emph{locally} have a linear scale structure.
We will usually refer to both as {\bf scale structures} with the precise
  meaning clarified by the context.

\begin{lemma}\label{lem:CompactInclusions}
  Let \(\mathbb{E}\) be a sc-Banach space.
  Then for each \(j<k\) the linear inclusions \(E_{k}\to E_j\) are
    compact (and hence bounded).
\end{lemma}
\begin{proof}
First, since the \((E_k)_{k\in \mathbb{N}}\) form an sc-topology,
  the inclusion \(E_{k}\to E_j\) for each \(j<k\) is continuous, and
  hence bounded.
Next, the precompactness condition implies that there exists an open
  neighborhood \(O_{jk}\subset E_k\) of \(0\) which has compact closure
  in \(E_j\).
Thus we find $\eps>0$ so that \(\{x\in E_k: \|x\|_{k}<\eps\}\) has
  compact closure in \(E_j\).
By rescaling, this proves that any $E_k$-bounded subset has compact
  closure in $E_j$; in other words the inclusion  \(E_{k}\to E_j\)
  is compact.
\end{proof}

\begin{remark} \label{rmk:sc-Banach} \rm \hspace{2mm} \\ \vspace{-5mm}
\begin{enumerate}
  \item
  There exists a natural product $\E\times\F$ of sc-Banach spaces given by
    the scale structure \mbox{$(E\times F)_k:=E_k\times F_k$.}
  The analogous product for sc-topologies preserves density as well
    as precompactness.
  \item 
  An sc-Banach space induces an sc-topology on the
    space itself and on any of its open subsets, which is both dense
    and precompact.
  \item
  Any scale $E_j$ of an sc-Banach space $(E_k)_{k\in\N_0}$ inherits an
    sc-structure $(E_{j,k}:= E_{j+k})_{k\in\N_0}$.
  This is not the sc-topology induced on the subset $E_j\subset E_0$,
    but a new (dense, precompact) sc-topology on a dense subset, obtained
    by a shift which ensures precompactness.
  \end{enumerate}
\end{remark}

\begin{example} \label{ex:trivial} \rm
Any finite dimensional Banach space $E$ carries the {\bf trivial
  sc-structure} $(E_k=E)_{k\in\N_0}$.
Due to the density requirement for $E_{k+1}\subset E_k$, there are no
  nontrivial sc-structures on finite dimensional spaces.
Moreover, the compactness requirement (ii) implies that any sc-Banach
  space with $E_{k+1}=E_k$ must be locally compact and therefore finite
  dimensional.
For $n\in\N$ we will denote by $\R^n$ and $\C^n$ the real and complex
  {\bf Euclidean space} with standard norm and trivial sc-structure.
\end{example}

The moduli spaces of holomorphic curves, to which we wish to apply
  polyfold theory, usually work with domains that are either compact or
  have strip-like or cylindrical ends, which are conformally equivalent
  to $[0,1]\times\R^+$ or $S^1\times\R^+$ as appropriate.
The following are the prototypical examples for sc-Banach spaces (and
  sc-Hilbert spaces in case $p=2$) of maps on such domains.

\begin{example} \rm  \label{ex:sc-Sob}
Let $\Sigma$ be a compact Riemannian manifold, $\ell,n\in\N_0$, and
  $1\leq p <\infty$.
Then the {\bf Sobolev space} $W^{\ell,p}(\Sigma,\R^n)$ can be equipped
  with an sc-structure
  $$
    \bigl( E_k = W^{\ell+k,p}(\Sigma,\R^n) \bigr)_{k\in\N_0} .
  $$
Here the Sobolev spaces are defined as
  $$
    W^{m,p}(\Sigma,\R^n)
    := \bigl\{ u:\Sigma\to\R^n \st  |u|, |\rD u|, \ldots, |\rD^m u|
    \in L^p(\Sigma) \bigr\}
  $$
  with the norm $\| u \|_{W^{m,p}} = \bigl( \int_\Sigma |u|^p + |\rD
  u|^p + \ldots + |\rD^m u|^p \bigr)^{\frac 1p}$, where $\rD^m u$
  is the tensor denoting the $m$-th differential of the map \(u\). 
\end{example}

\begin{lemma} \label{ex:sobolev}
Let $n\in\N$, $\ell\in\N_0$, $1\leq p <\infty$, and $\delta_0\in\R$.
Then the {\bf weighted Sobolev space} $W^{\ell,p}_{\delta_0}(\R,\R^n)$
  can be equipped with sc-structures
  $$
    \bigl( E_k = W^{\ell+k,p}_{\delta_k}(\R,\R^n) \bigr)_{k\in\N_0}
  $$
  for any weight sequence $\ud=(\delta_k)_{k\in\N_0}$ with
  $k>j\Rightarrow\delta_k>\delta_j$.
Here
  $$
    W^{m,p}_{\delta}(\R,\R^n)
    := \bigl\{ u:\R\to\R^n \st s\mapsto e^{\delta s \beta(s)} u(s)
    \in W^{m,p} \bigr\}
  $$
  is the Sobolev space of weight $\delta\in\R$ given by the norm
  $ \| u \|_{W^{m,p}_{\delta}} = \| e^{\delta s \beta} u \|_{W^{m,p}} $,
  where $\beta\in\cC^\infty(\R,[-1,1])$ is a symmetric cutoff function
  with $\beta(-s)=-\beta(s)$, $\beta|_{\{s\geq 0\}}\geq 0$, and
  $\beta|_{\{s\geq 1\}}\equiv 1$.
(Different choices of $\beta$ yield the same space with equivalent
  norms.)
\end{lemma}
\begin{proof}
The inclusion $E_k=W^{\ell+ k,p}_{\delta_k}(\R,\R^n) \subset
  W^{\ell+j,p}_{\delta_m}(\R,\R^n)=E_m$ for $k>j$ exists since $e^{\delta_k
  s \beta} \geq e^{\delta_j s \beta}$.
It is compact since the restriction $W^{\ell+k,p}_{\delta_k}(\R,\R^n)
  \to W^{\ell+j,p}_{\delta_k}([-R,R],\R^n)$ is a compact Sobolev imbedding
  for any finite $R\geq 1$ (due to the loss of derivatives $k>j$, see
  \cite{Adams}) and the restriction $W^{\ell+k,p}_{\delta_k}(\R,\R^n)
  \to W^{\ell+k,p}_{\delta_j}((\R\setminus[-R,R]),\R^n)$ converges to $0$
  in the operator norm as $R\to\infty$ (due to the exponential weight
  $\sup_{|s|\geq R} e^{\delta_j s \beta(s)} e^{- \delta_k s \beta(s)}
  = e^{-(\delta_k-\delta_j)R}$).

The smooth points $u\in E_\infty$ are those smooth maps
  $u\in\cC^\infty(\R,\R^n)$ whose derivatives decay exponentially,
  $\sup_{s\in\R} e^{\delta s \beta(s)} | \partial_s^{N} u (s) | < \infty$
  for all $N\in\N_0$ and every submaximal weight $\delta < \sup_{k\in\N_0}
  \delta_k$.
(In case of an unbounded weight sequence $\ul\delta$, this means that
  the maps decay faster than any linear exponential.)
In particular, the compactly supported smooth functions are a subset
  $\cC^\infty_0(\R,\R^n)\subset E_\infty$; and these are dense in any
  weighted Sobolev space (for ${p<\infty}$).
\end{proof}

Note that in typical applications, sc-Banach spaces must be chosen so
  that an elliptic regularity result will hold between scales; see the
  regularization property of scale operators as discussed at the end of
  Section \ref{ss:sc} above and Definition \ref{def:sc+} below.
Consequently, it should not be surprising that certain Sobolev spaces
  arise as sc-Banach spaces.
Another natural candidate is the collection of H\"{o}lder spaces
  \((\mathcal{C}^{k, \alpha})_{k\in \mathbb{N}}\) for \(\alpha\in (0,1)\), 
  however such spaces do not form an sc-Banach space because the
  infinity level, \(\mathcal{C}^\infty\), is not dense in any given finite
  scale.\footnote
  {
    For example, the function \(x\mapsto |x|^\alpha\) cannot be
      approximated by  differentiable functions in the \(\mathcal{C}^{0,
      \alpha}\) norm.
    Indeed, convergence \(f_n(x)-|x|^\alpha \to 0\) in
      \(\mathcal{C}^{0, \alpha}\) would imply $\lim_{n\to \infty}
      \lim_{x\to0} |x|^{-\alpha}\bigl|(f_n(x)-|x|^\alpha) - (f_n(0) -
      |0|^\alpha)\bigr|=0$.
    On the other hand, if the functions \(f_n\) are differentiable at
      $0$, then the limit for $\pm x>0$ can be rewritten as $\lim_{n\to
      \infty}\lim_{x\to 0} \bigl|\pm\frac{f_n(x)-f_n(0)}{x} |x|^{1-\alpha}
      -1\bigr| = \lim_{n\to \infty} |\pm f_n'(0) \, 0^{1-\alpha} -1| =1$
      since $1-\alpha>0$; a contradiction.
  } 
This difficulty can be resolved simply by defining the
  levels of an sc-Banach space to be the closure of the smooth functions
  in each level; in other words, define \(E_k:={\rm cl}_{\mathcal{C}^{k,
  \alpha}} (\mathcal{C}^\infty)\).
This idea holds more generally, as the following lemma illustrates.  

\begin{lemma}
Let \(E_0\) be a Banach space, and let \(E_0\supset E_1 \supset
  E_2\supset \cdots\) be a nested sequence of linear subspaces, each
  equipped with a Banach norm \(\|\cdot \|_k\).
Suppose further that the inclusion maps \(E_k\to E_j\) are compact for
  each \(j<k\), but also assume that \(E_\infty:=\cap_{k\in \mathbb{N}}
  E_k\) is \emph{not} dense so \((E_k)_{k\in \mathbb{N}}\) is not an
  sc-Banach space.
Define \(\widehat{E}_k:={\rm cl}_{E_k}(E_\infty)\); then
  \((\widehat{E}_k)_{k\in \mathbb{N}}\) (equipped with the norms
  \(\|\cdot\|_k\)) is an sc-Banach space.
\end{lemma}
\begin{proof}
We begin by observing that by continuity of the inclusion
  \(E_k\hookrightarrow E_j\), the closure \(\widehat{E}_k={\rm
  cl}_{E_k}(E_\infty)\) is a subset of $\widehat{E}_j={\rm
  cl}_{E_k}(E_\infty)$ for any \(j<k\).
Moreover, the inclusion map \(\widehat{E}_k\hookrightarrow
  \widehat{E}_j\) is compact since it is the restriction of a continuous
  compact map.
(For compactness note that any bounded set $\Om\subset \widehat E_k$
  is bounded in $E_k$ as well, and hence ${\rm cl}_{E_j}(\Om)\subset E_j$
  is compact.
However, this closure is also a subset of $\widehat E_j$ by construction,
  so that $\Om$ is precompact in $\widehat E_j$.)
Finally, \(E_\infty\subset \widehat E_k\) is dense for each \(k\in\N_0\)
  by construction.
In fact, we have $\bigcap_{k\in\N_0} \widehat E_k = E_\infty$ since
  this intersection is nested between $E_\infty$ and $\bigcap_{k\in\N_0}
  E_k = E_\infty$.
\end{proof}

Finally, we can define scale continuity for maps between open subsets
  of sc-Banach spaces by the same notion as for general sc-topological
  spaces, namely requiring continuity on every scale.

\begin{definition}\label{def:scContinuity}
Let $X$ and $Y$ be equipped with sc-topologies.
A map $f:X \to Y$ is called {\bf sc-continuous}, abbreviated {\bf
  sc}$\mathbf{^0}$, if for each $k\in\N_0$ the restriction $f|_{X_k} :
  X_k\to Y_k$ is continuous.
\end{definition}

\subsection{Scale differentiability and scale
  smoothness}\label{sec:scCalc}

The differences between standard and scale calculus in infinite
  dimensions stems exclusively from the following novel notion of scale
  differentiability and its implications.
This notion is chosen such that, on the one hand reparametrizations act
  differentiably on spaces of functions as in Example~\ref{ex:trans},
  and on the other hand the chain rule is satisfied, see
  Theorem~\ref{thm:ChainRule}.

\begin{definition} \label{def:sc1}
An sc$^0$ map $f:\E\to \F$ between sc-Banach spaces is {\bf continuously
  scale differentiable}, abbreviated  {\bf sc}$\mathbf{^1}$, if for every
  $x\in E_1$ there exists a bounded linear operator $\rD_x f:E_0\to F_0$
  such that
  $$
    \frac{ \bigl\| f ( x + h ) - f ( x ) - \rD_x f ( h ) \bigr\|_{F_0}
    }{\| h \|_{E_1} } \; \; \xrightarrow[\| h\|_{E_1} \to 0] \; \; 0
  $$
  and the map $E_1 \times E_0 \to F_0$ given by $(x, h) \mapsto  \rD_x f
  (h)$ is sc$^0$ with respect to the sc-structure \mbox{$(E_{k+1}\times
  E_k)_{k\in\N_0}$.}
\end{definition}

While this notion is structurally similar to the classical definition
  of continuous differentiability, in that it contains the existence of a
  bounded linear operator $\rD_x f$ and a notion of continuous variation
  with $x$, it differs in two essential ways:
Firstly, the classical pointwise differentiability uses $\|h\|_{E_0}$
  in the difference quotient, rather than $\|h\|_{E_1}$, and requires
  differentiability at every point $x\in E_0$, rather than just on $E_1$.
In other words, it looks like we are just requiring $f$ to restrict to
  a differentiable map $E_1\to F_0$.
 
Secondly, classical continuous differentiability from $E_1$ to $F_0$
  requires the continuity of the differential $E_1 \to \cL(E_1,F_0),
  x \mapsto \rD_x f$ with respect to the operator norm.\footnote
  {
  The space of bounded linear operators $\cL(H,K)=\{ D:H\to K \;
    \text{linear} \st \|D\|_\cL <\infty \}$ between Banach spaces $H,K$
    is itself a Banach space with norm $\|D\|_\cL:=\sup_{h\neq 0}
    \frac{\|Dh\|_K}{\|h\|_H} <\infty $.
  }
On the other hand, classical continuous differentiability from $E_0$
  to $F_0$ requires continuity of the differential as map $E_0 \to
  \cL(E_0,F_0)$.
Thus we see that scale differentiability is an intermediate notion,
  in which the differential is required to exist as a bounded operator
  $\rD_x f\in\cL(E_0,F_0)$, but only for $x\in E_1$, and the continuity
  requirement is weaker in that it only requires pointwise convergence
  $\|\rD_{x^\nu} f(h) - \rD_{x} f(h)\|_{F_0} \to 0$ for fixed $h\in E_0$
  as $\|x^\nu-x\|_{E_1}\to 0$, rather than convergence of operators
  $\sup_{\|h\|_{E_0}=1}\|\rD_{x^\nu} f(h) - \rD_{x} f(h)\|_{F_0} \to 0$
  as $\|x^\nu-x\|_{E_0}\to 0$.
However, at this point scale differentiabiliy adds requirements at every
  scale: The restrictions $\rD_x f|_{E_k}$ of the differential have to
  induce a map $E_{k+1} \to \cL(E_k, F_k)$, which is continuous in the
  pointwise sense as above.
(Equivalently, this map is continuous with respect to the compact open
  topology on $\cL(E_0,F_0)$.)
These considerations lead to the following comparison between classical
  and scale differentiability.

\begin{remark} \rm  \label{rmk:sc1}\hspace{2mm} \\ \vspace{-5mm}
\begin{enumerate}
  \item
  On a finite dimensional vector space with trivial sc-structure,
    the notion of scale differentiability is the same as classical
    differentiability.
  \item
  Assume that the restricted maps $f|_{E_{k}} : E_{k}\to F_k$ are
    classically $\cC^{1}$ for every $k\in\N_0$.
  Then $f$ is $sc^{1}$ by \cite[Prop.1.9]{HWZscSmooth}.
  \item
  Assume that $f:\E\to \F$ is sc$^1$, then the induced maps $f|_{E_{k+1}}
    : E_{k+1}\to F_k$ are classically $\cC^1$ for every $k\in\N_0$
    by \cite[Prop.1.10]{HWZscSmooth}.
  \item
  By \cite[Prop.2.1]{HWZscSmooth} an sc$^0$ map $f$ is sc$^1$ if and
    only if the following holds for every $k\in\N_0$.
  \begin{enumerate}
    \item 
    The restricted map $f|_{E_{k+1}}:E_{k+1}\to F_{k}$ is classically
      $\cC^1$.
    In particular, the differential $\rD f:E_{k+1}\to L(E_{k+1},F_{k}),
      x\mapsto \rD_x f$ is continuous.
    \item 
    The differentials $\rD_x f:E_{k+1}\to F_k$ for $x\in E_{k+1}$ extend
      to a continuous map ${E_{k+1}\times E_k\to F_k, (x,h)\mapsto \rD_x f
      (h)}$.
    In particular, each extended differential $\rD_x f: E_k \to F_k$
      is bounded.
    \end{enumerate}
  \end{enumerate}
\end{remark}

The motivating example for the development of scale calculus is the
  action of reparametrizations on map spaces, which we give here in
  the simplest form of real valued functions on $S^1$.

\begin{example} \label{ex:trans} \rm 
Recall that the {\bf translation action} on $S^1:=\R/\Z$, which is
  similar to Example~\ref{ex:MorseDiff} and given by
  $$
    \tau: \R \times \cC^0(S^1) \to \cC^0(S^1), \quad (s, \gamma)
    \mapsto \gamma(s + \cdot) ,
  $$
  has directional derivatives only at points $(s_0,\g_0)\in \R\times
  \cC^1(S^1)$ and is in fact nowhere classically differentiable.
However, $\tau$ is sc$^1$ if we equip $\cC^0(S^1)$ with the sc-structure
  $(\cC^k(S^1))_{k\in\N_0}$ of Example~\ref{ex:Ck}.
Indeed, the differential is
  \begin{align*}
    \rD_{(s_0,\gamma_0)}\tau (S, \Gamma) \;= \; S \, \dot\gamma_0(s_0 +
    \cdot) \;+\; \Gamma(s_0 + \cdot)
    \;=\; S\tau(s_0, \dot{\gamma}_0) +\tau(s_0, \Gamma) ,
    \end{align*}
  which for fixed $(s_0,\gamma_0)\in \R\times \cC^{k+1}(S^1)$ is a
  bounded operator $\R \times \cC^k(S^1) \to  \cC^k(S^1)$, and for
  varying base point is a continuous map $\R \times \cC^{k+1}(S^1)
  \times \R \times \cC^k(S^1) \to  \cC^k(S^1)$.
\end{example}

More conceptually, the notion of scale differentiability can equivalently
  be phrased as the existence and scale continuity of a tangent map.

\begin{definition}\label{def:scTangentBundle}
The {\bf sc-tangent bundle} of a Banach space $\E=(E_k)_{k\in\N_0}$ is
  $$
    \rT \E := E_1 \times E_0  \quad\text{with sc-structure} \quad
    (E_{k+1}\times E_k)_{k\in\N_0}.
  $$
The {\bf tangent map} of an sc$^1$ map $f:\E\to \F$ is
  $$
    \rT f : \rT \E \to \rT \F , \qquad (x, h) \mapsto \bigl( f(x) ,
    \rD_x f (h) \bigr) .
  $$
\end{definition}

Here a point $(p,v)\in \rT\E$ in the sc-tangent space is viewed as
  tangent vector $v\in E_0$ at the base point $p\in E_1$.
Hence, the sc-tangent bundle of $\E$ is a bundle $\rT\E \to E_1$ over
  the dense subspace $E_1\subset E$ whose fiber at each point is the
  entire vector space $E_0=E$.
We can now give a brief defininition of scale differentiability and
  extend it naturally to notions of $k$ times sc-differentiable and
  the notion of sc-smoothness.

\begin{definition}\label{def:scDiff}
Let $f:\E\to \F$ be a  sc$^0$ map between sc-Banach spaces.
\begin{enumerate}
  \item $f$ is sc$^1$ if the tangent map $\rT f: \rT\E \to \rT\F$
    exists and is sc$^0$.
  \item $f$ is  {\bf sc}$\mathbf{^k}$ for $k\geq 2$ if the tangent map
    $\rT f$ is sc$^{k-1}$.
  \item $f$ is {\bf scale smooth}, abbreviated {\bf
    sc}$\mathbf{^\infty}$, if the tangent map $\rT f$ is sc$^k$ for
    all $k\in\N_0$.
  \end{enumerate}
\end{definition}

\begin{remark}
[{\bf Scale calculus with boundary and corners}]
 \label{rmk:corners} \rm
The notions of tangent bundle, $sc^0$ map, tangent map, $sc^k$,
  and $sc^\infty$ extend naturally to maps defined on open sets
  $\cU\subset\E$ of sc-Banach spaces and relatively open sets
  $\cU\subset[0,\infty)^k\times\E$ in {\bf sectors} (special cases of the
  ``partial quadrants'' defined by HWZ).
Indeed, scale continuity is defined with respect to the induced topology
  on the subset $\cU$;  the differential \(D_xf\) must still satisfy
  the limiting property as displayed in Definition \ref{def:sc1} however
  only under the slightly weaker condition that \(x+h\in \mathcal{U}_1\)
  as \(\|h\|_{E_1}\to 0\); and the differential \((x, h)\mapsto D_xf(h)\)
  maps \(\mathcal{U}_1\times E_0\to F_0\) and is scale continuous.
See Definition 1.14 of \cite{HWZnew}.
\end{remark}

Note that in order to build a new sc-differential geometry based on
  the notion of scale differentiability, it is crucially important that
  the chain rule holds.
Indeed, we state this as a sample from the large body of work in
  which HWZ reprove the standard calculus theorems in the framework
  of sc-calculus.
The proof in \cite[Thm.2.16]{HoferWysockiZehnder1} makes crucial
  use of the compactness assumption on the scale structure in
  Definition~\ref{def:scBanachSpace}~(ii).

\begin{theorem}[{\bf Chain Rule}] \label{thm:ChainRule}
Let $\E,\F,\mathbb{G}$ be sc-Banach spaces, and suppose that $f:\E\to
  \F$ and $g:\F\to \mathbb{G}$ are sc$^1$ maps.
Then $g\circ f:\E\to \mathbb{G}$ is sc$^1$ and $\rT(g\circ f)=\rT g\circ
  \rT f $.
\end{theorem}
Finally, we can use the chain rule to prove scale smoothness of the
  translation action.

\begin{example}\rm \label{ex:transsmooth}
The tangent map of Example~\ref{ex:trans},
  \[
    \begin{aligned}
    \rT\tau : \;\; \R\times\cC^1(S^1)\times \R\times\cC^0(S^1) & \; \to \;
    \quad \cC^1(S^1)\times\cC^0(S^1)  \\
    (s_0,\gamma_0,S, \Gamma) \phantom{\cC^0(S^1)} &\;\mapsto\;  \bigl(
    \tau (s_0, \gamma_0) , S \cdot \tau(s_0, \dot\gamma_0) + \tau(s_0,
    \Gamma) \bigr)
    \end{aligned}
  \]
  can be expressed as composition of sum, multiplication, derivative
  $\cC^1(S^1)\to \cC^0(S^1), \gamma\mapsto\dot\gamma$, and the
  translation $\tau:\R\times\cC^0(S^1)\to \cC^0(S^1)$ itself.
All of these are sc-continuous, and, by linearity, the first three are
  in fact sc\(^\infty\).
Hence, by the chain rule (stated above as Theorem~\ref{thm:ChainRule}),
  $\rT\tau$ is as scale differentiable as $\tau$.
This proves that the translation $\tau$ is in fact sc$^\infty$.
\end{example}

\subsection{Scale manifolds} \label{ss:scmfd}

The scale calculus on Banach spaces can now be used to obtain a variation
  of the notion of a Banach manifold by replacing Banach spaces with
  scale Banach spaces and by replacing smoothness requirements with
  scale smoothness.
This new notion of scale manifold coincides with the classical notion
  of manifold in finite dimensions by Example~\ref{ex:trivial} and
  Remark~\ref{rmk:sc1}~(i); for a precise definition of scale manifold,
  see \cite[\S 2.4]{HoferWysockiZehnder1}.
In infinite dimensions, neither notion is stronger than the other,
  however in applications most Banach manifolds could be equipped with
  an additional scale structure.

In practice, scale manifolds are of limited utility, since they are
  not general enough for moduli problems involving broken trajectories
  or nodal curves, and they are a rather special case of the more general
  notion of an M-polyfold.
Nevertheless, they serve as a useful stepping stone between Banach
  manifolds and M-polyfolds, and we will use them here to illuminate the
  concept of scale smoothness by outlining how the space of maps, modulo
  reparametrization, is given the structure of a scale manifold; that is,
  it has metrizable topology, it is locally homeomorphic to open subsets of
  scale Banach spaces, and the induced transition maps are scale smooth.
In order to prevent isotropy, we restrict ourselves to maps from $S^1$
  to $S^1$ of degree~$1$,
  $$
    \cB := \bigl\{ \g \in \cC^1( S^1 , S^1 ) \,\big|\, \deg \g =
    1 \bigr\} .
  $$
By identifying $S^1=\R/\Z$, we observe that the translation action $\tau$
  from Example~\ref{ex:trans} descends to an action $S^1\times\cB\to \cB$,
  which by the degree restriction is free.
Next, we will sketch how to construct local slices for the action of
  ${\rm Aut}=S^1$ on $\cB$ along the lines of Remark~\ref{rmk:slice},
  and, from these observations, we will obtain sc-manifold charts for the
  quotient space $\cB/{\rm Aut}$.
\begin{itemlist}
  \item
  For any fixed $a\in S^1$, one can check that the space of maps that
    transversely intersect $a$ at $0\in S^1$, denoted
    $$
      \cB_a := \bigl\{ \g \in \cB \,\big|\, \g(0)=a , \rd_0 \g \neq 0
      \bigr\} ,
    $$
    is a local slice; in other words, the map $\cB_a \to \cB/{\rm Aut}$, $\g
    \mapsto [\g]$ is a local homeomorphism.
  \item
  Each $\cB_a$ is locally homeomorphic to an open set in the model
    Banach space
    $$
      E_0 :=  \bigl\{ \xi \in \cC^1( S^1 , \R ) \,\big|\, \xi(0)=0
      \bigr\} ,
    $$
    via the map $E_0 \to \cB_a$, $\xi \mapsto \g+\xi \;(\text{mod}\;
    \Z)$ centered at a fixed $\g\in\cB_a$.
  \item
  The Banach space $E_0$ can be equipped with the scale structure
    $$
      E_k:=  \bigl\{ \xi \in \cC^{1+k}( S^1 , \R ) \,\big|\, \xi(0)=0
      \bigr\} .
    $$
  \item
  For any $a\in S^1$ and $\g\in\cB_a$ there exists a sufficiently
    small open ball $N_{a,\g}\subset E_0$ such that the composition of
    maps $E_0\to\cB_a\to \cB/{\rm Aut}$ restricts to a homeomorphism
    $\Phi_{a,\g}:N_{a,\g} \overset{\sim}{\to} \cU_{a,\g}$ to a
    neighborhood of $[\g]\in \cB/{\rm Aut}$.
  \item
  Thus $\cB/{\rm Aut}$ is covered by (topological) Banach manifold
    charts, whose domain $E_0$ is enriched with a scale structure.
  \end{itemlist}

In order to equip $\cB/{\rm Aut}$ with the structure of a scale manifold,
  it remains to check scale smoothness of the transition maps, given by
  \begin{align*}
    \Phi_{a_2,\g_2}^{-1} \circ \Phi_{a_1,\g_1}  \,:\; \quad E_0
    \;\supset\; \Phi_{a_1,\g_1}^{-1}(\cU_{a_2,\g_2}) &\;\longrightarrow\;
    E_0 \\ \xi \qquad &\;\longmapsto\; \tau(s_\xi ,  \g_1 +  \xi )  -
    \g_2 ,
  \end{align*}
  where $s_\xi\in\R$ is determined\footnote
  { 
  In general, there may be several solutions to the equation
    \(\g_1(s_\xi)+\xi(s_\xi)=a_2\), however a unique solution can be
    determined by the \(\cC^1\)-smallness condition \(\tau(s_\xi, \g_1
    +\xi)\approx \g_2\).
  See Remark~\ref{small print} for further details.
  }
  by $\g_1(s_\xi)+\xi(s_\xi) = a_2$.
These transition maps are not classically differentiable but we can
  check that they are scale smooth by the following steps; c.f.\ \cite[\S
  4.1]{HWZscSmooth}.

\begin{itemlist}
  \item
  The map $\g \mapsto s_\g$ from a $\cC^1$ neighborhood of $\g_2$ to
    a neighborhood $I_2\subset\R$ of $0$, given by solving $\g(s_\g) =
    a_2$ for $s_\g\in I_2$, is well defined for sufficiently small choices
    of the neighborhoods.
  It is $\cC^1$, by the implicit function theorem, if the neighborhoods
    are also chosen to guarantee transversality.
  Next, one can differentiate the implicit equation for $s_\g$ to
    check that the variation of $s_\g$ with $\g\in\cC^k$ is $k$ times
    continuously differentiable.
  This proves property a) of Remark~\ref{rmk:sc1}~(iv).
  To check the refined continuity required in b) one inspects the
    expression for the differential that arises from the implicit equation.
  After employing the classically smooth map \(\xi\mapsto \g_1+\xi\)
    to model the problem on a Banach space, this shows that the map
    $\xi\mapsto s_\xi$ is sc$^\infty$.
  \item
  Note that $\Phi_{a_2,\g_2}^{-1} \circ \Phi_{a_1,\g_1}$ is a composition
    of the above map with addition and translation.
  The latter was shown to be sc$^\infty$ in Example~\ref{ex:transsmooth}.
  Addition is classically smooth on each level, hence scale smooth.
  Now the chain rule for composition of scale smooth maps,
    Theorem~\ref{thm:ChainRule}, implies scale smoothness of the
    transition map.
  \end{itemlist}

In order to conclude that $\cB/\Aut$ is a scale manifold, it now remains
  to check that its quotient topology (in which the chart maps are local
  homeomorphisms) is Hausdorff and paracompact.
The latter follows if we can cover $\cB/\Aut$ with finitely many charts,
  and the Hausdorff property holds if the equivalence relation induced
  by $\Aut$ is closed (preserved in limits).

{\small
\begin{remark}[{\bf Small print on uniqueness in the slicing conditions}]
  \rm \label{small print}
In general, the implicit equation $\g_1(s_\xi)+\xi(s_\xi) = a_2$ for
  $s_\xi$ may have a large irregular set of solutions, but the formula
  for the transition maps $\Phi_{a_2,\g_2}^{-1} \circ \Phi_{a_1,\g_1}$,
  and similarly the proof of injectivity of each chart $\Phi_{a,\g}$,
  requires a unique solution.
Since we guaranteed trivial isotropy, this uniqueness can be achieved
  by solving for $s_\xi$ in a small subdomain of $\R$.

More precisely, one can construct the local slice near $[\gamma_2]$
  in a neighborhood $\cU_{a_2,\g_2}$ that is given as quotient of an
  $\eps$-neighborhood around $\gamma_2\in\cC^1(S^1,S^1)$ so that for
  given $\delta>0$ the following holds:
For each equivalence class $[\g_0]\in \cU_{a_2,\g_2}$, there exists a
  (not necessarily unique) \(s_0\in S^1\) so that $
  d_{\cC^1}(\g_0(s_0+\cdot),\g_2) <\eps $ and
  \begin{equation*}
    d_{\cC^1}(\g_0(s+\cdot),\g_2)\le \eps  \quad\Rightarrow\quad
    |s-s_0| < \delta .
    \end{equation*}
In other words, the set of shifts of \(\g_0\) which are
  \(\epsilon\)-close to \(\g_2\) in \(\mathcal{C}^1\) is a $2\delta$-small
  interval in $S^1$.
Moreover, the constants $\eps,\d>0$ can be chosen so that for each $\g$
  in the $\eps$-neighborhood of $\g_2$ there exists a unique
  \(|s_2|<\d\) for which \(\g(s_2) = a_2\) and \(\g'(s_2)\neq 0\).
Consequently, for any choice of \(\xi_{12}\in N_{a_1, \g_1}\)
  with the property that \(\Phi_{a_1, \g_1}(\xi_{12})\in
  \cU_{a_1, \g_1}\cap \cU_{a_2, \g_2}\), one can find a shift
  value \(s_{12}\in \mathbb{R}/\mathbb{Z}\) with the property
  that \(\gamma_1(s_{12})+\xi_{12}(s_{12})=a_2\); furthermore
  for each \(\xi\approx \xi_{12}\) there exists a unique
  \(s_\xi\) satisfying \(|s_{12}-s_\xi|<\delta\) which solves
  \(\gamma_1(s_\xi)+\xi(s_\xi)=a_2\).
For a more detailed construction of $\eps,\de$ see e.g.\
  \cite{AW,jiayong, HWZ_lectures}.

For a more general quotient of nonconstant, continuously differentiable
  functions modulo translation, denoted by $\cC^1_{\rm nc}(S^1)/S^1$,
  the above constructions will just provide a scale orbifold structure
  due to the possible finite stabilizers $G\subset S^1$, fixing a map
  $\tau(G,\gamma) = \gamma$.
This can be seen above as the lifts from $\cB/\Aut$ to a $\cC^1$
  neighborhood of the center of the chart $\g_2$ being unique only up to
  shift by a tuple of intervals $G+(s_0-\d,s_0+\d)$, where $G\subset S^1$
  is the isotropy group of $\g_2$.
\end{remark}
}

We end this section by transferring the previous slicing construction
  to maps with noncompact domain $\R$, as required for the application
  to Morse theory.

\begin{example}[{\bf 
Scale smooth structure on trajectory
  spaces}]\label{ex:MorseTrajectorySpaces} \rm
For simplicity we will consider a Morse function \(f:X\to \mathbb{R}\)
  where \(X=\mathbb{R}^n\).
In order to construct the space of (not necessarily Morse) trajectories
  between two critical points $a\neq b$, we begin by fixing a reference
  path \(\psi_a^b\in \mathcal{C}^\infty(\mathbb{R}, X)\) from \(\lim_{t\to
  -\infty} \psi_a^b(t) = a\) to \(\lim_{t\to\infty}\psi_a^b(t)=b\),
  whose derivative has compact support.
Then we define a metric space of paths from $a$ to $b$ by\footnote
  {
  One can check that \(\mathcal{B}_a^b\) does not depend on the choice
    of reference path \(\psi_a^b\) as specified above.
  }
  \begin{equation*}
    \mathcal{B}_a^b:=\bigl\{\g\in W_{loc}^{2, 2}(\mathbb{R}, X)\,\big|\,
    \; \exists\; v\in W^{2, 2}(\mathbb{R}, X) \text{ s.t. } \g =
    \psi_a^b + v  \bigr\}.
    \end{equation*}
Now let  the automorphism group \({\rm Aut}:=\mathbb{R}\)
  act on \(\mathcal{B}_a^b\)  by the translation action as in
  Example~\ref{ex:MorseFredholm},
  \begin{equation}\label{2tau}
    \tau: \R\times W^{2,2}(\R, X)\to W^{2,2}(\R, X) \qquad \text{given
    by}\qquad\tau(s,\gamma):=\gamma(s+\cdot).
    \end{equation}
Then we define the {\bf space of trajectories} from $a$ to $b$ as the
  metric space
  $$
    \widetilde{\mathcal{B}}_a^b\;:=\; \quo{\mathcal{B}_a^b}{{\rm Aut}}
    , \qquad d([\g_1], [\g_2]):= \inf_{t\in \mathbb{R}} \|\g_1(t+\cdot)
    - \g_2(\cdot)\|_{W^{2, 2}}.
  $$
This space can be given the structure of an sc-manifold in the following
  manner.
For any given point \([\psi]\in \widetilde{\mathcal{B}}_a^b\), we pick
  a representative \(\psi \in \mathcal{B}_a^b \) such that \(\psi'(0)\neq
  0\). 
(For simplicity we also assume that $\psi$ is constant near $\pm
  \infty$.)
Then the following open subsets of Banach spaces will provide local
  models for \(\mathcal{B}_a^b\) and \(\widetilde{\mathcal{B}}_a^b\),
  \begin{align*}
    &U^\psi:=\bigl\{u\in W^{2, 2}(\mathbb{R}, X) \,\big|\, \|u\|_{W^{2,
    2}}<\epsilon \bigr\}
    &V^\psi:=\bigl\{u\in U^\psi \,\big|\, \langle \psi'(0), u(0)\rangle =
    0 \bigr\}.
  \end{align*}
Here \(\epsilon, \delta>0\) are chosen so that
  \begin{enumerate}
  \item the map \(\Psi:V^\psi\to \widetilde{\mathcal{B}}_a^b\) given by
    \(\Psi(u) = [\psi+u]\) is injective,
  \item for each \(u\in U^\psi\), the restricted map \(\psi+u:(-\delta,
    \delta)\to X\) has unique and transverse intersection with the
    hyperplane $H^\psi:=\{p\in X \,|\, \langle p-\psi(0), \psi'(0)\rangle
    =0\}$.
  \end{enumerate}
Then the fact that \(v\in V^\psi\) implies \((\psi+v)(0)\in H^\psi\),
  together with the above two conditions, guarantees that \(\Psi:V^\psi\to
  \widetilde{\mathcal{B}}_a^b\) given by \(u\mapsto [u+\psi]\) is a local
  chart for $\widetilde{\mathcal{B}}_a^b$; in other words, it is a
  homeomorphism to an open subset.

In order to give the trajectory space \(\widetilde{\mathcal{B}}_a^b\)
  the structure of an sc-manifold, it remains to exhibit \(V^\psi\)
  as an open subset of an sc-Banach space and to verify that the transition
  maps induced by different choices of centers $[\psi]$, or representatives	
  $\psi$, are sc\(^\infty\)-diffeomorphisms.
For the first step, recall the sc-structure $W_{\delta_k}^{2+k,
  2}(\R,\R^n)$ from Lemma \ref{ex:sobolev}, where we fix a weight
  sequence\footnote{In order to capture all Morse-trajectories,
  it will be important to choose this sequence so that \(\sup
  \delta_k<\inf_{\|x\|=1}{\rm min}_{p\in \{ a, b, c\}}|  D^2f_p(x, x)|\).
We do not make use of this condition in the present example however.}
  \(0=\delta_0< \delta_1< \delta_2<\cdots\;\).
The slicing condition cuts out closed codimension $1$ subspaces from
  each scale, which then yields an sc-Banach space with scales $E_k:=\bigl\{u\in
  W_{\delta_k}^{2+k, 2}(\R,\R^n) \,\big|\, \langle \psi'(0), u(0)\rangle =
  0 \bigr\}$ so that $V^\psi\subset E_0$ is an open subset.
Finally, scale smoothness of the transition maps is proven by arguments
  similar to those above for the case of trajectories parametrized
  by $S^1$.
\end{example}

{\small
\begin{remark}[{\bf Small print on covering by charts with smooth
  center}]
There is a subtle but important point to be made about the above example,
  namely that our local chart is centered at the point
  \([\psi_a^b]\in\widetilde{\mathcal{B}}_a^b\) which is represented by
  a \(\mathcal{C}^\infty\) map $\psi_a^b$.
For all standard applications like gradient flow lines, Floer
  trajectories, or pseudoholomorphic curves, it is essential that this
  base map be \(\mathcal{C}^\infty\).
This is due in part to the fact that (equivalence classes of) maps of
  any regularity near this base point need to be obtained by exponentiating
  vector fields along the base map.
However, there is no well-defined notion of a \(\mathcal{C}^{k+\ell}\)
  vector field along a \(\mathcal{C}^k\) map for \(\ell>0\), and even if
  there was, then the map resulting from exponentiation (e.g.\ addition
  in the simplest case) would only be \(\mathcal{C}^k\).
This also points to the second issue of transition maps between different
  charts centered at  \(\mathcal{C}^k\), which is that maps generally
  do not preserve \(\mathcal{C}^{k+\ell}\) regularity for \(\ell>0\),
  and hence will not be scale continuous, let alone scale smooth.

Note however, that by constructing only charts with
  \(\mathcal{C}^\infty\) centers, we run the risk of failing to cover the
  given topological space despite the fact that (equivalence classes of)
  \(\mathcal{C}^\infty\) maps are dense in this total space.
Indeed, it is worth recalling that the rational numbers are dense in the
  reals and can be covered by a countable collection of open intervals,
  the union of which can be made to have arbitrarily small measure.
Consequently this collection of charts (given by the open sets) does
  not cover all of $\R$.
This is a general issue in constructing atlases for a scale manifold,
  or more generally M-polyfolds.
There are two approaches for dealing with this issue.
First, in the standard applications, all elements of the compactified
  moduli space are in fact represented by \(\mathcal{C}^\infty\) maps.
Hence one could redefine the scale manifold (or M-polyfold) as the
  subset that is covered by the charts with \(\mathcal{C}^\infty\) centers.
This possibly smaller set still contains the compactified moduli space,
  and if it is the zero set of an appropriate Fredholm section, then an
  M-polyfold perturbation scheme and implicit function theorem can be
  used to regularize it.
Note that the invariance part of the regularization would now also have
  to address changes in the cover used for perturbation, and thus in the
  ambient M-polyfold used.
Roughly speaking, any two such covers should cover a common open
  neighborhood of the compactified moduli space, which itself is an
  M-polyfold within which the moduli space can be regularized.

As a second approach, one could try to control the size of the charts
  with \(\mathcal{C}^\infty\) centers in such a way that density of the
  \(\mathcal{C}^\infty\) points guarantees that the whole space is covered
  by charts.
Note that in fact a local lower bound on the chart size would suffice.
This is the argument by which usual Sobolev completions of maps
  are given the structure of a Banach manifold, but it is complicated by 
  slicing conditions in the polyfold applications. 
To prove such a bound, \cite{HWZI_applications} develop the following
  technique that can also be employed in other standard applications:
  To show that a given point, which is represented by a map of less
  than \(\mathcal{C}^\infty\) regularity, lies within a chart with
  \(\mathcal{C}^\infty\) center, we build a ``tentative chart''
  centered at the given representative in the same way as if it was
  \(\mathcal{C}^\infty\).
This involves geometric constructions like choosing a transverse
  hypersurface, which are possible at general centers if we choose the
  basic regularity of maps in the total space sufficiently high.
For pseudoholomorphic curves, this (again) motivates using spaces of
  $W^{3,2}$ maps, since the Sobolev embedding to $\cC^1$ ensures that
  the notion of transversality to a hypersurface is meaningful.
The resulting tentative chart is a homeomorphism to an open subset of
  the total space, just with less smooth structure on its domain (e.g.\
  an open subset in a Banach space instead of a scale Banach space).
Then, by density of the \(\mathcal{C}^\infty\) points in the total
  space, the corresponding points will also be dense in the domain of
  the tentative chart, where they correspond to \(\mathcal{C}^\infty\)
  maps near the given representative.
One can now ``re-center'' the tentative chart at these
  \(\mathcal{C}^\infty\) maps to obtain new charts whose size is controlled
  analogously to the radii of balls of varying center that are contained
  in a given ball.
More concretely, one uses the geometric choices (e.g.\ of transverse
  hypersurface) of the tentative chart in the construction of charts
  centered at the \(\mathcal{C}^\infty\) maps within the tentative chart.
Then the size of the new scale-smooth charts (which is related to
  injectivity radii and preserving the transversality) is bounded
  below for centers in a neighborhood of the original center.  This,
  in turn, ensures that this original center is contained in the new
  charts whose \(\mathcal{C}^\infty\)  centers are sufficiently close.
\end{remark}
}

\section{M-Polyfolds}\label{s:Mpolyfolds}

This section defines the notion of an M-polyfold, which is something akin
  to a manifold locally modeled on scale smooth retractions.
In order to provide a roadmap, we begin by stating the definition of an
  M-polyfold, which is obtained by simply replacing the notion of charts
  and smooth transition maps in the definition of a classical manifold
  with the generalized concepts that will be the topic of discussion
  in this section.
As a running application, we will consider examples from Morse theory
  to illuminate the definitions and theorems of this section.

\setcounter{theorem}{0}

\begin{definition} \label{def:Mpolyfold}
An {\bf M-polyfold} is a metrizable space $\cX$ together with an
  open covering by the images of M-polyfold charts (see Definition
  \ref{def:chart}), which are compatible in the sense that the transition
  map induced by the intersection of the images of any two charts is
  scale smooth (see Definition \ref{def:sc smooth map from retract}).
\end{definition}

The notions of M-polyfold charts and scale smoothness between their local
  models will be developed in Sections~\ref{ss:charts} and~\ref{ss:rcalc}.
As for manifolds, we will then see in Section~\ref{ss:corners} that a
  notion of M-polyfold with boundary (and corners) can be obtained by
  allowing M-polyfold charts with boundary (and corners) and by making
  sense of scale smoothness on their underlying local models.

{\small
\begin{remark}[{\bf Topological small print}] \hspace{2mm} \\ \vspace{-5mm} \rm
\begin{enumerate}
  \item 
  Just as for finite dimensional manifolds, any covering by compatible
    charts induces a maximal atlas of compatible charts, which is more
    commonly viewed as manifold or M-polyfold structure on a given space.
  \item
  One could weaken the assumption of metrizability in
    Definition~\ref{def:Mpolyfold} to the assumption that the
    topological space $\cX$ be Hausdorff and paracompact.
  Then, because $\cX$ is covered by M-polyfold charts, which (just
    like manifold charts) provide local homeomorphisms to a metrizable
    space, it immediately follows that $\cX$ is locally metrizable;
    indeed, any point has a neighborhood on which the subspace topology
    is metrizable.
  Thus $\cX$ will automatically be metrizable by the Smirnoff metrization
    theorem \cite[Thm.42.1]{Munkres}.
  We note that, conversely, metric spaces are automatically paracompact
    (by e.g.\ \cite[Thm.41.4]{Munkres}), and hence allow partitions of
    unity subordinate to any open cover.

  \item
  The definition of M-polyfolds in \cite{HoferWysockiZehnder2} works
    under the assumption of second countability instead of the
    assumption of paracompactness; this ensures that the zero set
    $s^{-1}(0)\subset\cX$ of a transverse section $s$ over $\cX$
    inherits the structure of a manifold, which is commonly defined to be 
    second countable, Hausdorff (which follows from being a subset of a
    Hausdorff space), and locally homeomorphic to Euclidean space (which
    follows from an implicit function theorem).
  This was updated in \cite{HWZnew} thanks to two observations.
  Firstly, since the theory is limited to compact zero sets $s^{-1}(0)$,
    second countability follows from metrizability.
  Secondly, paracompactness suffices for the existence of partitions
    of unity, as mentioned above.
  \item
  We will define the notion of an M-polyfold modeled on sc-retracts in
    scale Banach spaces.
  However, the regularization Theorem~\ref{thm:PolyfoldRegularization2}
    will require M-polyfolds modeled on sc-retracts in scale Hilbert
    spaces.
  This guarantees the existence of scale smooth cutoff functions.
  \end{enumerate}
\end{remark}
}

\begin{example}[{\bf 
Space of broken and unbroken
  trajectories}]\label{ex:TrajectoryTopology}\rm
The simplest example of Morse trajectory breaking can be discussed by
  considering a Morse function \(f:\R^n\to \mathbb{R}\) with critical
  points \(\Crit f =\{a, b, c\}\) so that \(b=0\) and $\inf_{\|x\|>R}
  f(x) < f(a)<f(b)<f(c)$ for some $R>>1$.\footnote
  {
  As example of such a Morse function one could take the $2$-sphere
    $\R^2\cup\{\infty\}$ with one maximum $c$, one saddle point $b$,
    and two minima at $a$ and $\infty$.
  Then $\CM_a^c$ is given by a single one-parameter family of unbroken
    trajectories converging to two different broken trajectories at
    the ends.
  }
The constructions of Example~\ref{ex:MorseTrajectorySpaces} equip
  the spaces of unbroken trajectories \(\widetilde{\mathcal{B}}_a^b\),
  \(\widetilde{\mathcal{B}}_b^c\), and \(\widetilde{\mathcal{B}}_a^c\)
  with unique scale topologies and scale smooth structures for any fixed
  weight sequence, and in particular induce a natural $W^{2,2}$-topology.
Given any metric on $\R^n$, the assumption $\inf_{\|x\|>R} f(x) <
  f(\Crit f)$ guarantees that the space of Morse trajectories $\cM_a^c=
  \{[\g]\in \widetilde{\mathcal{B}}_a^c \,|\, \dot \g - \nabla f(\g) =
  0 \}$ is compact up to breaking at $b$.
Here the space of broken trajectories from $a$ to $c$, broken at $b$,
  is given by the Cartesian product \(\widetilde{\mathcal{B}}_a^b\times
  \widetilde{\mathcal{B}}_b^c\) and hence also inherits a natural
  $W^{2,2}$-topology and structure of an sc-manifold.
In order to build an M-polyfold $\cX_a^c$ which contains the compactified
  Morse trajectory space $\CM_a^c$ as compact zero set of a Fredholm
  section, we need to equip the union of the spaces of broken and unbroken
  trajectories
  \begin{equation*}
    \cX_a^c\;:=\; \widetilde{\mathcal{B}}_a^c \;\sqcup\;
    \widetilde{\mathcal{B}}_a^b \times \widetilde{\mathcal{B}}_b^c
    \;\;=\;\; \quo{\mathcal{B}_a^c}{\rm Aut} \;\sqcup\;
    \quo{\mathcal{B}_a^b}{\rm Aut} \times \quo{\mathcal{B}_b^c}{\rm Aut}
    \end{equation*}
  with a topology so that a sequence of gradient trajectories may converge
  to a broken trajectory.
We achieve this by defining the notion of convergence in $\cX_a^c$
  as follows:
For \(p_\infty=[\g]\in \widetilde{\mathcal{B}}_a^c\), we say \(p_n\to
  p_\infty\) if and only if the tail of the sequence is contained
  in \(\widetilde{\mathcal{B}}_a^c\) and \(p_n\to [\g]\) in the
  $W^{2,2}$-topology.
For \(p_\infty=([\g_1], [\g_2])\in \widetilde{\mathcal{B}}_a^b \times
  \widetilde{\mathcal{B}}_b^c\), we say \(p_n\to p_\infty\) if and only if
  there exist local charts \(\Phi:V^\phi\to \widetilde{\mathcal{B}}_a^b\)
  and \(\Psi: V^\psi\to \widetilde{\mathcal{B}}_b^c\) and convergent
  sequences \((0,\infty]\ni R_n\to \infty\), \(V^\psi \ni v_n^\psi\to
  v_\infty^\psi\), and \(V^\phi \ni v_n^\phi\to v_\infty^\phi\) for which
  the tail satisfies
  \begin{equation*}
    p_n = \begin{cases} 
    \bigl[\oplus_{R_n}\big(\phi + v_n^\phi, \psi + v_n^\psi \big) \bigr]
    &; R_n<\infty \\
    \bigl([\phi + v_n^\phi], [\psi + v_n^\psi ]\bigr) &; R_n=\infty 
    \end{cases}
    \qquad \text{and}\qquad  p_\infty =\bigl( [\phi+v_\infty^\phi]
    ,[\psi+v_\infty^\psi]\bigr) .
    \end{equation*}
Here \(\oplus\) is the pregluing map given in Section~\ref{ss:ret},
\begin{align}
  \oplus: (R_0, \infty)\times \bigl( \phi+ V^\phi\bigr)\times \bigl(
  \psi+ V^\psi \bigr) &\;\to\; \mathcal{B}_a^c\label{2oplus} \\
  (R,\g^\phi,\g^\psi) \;\mapsto\; \oplus_R (\g^\phi, \g^\psi) &:=
  \beta \g^\phi(\cdot + {\textstyle \frac{R}{2}}) + (1-\beta)\g^\psi(
  \cdot - {\textstyle \frac{R}{2}}), \notag
  \end{align} 
  where $\beta:\R\to[0,1]$ is a smooth cutoff function with
  $\beta|_{(-\infty,-1]}\equiv 1$ and $\beta|_{[1,\infty)}\equiv 0$.

In other words, a sequence of unbroken or broken trajectories converges
  to a broken trajectory if and only if the sequence and limit are
  the image of a convergent triple \((R_n, v_n^\phi, v_n^\psi) \) with
  \(R_n\to \infty\) under the prospective chart map resulting from the
  pregluing map,
  \[
    (R,\g^\phi,\g^\psi) \;\mapsto\;  \begin{cases}
    \bigl[ \oplus_R(\g^\phi,\g^\psi) \bigr] &; R<\infty, \\
    \bigl( [\g^\phi],[\g^\psi] \bigr) &; R=\infty.
    \end{cases}
  \]
Note that the topologies induced on the subsets of unbroken
  trajectories \(\widetilde{\mathcal{B}}_a^c\) and broken trajectories
  \(\widetilde{\mathcal{B}}_a^b \times \widetilde{\mathcal{B}}_b^c\)
  agree with the $W^{2,2}$-topologies constructed in Example
  \ref{ex:MorseTrajectorySpaces}.
\end{example}

\subsection{M-polyfold charts}\label{ss:charts}

To introduce the notion of charts for M-polyfolds, let us again
  move backwards and start with the main definition, which is a direct
  generalization of a (scale) Banach manifold chart.

\begin{definition}\label{def:chart}
An {\bf M-polyfold chart} for a second countable and metrizable
  topological space $\cX$ is a triple $(U,\phi,\cO)$ consisting of an
  open subset $U\subset \cX$, an sc-retract $\cO\subset\E$ (see Definition
  \ref{def:scRetraction}) in an sc-Banach space $\E$, and a homeomorphism
  $\phi:U\to\cO$.
\end{definition}

A scale manifold chart is the special case of the above definition in the
  case that the sc-retracts \(\mathcal{O}\) are all open subsets in $\E$.
Due to the scale structure, a scale Banach manifold chart has a slightly
  richer structure than a Banach manifold chart, which is obtained by
  replacing open subsets in \emph{Banach spaces} with open subsets in
  \emph{scale Banach spaces}.
The notion of an M-polyfold chart, however, will be much more general
  in the sense that the sets $\cO$ will no longer need to be open (in
  fact, as subsets they  may have empty interior), however they will have
  the structure of being the image of a scale smooth retraction on $\E$.
In particular, this allows a single neighborhood \(U\) in \(\mathcal{X}\)
  to have two M-polyfold charts \(\phi:U \to \mathcal{O}\subset
  \mathbb{E}\) and \(\phi':U \to \mathcal{O}'\subset \mathbb{E}'\) in which
  \(\mathbb{E}\) and \(\mathbb{E}'\) are not isomorphic, but nevertheless
  \(\phi'\circ \phi^{-1}:\mathcal{O}\to \mathcal{O}'\) is sc-smooth.

\begin{definition}\label{def:scRetraction}
A {\bf scale smooth retraction} (for short {\bf sc-retraction})
  on an sc-Banach space $\E$ is an sc$^\infty$ map \(r:\mathcal{U}\to
  \mathcal{U}\subset\E\) defined on an open subset $\mathcal{U}\subset \E$,
  such that $r\circ r = r$, and hence $r|_{r(\cU)} = \id|_{r(\cU)}$.
A {\bf sc-retract} in $\E$ is a subset $\cO\subset\E$ that is the
  image $r(\cU)=\cO$ of an sc-retraction on $\E$. (We will see that most
  subsequent notions are independent of the choice of $r$.)
\end{definition}

Comparing the above definition with the classical notion of retract,
  we note that an sc-retraction is a retraction of the open set $\cU$
  and not the ambient space $\E$.
The latter is relevant only for the notion of smoothness on $\cU$.
Hence, in particular, an sc-retract in $\E$ is not a retract of $\E$, but
  could have nontrivial topology, though such topological considerations
  are of little importance to M-polyfolds.

Next, we present a special case of sc-retracts, namely sc-smooth
  splicing cores, which were introduced as basic models for M-polyfolds
  in \cite{HWZ_lectures,HoferWysockiZehnder1, HoferWysockiZehnder2}
  and later got generalized to sc-retracts in \cite{Hofer,HWZscSmooth,hwzbook}.
Since this notion of splicing will likely no longer be used, we allow
  ourselves to change the notation and restrict to a further special case
  (using a finite dimensional parameter space $V$).
All sc-retractions relevant for Morse theory and holomorphic curve
  moduli spaces can be put into this setup of ``splicing with finitely
  many gluing parameters,'' which is also helpful for developing a
  simplified notion of Fredholm sections; see Section~\ref{ss:fred}.

\begin{definition} \label{def:splicing}
A {\bf sc-smooth splicing} on an sc-Banach space $\E'$ is a family of
  linear projections $\bigl(\pi_v : \E'\to\E'\bigr)_{v\in U}$, 
  which then necessarily satisfy $\pi_v\circ \pi_v = \pi_v$, that
  furthermore are parametrized by an open subset $U\subset\R^d$ in a
  finite dimensional space in such a way that the associated map
  \begin{equation*}
    \pi \,:\; U\times \E' \to \E' , \qquad (v,f) \mapsto \pi_v(f)  
    \end{equation*}
  is sc$^\infty$. 
In particular, each projection restricts to a bounded linear operator
  $\pi_v|_{E'_m} \in L(E'_m,E'_m)$ on each scale, but these may not vary
  continuously in the operator topology with $v\in U$.

The {\bf splicing core} of a splicing $(\pi_v)_{v\in U}$ is the subset
  of $\R^d\times\E'$ given by the images of the projections,
  \begin{equation*}
    K^\pi := \{ (v,e)\in U \times \E' \,|\, \pi_v e=e \} \;=\;
    \bigcup_{v\in U} \{v\}\times \im\pi_v \;\subset\; \R^d\times\E' .
    \end{equation*}
\end{definition}

\begin{remark}\rm
Any sc-smooth splicing $\bigl(\pi_v : \E'\to\E'\bigr)_{v\in U}$
  for $U\subset\R^d$ induces an sc-retraction on $\R^d\times\E'$,
  which is given by the open set $\cU:=U\times \E'$ and the map
  \begin{equation*}
    r_\pi \,:\; U \times \E' \to U\times \E' ,\qquad (v,e) \mapsto
    (v,\pi_v e) .
    \end{equation*}
The image of this retraction is the splicing core $K^\pi = r_\pi(U\times
  \E')$.
\end{remark}

Here we may observe that splicings on a finite dimensional space
  $\E'=(E')_{m\in\N_0}$ have splicing cores that are homeomorphic to
  open subsets in Euclidean spaces because the pointwise continuity
  automatically implies continuity in the operator topology $L(E',E')$, and
  hence the dimension of the images $\pi_v(E')$ must be locally constant.
Thus, the notion of an M-polyfold modeled on open subsets of splicing
  cores in finite dimensional spaces will reproduce the definition of a
  finite dimensional manifold.

We end this subsection by presenting two examples of sc-smooth
  retractions:
Example \ref{ex:trivial_splicing} can also be found in
  \cite{HWZ_lectures} and \cite[Ex.1.22]{HWZscSmooth}.
Although it has exceedingly little to do with polyfolds for moduli
  problems, it does serve as an important visual reminder that -- unlike
  their classical counterparts -- sc-smooth retracts may have locally
  varying dimension and yet simultaneously support an sc-smooth structure.
It also has a fascinating connection to Kuranishi structures.
Example \ref{ex:pregluing_retraction} introduces the retraction which
  can be used in Morse theory to  glue the space of broken trajectories
  to the space of unbroken trajectories.
   
\begin{example}[\bf a ``finite dimensional'' retract]
  \label{ex:trivial_splicing}\rm
Fix a non-negative function  \(\beta\in \mathcal{C}_0^\infty\) for
  which \(\|\beta\|_{E_0} = \|\beta\|_{L^2}=1\).
We consider the sc-Banach space \(\E= \bigl(W_{\delta_k}^{k,
  2}(\mathbb{R}, \mathbb{R}) \bigr)_{k\in\N_0}\) as in Lemma
  \ref{ex:sobolev} with \(\delta_0 = 0\).
Define a family of linear \emph{projections} \( \pi_t: E_0\to E_0\)
  for $t\in\R$ by $L^2$-projection onto the subspace spanned by $\b_t:=
  \beta(e^{1/t}+\cdot)$ for $t>0$ and $\b_t:=0$ for $t\le 0$.
The corresponding retraction 
  \begin{equation*}
    \mathbb{R}\times \mathbb{E} \to \mathbb{R}\times \mathbb{E}, \qquad
    (t, e)\mapsto (t, \pi_t(e)) =
    \begin{cases}
    \bigl( t, \langle f, \beta_t \rangle_{L^2} \beta_t \bigr)&; t>0\\
    (t, 0 ) &; t\leq 0
    \end{cases}
  \end{equation*}
  is sc\(^\infty\) (see \cite[Lemma~1.23]{HWZscSmooth}), and it is a retraction
  (in fact, a splicing).
The sc-retract (i.e.\ the splicing core) is given by
  \begin{equation*}
    \{(t, 0) \,|\, t\leq 0\}\; \cup\; \{(t, s\beta_t) \,|\, t>0,
    s\in \mathbb{R}\},
    \end{equation*}
  which is (in the topology of $\R\times E_0$) homeomorphic to the
  subset of \(\mathbb{R}^2\) given by $(-\infty, 0]\times \{0\}\;
  \cup\; (0, \infty) \times \mathbb{R}$ and depicted in Figure
  \ref{fig:trivial_splicing}.

A similar topological space appears in the theory of Kuranishi
  structures, where a moduli space is covered by finitely many charts
  $\CM=\bigcup_{i=1,\ldots,N} \psi_i(s_i^{-1}(0)/G_i)$, each of which is
  homeomorphic to a finite group quotient of the zero set $s_i^{-1}(0)$
  of a section $s_i:U_i\to E_i$ in a finite dimensional bundle.
Here the regularization approach (simplified to the case of trivial
  isotropy groups $G_i$) is to find compatible perturbations $\nu_i$
  of these sections so that one obtains a compact manifold from the
  resulting quotient space $\bigsqcup_{i=1,\ldots,N} (s_i+\nu_i)^{-1}(0)
  /\sim$ of perturbed zero sets modulo transition maps.
One might hope to achieve the compactness from local compactness of an
  ambient space such as  $\bigsqcup_{i=1,\ldots,N} U_i /\sim$.
However, the basic nontrivial example with domains $U_i$ of varying
  dimensions is given by $U_1=\R$ and $U_2=(0, \infty) \times \mathbb{R}$
  with equivalence relation $U_1\ni x\sim (x,0)\in U_2$ for $x>0$.
The quotient space $(\R \sqcup (0, \infty) \times \mathbb{R})/\sim$
  has a natural bijection with the splicing core $\cK$ obtained above,
  but the natural quotient topology on this space is very different from
  the relative topology on $\cK$ induced from the ambient sc-Banach space.
While both of these spaces fail to be locally compact, $\cK$
  nevertheless carries a natural metric, whereas the Kuranishi quotient space
  fails to be first countable, and thus it cannot be metrizable; see
  \cite[Ex.6.1.14]{mcduff-wehrheim}.
\end{example} 

\begin{figure}
  \includegraphics[scale=.5]{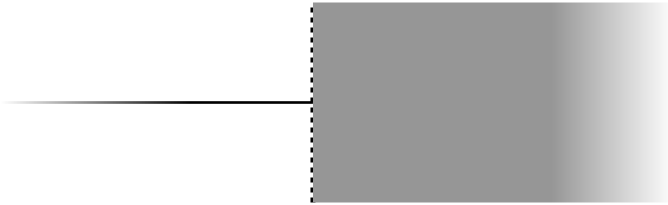}
  \caption{A subset of \(\mathbb{R}^2\) homeomorphic to an sc-smooth
  retract.}
  \label{fig:trivial_splicing}
\end{figure}

\begin{example}[\bf retraction arising from pregluing]
  \label{ex:pregluing_retraction} \rm
Let us more rigorously construct the sc-retract outlined in Section
  \ref{ss:ret}, where we motivated it by the need of a chart that covers
  both broken and unbroken trajectories.
Building on the notation and spaces introduced in Example
  \ref{ex:MorseTrajectorySpaces}, the pregluing and antipregluing maps
  \begin{equation*}
    \oplus:(0, v_0)\times V^\phi\times V^\psi\to \mathcal{B}_a^c
    \qquad\qquad \ominus:(0, v_0)\times V^\phi\times V^\psi\to W^{2,
    2}(\mathbb{R}, X)
  \end{equation*}
  are given by
  \begin{align*}
    \oplus_v(u, w) &:= \beta \cdot \tau\bigl({\textstyle \frac{R_v}{2}},
    u+\phi\bigr) + (1-\beta) \cdot \tau\bigl({\textstyle \frac{-R_v}{2}},
    w+\psi\bigr)\\
    \ominus_v(u, w) &:= (\beta-1)\cdot \tau\bigl({\textstyle
    \frac{R_v}{2}}, u+\phi\bigr)  + \beta \cdot\tau\bigl({\textstyle
    \frac{-R_v}{2}}, w+\psi\bigr),
  \end{align*}
  where \(\beta:\R\to[0,1]\) is a smooth cut-off function with
  \(\beta\big|_{(-\infty, -1]}=1\) and  \(\beta\big|_{[1, \infty)}=0\).
Moreover, we use the gluing profile $\mu:(0,1)\mapsto (0,\infty),
  v\mapsto R_v:=e^{1/v}-e$ restricted to $(0,v_0)\subset(0,1)$ so that the
  antigluing contributions $(\beta-1)\phi(\cdot{\textstyle +\frac{R}{2}})$
  and $\beta\psi(\cdot -{\textstyle \frac{R}{2}})$ vanish for $R>R_{v_0}$.
As in Section \ref{ss:ret}, this gives rise to a retraction $r:
  [0,v_0)\times V^\phi\times V^\psi \to [0,v_0)\times V^\phi\times V^\psi$
  given by
  \begin{equation*}
    r(v, u, w):= 
    \begin{cases}
    \boxplus^{-1}\circ {\rm pr} \circ \boxplus (v, u, w) & \text{if } v> 0,\\
    (v, u, w) & \text{if }v=0,
    \end{cases}
  \end{equation*}
  where \(\boxplus=(\oplus, \ominus)\), and \({\rm pr}\) is the canonical
  projection to the first factor.
For each fixed gluing parameter \(v\in [0,v_0)\), we see that \(r(v,
  \cdot, \cdot)\) is given by the unpleasant formula
 \begin{align*}
  \left(
  \begin{matrix}
  u\\
  w
  \end{matrix}
  \right)
  &\mapsto
  -
  \left(
  \begin{matrix}
  \phi \\ \psi
  \end{matrix}
  \right)
  +\left(
  \begin{matrix} 
  \tau(\frac{-R_v}{2}, \cdot )& 0\\
  0 & \tau(\frac{R_v}{2}, \cdot )  
  \end{matrix}
  \right)
  \left(
  \begin{matrix}
  \beta & 1-\beta\\
  \beta-1 & \beta
  \end{matrix}
  \right)^{-1}
  \left(
  \begin{matrix}
  1 & 0\\
  0 & 0
  \end{matrix}
  \right)\\
  &\qquad\qquad\qquad\qquad
 \cdot\left(
  \begin{matrix}
  \beta & 1-\beta\\
  \beta-1 & \beta
  \end{matrix}
  \right)
  \left(
  \begin{matrix} 
  \tau(\frac{R_v}{2},\cdot) & 0\\
  0 & \tau(\frac{-R_v}{2}, \cdot )
  \end{matrix}
  \right)
  \left(
  \begin{matrix}
  u + \phi\\
  w + \psi
  \end{matrix}
  \right).
  \end{align*}
The upshot of such an unsightly formulation is that it is then elementary
  to show that the map \(r\)
  will be sc-smooth provided that the following two maps are sc-smooth:
  \begin{align}
    &\mathbb{R}\times W^{2, 2}(\mathbb{R}, X) \to W^{2, 2}(\mathbb{R},
    X)\qquad(v, u)\mapsto
    \begin{cases}
    \tau\big(\frac{-R_v}{2}, \tilde{\beta}\big)\cdot u & \text{if } v>
    0 , \label{eq:toolbox1}\\
    u & \text{if } v=0 ,
    \end{cases} \\
    &\mathbb{R}\times W^{2, 2}(\mathbb{R}, X) \to W^{2, 2}(\mathbb{R},
    X)\qquad(v, u)\mapsto
    \begin{cases}
    \tau\big(\frac{R_v}{2}, \hat{\beta} \big) \cdot \tau\big(R_v,
    u  \big)&\text{if } v> 0 , \\
    0&\text{if } v=0,\label{eq:toolbox2}
    \end{cases}
    \end{align}
  where \(\tilde{\beta}\) is a smooth function with support near
  \(\{-\infty\}\) and \(\hat{\beta}\) is a smooth function with compact
  support.
This is essentially the content of \cite[Prop.2.8]{HWZscSmooth};
  consequently the map \(r\) defined above is in fact an sc-smooth
  retraction.
\end{example}

\subsection{Scale calculus for sc-retracts}\label{ss:rcalc}
Sc-retracts and splicing cores are naturally equipped with the
  sc-topology induced from the ambient sc-Banach space, so we already
  have a well-defined notion of scale continuous maps between them.
Moving towards the notion of scale smooth maps between sc-retracts, we
  next note that, somewhat surprisingly, sc-retracts have a well-defined
  notion of a tangent bundle.
Indeed, observe that since $r\circ r = r$, it follows by the chain
  rule that the associated tangent map $\rT r: \rT\mathcal{U} \to
  \rT\mathcal{U}$ satisfies $\rT r \circ \rT r = \rT r$ on the open subset
  $\rT \cU := (E_1\cap \cU) \times E_0 \subset \rT \E$ of the sc-tangent
  bundle $\rT\E=(E_k\times E_{k+1})_{k\in\N_0}$.
In other words $\rT r$ is an sc-retraction.
Consequently, we simply \emph{define} the sc-tangent bundle of a retract
  as the image of an associated sc-retraction.

\begin{definition}\label{def:Tretract}
The {\bf sc-tangent bundle} of an sc-retract $\cO\subset\E$ is the image
  $\rT\cO:= \rT r (\rT \cU) \subset \rT\E$ of the tangent map for any
  choice of retraction \(r:\mathcal{U}\to\mathcal{U}\subset\mathbb{E}\)
  with $r(\cU)=\cO$.
In particular, its fibers are the tangent spaces\footnote
  {
  Here we used the fact that each differential $\rD_p r : \E \to \E$ at
    $p\in\cO \cap E_1$ is a retraction as well, and, since it is
    linear, it is a projection whose image $\im \rD_p r = \ker(\id_{E_0}
    - \rD_p r)$ is the kernel of the complementary projection $\id_{E_0}
    - \rD_p r$.
  } 
  at $p\in\cO\cap E_1$, 
  $$
    \rT_p\cO \,:=\; \rT r (\{p\}\times E_0) \;=\; \{p\}\times \im \rD_p
    r  \;\subset\; \{p\}\times E_0.
  $$
\end{definition}

Of course, at first the definition of sc-tangent bundle looks entirely
  ad hoc, however it reproduces Definition~\ref{def:scTangentBundle} of
  $\rT\E=E_1\times E_0$ (arising from the retraction $r={\rm id}_\E$),
  it is generally well defined, and it coincides with the
  tangents of paths in the retract as follows; see \cite[Prop.2.4,
  Lemma~2.29]{HWZnew}.
\begin{lemma}\label{lem:TangentRetract}
Let \(r:\mathcal{U}\to\mathcal{U}\subset\mathbb{E}\) be an sc-retraction
  with $r(\cU)=\cO$.
\begin{enumerate}
  \item
  Let \(r':\mathcal{U}'\to\mathcal{U}'\subset\mathbb{E}\) be another
    sc-retraction with $r'(\cU')=\cO$. 
  Then $\rT r(\rT\mathcal{U}) = \rT r'(\rT\mathcal{U}')$, hence $\rT\cO$
    is well defined.
  \item 
  The $E_0$-closure of the set of tangent vectors to scale smooth
    paths in $\cO$ through a given smooth point $p\in \cO\cap E_\infty$
    coincides with the tangent space of the retract at $p$,
    $$
      {\rm cl}_{E_0} \bigl\{ \bigl(\g(0),\g'(0)\bigr) \,\big|\,
      \gamma:(-\epsilon,\epsilon) \to \E \;\; \text{sc}^\infty,
      \g((-\eps,\eps))\subset\cO , \g(0)=p \bigr\} \;=\; \rT_p \cO .
    $$
  \end{enumerate}
\end{lemma}

Guided by this notion (but not explicitly using it), the notions of
  scale differentiability and scale smoothness for maps between open
  subsets of sc-Banach spaces can be generalized to sc-retracts.
This notion will in particular be used in the compatibility condition on
  the transition maps between different M-polyfold charts $\phi_i: U_i \to
  \cO_i$ for $i=1,2$ with overlap $\cX\supset U_1\cap U_2 \neq \emptyset$.
Here $\cO_i\subset\E_i$ are sc-retracts in possibly different sc-Banach
  spaces, so we need a notion of scale smoothness of the transition map
  $$
    \phi_2 \circ \phi_1^{-1} \,:\; \cO_1\supset \phi_1 (U_1\cap U_2)
    \;\longrightarrow\; \cO_2 .
  $$
Since $\phi_1$ is a homeomorphism, it maps the overlap $\phi_1 (U_1\cap
  U_2)\subset \cO_1$ to an open subset of the sc-retract $\cO_1=r_1(\cU_1)$
  given by some choice of retraction $r_1:\cU_1\to\E_1$.
Since the latter is continuous, its preimage $\cU_{12}:=
  r_1^{-1}\bigl(\phi_1 (U_1\cap U_2)\bigr)\subset \E_1$ is open, so
  the retraction $r_1|_{\cU_{12}}$ is an sc-retraction on $\E_1$
  with image $r_1(\cU_{12})=\phi_1 (U_1\cap U_2)$.
Thus it remains to define the notion of scale smoothness for maps
  between sc-retracts in different sc-Banach spaces.

\begin{definition}\label{def:sc smooth map from retract}
Let $f:\cO \to \cR$ be a map between sc-retracts $\cO\subset \E$ and
  $\cR\subset\F$, and let $\io_\cR:\cR\to\cF$ denote the inclusion map.
Then we say that $f$ is sc$^k$ for $k\in\N$ or $k=\infty$ if $\io_\cR
  \circ f\circ r : \cU \to \F$ is sc$^k$ for some choice of sc-retraction
  \(r:\mathcal{U}\to\mathcal{U}\subset\mathbb{E}\) with $r(\cU)=\cO$.
In particular, a bijection $f:\cO \to \cR$ is called {\bf
  sc-diffeomorphism} if both $f$ and $f^{-1}$ are sc$^\infty$.
\end{definition}

The definition of the regularity of a map $f:\cO\to \cR$ is independent
  of the choice of the sc-retraction with $r(\cU)=\cO$ by the following
  lemma. 
We provide a proof of this result since it seems so unlikely and yet
  elementary.

\begin{lemma}\label{lem:scRegIndepOfr}
Let $f:\cO \to \cR$ be a map between
  sc-retracts $\cO\subset \E$ and $\cR\subset\F$, let
  \(r_i:\mathcal{U}_i\to\mathcal{U}_i\subset\mathbb{E}\) for $i=1,2$ be
  two retractions with $r_i(\cU_i)=\cO$, and set $k\in\N_0$ or $k=\infty$.
Then $\io_\cR \circ f\circ r :\cU \to \F$ is sc$^k$ if an only if
  $\io_\cR \circ f\circ r':\cU' \to \F$ is sc$^k$.
\end{lemma}

\begin{proof}
Since $\cO\subset\cU\cap\cU'$ is the fixed point set of both $r$
  and $r'$, we have the identities $r'\circ r = r$ on $\cU$ as well as
  $r\circ r' = r'$ on $\cU'$.
Thus we have $\io_\cR \circ f\circ r = \io_\cR \circ f\circ r' \circ r$,
  so that the sc$^k$ regularity of $\io_\cR \circ f\circ r' $ implies that
  of $\io_\cR \circ f\circ r$ by the chain rule theorem~\ref{thm:ChainRule}
  for composition with the sc$^\infty$ map $r$.
The reverse implication holds analogously.
\end{proof}

\begin{example}[{\bf M-polyfold charts and transition maps in Morse
  theory}]\label{ex:localModelInMorseTheory}\rm
In Example \ref{ex:pregluing_retraction}, we constructed a retraction,
  which arises in Morse theory from the pregluing map \(\oplus\).
We now build on that example, and we indicate how such retracts provide
  local models for the space of broken and unbroken
  trajectories $\cX_a^c = \widetilde{\mathcal{B}}_a^c \sqcup
  \widetilde{\mathcal{B}}_a^b\times\widetilde{\mathcal{B}}_b^c$ defined
  in Example~\ref{ex:TrajectoryTopology}.
Recall that \(\widetilde{\mathcal{B}}_a^c \) and
  \(\widetilde{\mathcal{B}}_a^b\times\widetilde{\mathcal{B}}_b^c
  \) were given the structure of an sc-manifold in
Example~\ref{ex:MorseTrajectorySpaces}.
Using the previous notation, the local charts are given by
  $$ 
  \Phi:V^\phi\to \Ti\cB_a^b,\quad u\mapsto [\phi+u]
  \qquad\text{and}\qquad
  \Psi:V^\psi\to \Ti\cB_b^c,\quad v\mapsto [\psi+w] .
  $$
To obtain a local chart centered at a broken trajectory
  $([\phi],[\psi])$, we use pregluing, as in Example
  \ref{ex:pregluing_retraction}, to obtain a retraction \(r^{\phi, \psi}:[0,
  1)\times V^\phi\times V^\psi\to [0, 1)\times V^\phi\times V^\psi\),
  whose image is an sc-retract \(\mathcal{O}^{\phi, \psi}\).
Then an M-polyfold chart for $\cX_a^c$ is given by 
  \begin{equation*}
    \Xi: \mathcal{O}^{\phi, \psi}\to \cX,\qquad
    \Xi(v, u, w)=
    \begin{cases}
    [\oplus_{R_v}(u + \phi, w+\psi)]&\text{if } v\neq 0\\
    ([u + \phi], [w+\psi])&\text{if } v=0.
    \end{cases}
  \end{equation*} 
The restricted maps \(\Xi:\mathcal{O}^{\phi, \psi}\cap \{v=0\}\to
  \widetilde{\mathcal{B}}_a^b\times\widetilde{\mathcal{B}}_b^c\)
  and \(\Xi:\mathcal{O}^{\phi, \psi}\cap\{v\neq 0\}\to
  \widetilde{\mathcal{B}}_a^c\) are in fact sc-diffeomorphisms.
In particular one can check injectivity with respect to $v$ by
  observing that $R_v$ governs the distance between the intersection
  points of $\oplus_{R_v}(u + \phi, w+\psi)$ with the hyperplanes $H^\phi$
  and $H^\psi$.
In order to show that the sc-manifold charts for
  \(\widetilde{\mathcal{B}}_a^c\), together with charts \((\Xi,
  \mathcal{O}^{\phi, \psi})\) arising from pregluing, indeed yield an
  M-polyfold structure for $\cX_a^c$, we must verify that the induced
  transition maps  are sc-smooth.
To that end, we can write, for example, the transition map between two
  pregluing charts \(\Xi'\,\!^{-1}\circ \Xi: \mathcal{O}^{\phi, \psi}\to
  \mathcal{O}^{\phi', \psi'}\), where they are defined, as
  $$
    \big( \oplus_{\mu^{-1}(\mu(v) - s(u) +
    t(w))}\big)^{-1}\Big(\tau\big({\textstyle \frac{s(u)+t(w)}{2}},
    \oplus_v(u + \phi, w +\psi)\big)\Big).
  $$
Here $\mu:(0,1)\mapsto (0,\infty)$ is the gluing profile, \(\tau\)
  is the translation map (\ref{2tau}), and the functions \(u\mapsto
  s(u)\),  \(w\mapsto t(w)\) are determined by the equation \((u +
  \phi)(s(u))\in H^{\phi'}\) and \((w + \psi)(t(w))\in H^{\psi'}\),
  where \(H^{\phi'},H^{\psi'}\subset X\) are the hyperplanes used as
  slicing conditions, as in Example \ref{ex:MorseTrajectorySpaces}.
After expanding this expression, one can see that the
  sc-smoothness of the transition map \(\Xi'\,\!^{-1}\circ \Xi\)
  follows from the sc-smoothness of the functions \(s,t\), proven
  as in Section~\ref{ss:scmfd}, and maps \eqref{eq:toolbox1},
  \eqref{eq:toolbox2}.
Compatibility of pregluing charts with ``interior charts'' for
  $\Ti\cB_a^c$ is checked similarly, so that one indeed obtains an
  M-polyfold structure on \(\cX_a^c\).
\end{example}

\subsection{M-polyfolds with boundaries and corners} \label{ss:corners}

The notion of M-polyfolds with boundary and corners is central for
  applications.
For instance, in Morse theory the broken trajectories form the boundary
  of an M-polyfold whose interior are the unbroken trajectories.
More precisely, the once broken trajectories are the smooth part of
  the boundary (the codimension $1$ part of the boundary strata), and
  the $k$-fold broken trajectories are the codimension $k$ part of the
  boundary strata; here, corners are understood as $k\ge 2$.
We will develop this notion by introducing boundaries and corners into
  the notions of sc-retracts (where it requires a nontrivial modification
  to allow for an implicit function theorem later on) and then introducing
  sc-smoothness, following Remark~\ref{rmk:corners}.
We begin by considering a special case of the notion of a partial
  quadrant,\footnote
  {
    For a general definition of partial quadrants, see
    \cite{HoferWysockiZehnder1}.
  } 
  which we call an sc-sector, and we introduce the degeneracy index which
  will be used to define the boundary and corner strata.

\begin{definition} \label{def:DegenIndex}
A {\bf sc-sector} $C$ is the subset $C=[0,\infty)^k \times \E \subset
  \R^k\times \E$ in the product of a finite dimensional space $\R^k$
  and an sc-Banach space $\E$.
Its {\bf degeneracy index} $d_C : C \to \N_0$ is given by counting the
  number of coordinates in $\R^k$ that equal to $0$; in other words,
  $$
    d_C\bigl( (x_i)_{i=1,\ldots,k},e\bigr)  = \# \bigl\{ i \in
    \{1,\ldots, k\} \,|\, x_i = 0 \bigr\} .
  $$
\end{definition}

\begin{remark} [{\bf Degeneracy index vs.\ gluing parameters}]
  \label{rk:sectors and degeneracy} \rm
In practice, sc-sectors are usually of the form \([0, \infty)^k\times
  \mathbb{R}^\ell \times\mathbb{E}\), where $\E$ is a function space and
  the first two factors are gluing parameters.
For example, for charts near a once-broken Morse-trajectory we would
  have \(k=1\) and \(\ell=0\); near a twice-broken Morse-trajectory we
  would have \(k=2\) and \(\ell=0\).
In this way, we think of the degeneracy index as a means of measuring
  in which ``corner-stratum'' a point lies: a point with degeneracy index
  of zero, one, or two is respectively an interior point, boundary point,
  or corner point.
However, the degeneracy index does not necessarily measure the number of
  regular components of a curve or trajectory (whose domains are smooth,
  connected Riemann surfaces).
For instance, near a nodal curve (or cusp curve) in Gromov-Witten theory,
  the pregluing construction involves two shift parameters $(R,\theta)
  \in (R_0,\infty)\times S^1$.
These can be encoded in a single complex gluing parameter $c\approx
  0 \in \C$ by $R=e^{1/|c|}$ and $\theta=\arg(c)$, which is naturally
  extended by $c=0\in\C$ corresponding to the nodal curves.
Hence a chart near a curve with one nodal point will involve an sc-sector
  with \(k=0\) and \(\ell=2\), and near a curve with two nodal points the
  sc-sector has \(k=0\) and \(\ell=4\); that is, all of these sc-sectors
  are in fact sc-Banach spaces.
This indicates the important point that \emph{nodal curves in
  Gromov-Witten theory have degeneracy index zero;} in other words, all
  such nodal curves are \emph{interior} points of the ambient M-polyfold
  as well as the regularized moduli space.
\end{remark}

Unfortunately, scale smooth bijections between open subsets of sc-sectors
  do not generally preserve the degeneracy index.
However, the following refined notion of an sc-retract in an sc-sector
  will guarantee ``corner recognition'' as stated in the subsequent
  theorem.
First, however, we need to introduce the notion of direct sums in
  sc-Banach spaces.

\begin{definition} 
Let $\E$ be an sc-Banach space.
Two linear subspaces $X,Y \subset E_0$ split $\E$ as a {\bf sc-direct
  sum} $\E = X \oplus_{sc} Y$ if
  \begin{enumerate}
    \item 
    both $X,Y\subset E_0$ are closed and $(X\cap E_m)_{m\in\N_0},
      (Y\cap E_m)_{m\in\N_0}$ are scale Banach spaces;
    \item
    on every level $m\in\N_0$ we have the direct sum $E_m = (X\cap E_m)
      \oplus (Y\cap E_m)$.
    \end{enumerate}
We call $Y$ the sc-complement of $X$.
\end{definition}

\begin{definition}\label{def:neat}
Let $\cU\subset [0,\infty)^k\times\E$ be a relatively open set in
  an sc-sector.
Then $r:\cU\to \cU$ is a {\bf neat sc-retraction} if it satisfies
  $r\circ r = r$ and the following regularity and neatness conditions.
\begin{enumerate}
  \item 
  $r$ is sc$^\infty$; that is, the restriction $r|_{\cU^{\rm
    int}}$ to the open subset $\cU^{\rm int}:=\cU\cap
    (0,\infty)^k\times\E\subset \R^k\times\E$ is sc$^\infty$ in the sense
    of Definition~\ref{def:scDiff}, and the iterated tangent map $\rT^\ell
    r$ on $\rT\ldots\rT\cU^{\rm int} = \bigl(\cU \cap (0,\infty)^k\times
    E_\ell \bigr) \times \diamondsuit$ extends\footnote
    {
      Here $\diamondsuit$ is a complicated product of sc-Banach spaces,
	arising from iterating Definition~\ref{def:scTangentBundle}
	of the sc-tangent bundle.
      For example, $\diamondsuit$ is trivial for $\ell=0$, for $\ell=1$
	we have $\diamondsuit = \R^k \times E_1$, and for $\ell=2$ it
	is $\diamondsuit = \R^k \times E_1 \times \R^k \times E_1 \times
	\R^k \times E_0$.
      The point is that an extension to the boundary only appears in
	the first factor.
    }
    to an sc$^0$ map on $\rT\ldots\rT\cU :=  \bigl(\cU \cap
    [0,\infty)^k\times E_\ell\bigr) \times \diamondsuit$ for all
    $\ell\in\N_0$.
  \item 
  For every ``smooth point'' $p\in r(\cU)\cap (\R^k \times E_\infty)$ in
    the retract, the tangent space $\rT_p \cO \cong\im \rD_p r \subset
    \R^k\times \E$ is {\bf sc-neat with respect to the sc-sector}
    $[0,\infty)^k\times\E$, that is it has an sc-complement $Y\subset
    \{0\}\times \E$ so that $\R^k\times \E = \im\rD_p r \oplus Y$.
  \item 
  Every point in the retract $p\in r(\cU)$ has an approximating sequence
    $p_n\to p$ of ``smooth points'' $(p_n)_{n\in \mathbb{N}}\subset
    r(\cU)\cap E_\infty$ in the same corner stratum, that is with
    $d_C(p_n)=d_C(p)$.
  \end{enumerate}

A {\bf sc-retract with corners} in the sc-sector $[0,\infty)^k\times\E$
  is a subset $\cO\subset[0,\infty)^k\times\E$ that is the
  image $r(\cU)=\cO$ of a neat sc-retraction $r:\cU\to\cU\subset
  [0,\infty)^k\times\E$.
\end{definition}

The neatness condition is phrased by HWZ as having a sc-complement
  $Y\subset C$ in the partial quadrant $C$.
For the sc-sector $C=[0,\infty)^k\times \E$ this is equivalent to
  $Y\subset\{0\}\times\E$ and implies that $\im \rD_p r$ projects
  surjectively to the $\R^k$ factor.
It is our understanding that this condition a weaker notion of neatness
  that has been introduced in \cite{HWZnew} and is still sufficient for
  regularization.

The neatness conditions (ii) and (iii) were added in the generalization
  from splicings to retracts, since splicings satisfy them automatically,
  as we show in the following.

\begin{remark} \label{rmk:SplicingCorners} \rm
An {\bf sc-splicing with corners} is a family of linear
  projections $\bigl(\pi_v : \E'\to\E'\bigr)_{v\in U}$ as in
  Definition~\ref{def:splicing}, with the exception that we allow splicings
  parametrized by open subsets $U\subset [0,\infty)^k\times\R^{d-k}$
  in finite dimensional sectors.
The corresponding sc-retraction $r_\pi: U \times \E' \to U\times
  \E'$, $(v,e) \mapsto (v,\pi_v e)$ then is a neat sc-retraction on
  $[0,\infty)^k\times\R^{d-k}\times\E'$, as can be seen by checking
  conditions (ii) and (iii).
\begin{itemlist}
  \item[(ii)]
  The ``smooth points'' are $(v,e) \in U \times E'_\infty$, and the
    differential of the retraction is $\rD_{(v,e)} r_\pi : (X,Y) \mapsto
    \bigl(X, \rD_{(v,e)} \pi (X,Y)\bigr)$, so that the tangent space to
    the retract $\cO=\im r_\pi$ at $(v,e=\pi_v e)$ is
    $$
      \rT_{(v,e)} \cO \;=\; \im \rD_{(v,e)} r_\pi \;=\; \bigl( X,
      \rD_{(v,e)} \pi (X,0) + \pi_v Y \bigr).
    $$
  We claim that it has an sc-complement $\R^d\times\E'=\im \rD_{(v,e)}
    r_\pi \oplus \im L$ given by the image of the sc$^0$ operator
    $L:\R^d\times\E' \to \R^d\times \E'$, $(X,Y)\mapsto\bigl( 0 , Y- \pi_vY
    \bigr)$, which is contained in $\{0_{\R^k}\}\times \R^{d-k}\times E_1$
    (in fact, in $\{0_{\R^d}\}\times E_1$).
  Indeed, the decomposition is given by an sc$^0$ isomorphism where
    we abbreviate $Z_{X,Y} = Y - \rD_{(v,e)} \pi (X,0)$,
    \begin{align*}
      \R^d\times \E' &\;\longrightarrow\; \im \rD_{(v,e)} r_\pi \times
      \im L \\
      (X,Y) &\;\longmapsto\; \bigl( (X , \rD_{(v,e)} \pi (X,0) + \pi_v
      Z_{X,Y}) , (0, (\id - \pi_v) Z_{X,Y}) \bigr).
    \end{align*}
  \item[(iii)] 
  For any point in the splicing core $(v,e)\in K^\pi$ we obtain a
    ``smooth'' approximating sequence by picking $E'_\infty \ni e_i
    \to e$, since then $(v,\pi_v(e_i))\to (v,\pi_v(e)) = (v,e)$, and
    the degeneracy index is preserved since it is determined by $v\in
    [0,\infty)^k\times\R^{d-k}$.
  \end{itemlist}
\end{remark}

Observe that, if given an sc-retract (or splicing core) with corners
  $\cO\subset [0,\infty)^k\times\E$, we can restrict the degeneracy index
  from the ambient sector (where $\E=\R^{d-k}\times\E'$ in the case of
  a splicing) to a well defined map $d_\cO : \cO \to \N_0$.
That this is well defined also under ``sc$^\infty$ diffeomorphisms''
  between retracts is proven in \cite[Prop.2.24]{HWZnew}.

\begin{proposition} \label{prop:CornerRecognition}
Let $f:\cO \to \cO'$ be an sc$^\infty$ diffeomorphism between open
  subsets of splicing cores with corners -- that is a sc$^\infty$ bijection
  with sc$^\infty$ inverse $f^{-1}:\cO' \to \cO$.
Then it intertwines the degeneracy indices; in other words $d_\cO =
  d_{\cO'}\circ f$.
\end{proposition}

With this language in place, we define the notion of an M-polyfold with
  boundary and corners in more technical detail than previously outlined.
Here $I$ can be any index set.

\begin{definition} \label{def:MpolyfoldCorners}
An {\bf M-polyfold with corners} is a second countable and metrizable
  space $\cX$ together with an open covering $\cX=\bigcup_{i\in I} U_i$
  by the images under homeomorphisms $\phi_i:U_i\to\cO_i$ from sc-retracts
  with boundary and corners $\cO_i\subset[0,\infty)^{k_i}\times\E_i$.
These chart maps are required to be compatible in the sense that
  the transition map is sc$^\infty$ for any $i,j\in I$ with $U_i\cap
  U_j\neq\emptyset$; in other words, this requires sc$^\infty$ regularity
  of the map
  $$
    \io_j \circ \phi_j \circ \phi_i^{-1} \circ r_i \,:\;
    [0,\infty)^{k_i} \times \E_i \;\supset\;  r_i^{-1} \bigl( \phi_i
    (U_i\cap U_j) \bigr) \;\longrightarrow\; [0,\infty)^{k_j} \times
    \E_j ,
  $$
  where $r_i$ is any sc-retract with boundary on  $[0,\infty)^{k_i}
  \times \E_i$ with image $\cO_i$.

An {\bf M-polyfold with corners modeled on sc-Hilbert spaces}
  is a metrizable space with compatible charts as above, such
  that each $\E_i$ is an sc-Hilbert space in the sense of
  Definition~\ref{def:scBanachSpace}.
\end{definition}

Taking $k_i=0$ for all $i\in I$ in the above definition reproduces the
  notion of an M-polyfold without boundary.
Restricting to $k_i=0$ or $1$ provides the definition of an M-polyfold
  with boundary (but no corners).
Unfortunately, such a notion of ``cornerless'' M-polyfold is
  not applicable to general moduli spaces of Morse trajectories or
  pseudoholomorphic curves with Lagrangian boundary values, even if their
  ``expected dimension'' does not allow for corners.
This is because the M-polyfold must contain all -- however nongeneric --
  unperturbed solutions.

Due to Proposition~\ref{prop:CornerRecognition} and the sc$^0$ regularity
  of transition maps, we now obtain two stratifications of an M-polyfold
  with corners.
Neither of these will be a stratification in the sense of Whitney;
  they are just sequences of subsets of $\cX$.
To obtain a stratification by ``regularity'' we denote the scales of
  the sc-Banach spaces $\E_i$ in the domain of the chart maps $\phi_i$ by
  $\E_i=(E_{i,m})_{m\in\N_0}$, and the dense subset by $E_{i,\infty}\subset
  E_{i,m}$.

\begin{definition} \label{def:MpolyfoldStrata}
Let $\cX$ be an M-polyfold with corners.
For $k\in\N_0$ the $k$-th {\bf corner stratum} $\cX^{(k)}\subset
  \cX$ is the set of all $x\in \cX$ such that in some chart
  $d_{\cO_i}(\phi_i(x))=k$.

For $m\in\N_0$  the $m$-th {\bf regularity stratum} $\cX_{m}\subset
  \cX$ is the set of all $x\in \cX$ such that for some chart (and hence
  for all charts) we have $\phi_i(x)\in [0,\infty)^{k_i}\times E_{i,m}$.
In particular, the {\bf smooth points} of $\cX$ are those $x\in\cX$
  with $\phi_i(x)\in [0,\infty)^{k_i}\times E_{i,\infty}$ for all charts;
  in other words, the smooth points are those in the intersection
  $\bigcap_{m\in\N_0} \cX_m$.
  
\end{definition}

Observe that ``corner strata''  are disjoint, with one dense stratum,
  whereas the ``regularity strata'' are nested and all dense in $\cX$.

\begin{example}[\bf corner and regularity strata in Morse theory]\rm
To see examples of the above strata in an M-polyfold, we again consider
  the Morse trajectory spaces of Example \ref{ex:MorseTrajectorySpaces}.
Using notation of Definition \ref{def:MpolyfoldStrata} we see that
  the \(m\)-th regularity stratum of \(\cX=\cX_a^c\), denoted \(\cX_m\),
  is given by union of two sets:
\begin{enumerate}
  \item 
  equivalence classes of the form \([\chi + u^\chi]\in
    \widetilde{\mathcal{B}}_a^c\) for which \(\chi\in \mathcal{C}^\infty\)
    is constant outside of a compact domain and \(u^\chi\in
    W_{\delta_m}^{m+2, 2}\),
  \item 
  pairs of equivalence classses of the form \(([\phi + u^\phi],
    [\psi + u^\psi])\in \widetilde{\mathcal{B}}_a^b\times
    \widetilde{\mathcal{B}}_b^c\) for which \(\phi, \psi\in
    \mathcal{C}^\infty\) are constant outside of a compact domain and
    \(u^\phi, u^\psi\in W_{\delta_m}^{m+2, 2}\).
\end{enumerate}
This demonstrates that the regularity strata are determined by the degree
  of differentiability (i.e.\ regularity) of the maps (or pairs thereof)
  representing points in our M-polyfold.
This is further justification for calling the infinity level the space of
  ``smooth points''.

To identify the corner strata in our Morse theory example, we employ
  Example \ref{ex:localModelInMorseTheory} which provides local models
  and shows that
  \begin{equation*}
    \cX^{(0)} =  \widetilde{\mathcal{B}}_a^c \qquad\text{and}\qquad
    \cX^{(1)} =  \widetilde{\mathcal{B}}_a^b \times
    \widetilde{\mathcal{B}}_b^c.
    \end{equation*}
If the Morse function had additional critical points, say \(d\in
  \mathbb{R}^n\) with \(f(a)< f(b)< f(c) < f(d)\), then one could build
  an M-polyfold $\cX=\cX_a^d$  which contains all broken and unbroken
  trajectories between \(a\) and \(d\). Its corner strata would be given by
  \begin{align*}
    \cX^{(0)} = \widetilde{\mathcal{B}}_a^d, \qquad\quad
    \cX^{(1)} =  \widetilde{\mathcal{B}}_a^b\times
    \widetilde{\mathcal{B}}_b^d \;\sqcup\;
    \widetilde{\mathcal{B}}_a^c\times \widetilde{\mathcal{B}}_c^d,
    \qquad\quad \cX^{(2)} = \widetilde{\mathcal{B}}_a^b\times
    \widetilde{\mathcal{B}}_a^b \times\widetilde{\mathcal{B}}_a^c.
    \end{align*}
As before, the unbroken trajectories comprise the ``interior points''
  \(\cX^{(0)}\), and the once broken trajectories comprise the
  ``boundary points'' \(\cX^{(1)}\) essentially because there exist
  local charts given by pregluing maps of the form given in Example
  \ref{ex:localModelInMorseTheory} which attach each of the spaces
  \(\widetilde{\mathcal{B}}_a^b\times \widetilde{\mathcal{B}}_b^d\) and
  \(\widetilde{\mathcal{B}}_a^c\times \widetilde{\mathcal{B}}_c^d \) to
  \(\widetilde{\mathcal{B}}_a^d\) using a single gluing parameter \(v\in
  [0, 1)\).
To establish that \(\cX^{(2)} = \widetilde{\mathcal{B}}_a^b\times
  \widetilde{\mathcal{B}}_a^b \times\widetilde{\mathcal{B}}_a^c \) one must
  construct an sc-retract on \([0, 1)\times [0, 1) \times W^{2, 2} \times
  W^{2, 2}\times W^{2, 2}\) and a pregluing map \((v_1, v_1, u_a, u_b,
  u_c)\mapsto \oplus_{R_{v_1}, R_{v_1}}(u_a, u_b, u_c)\) which attaches the
  twice broken trajectories to the once broken and unbroken trajectories.
By doing so, one shows that the twice broken trajectories are ``corner
  points'' in \(\cX^{(2)}\).

Finally, we note that it is tempting to think of the corner stratum as
  measuring complexity of broken or nodal objects (for example,
  as a count of number of components, or as a count of the number
  of non-vanishing gluing parameters needed to construct a smooth,
  i.e.\ non-nodal and unbroken, map or trajectory), however this is
  completely incorrect.
Indeed, as mentioned in Remark~\ref{rk:sectors and degeneracy}, the
  closed curves arising in Gromov-Witten theory may have many nodal
  components, which then requires many gluing parameters to be
  attached to the space of non-nodal curves; however each of these
  gluing parameters lies in an open disk rather than in a neighborhood
  of $0$ in  \([0, 1)\) (or more generally  \([0, 1)^k\)).
Consequently, all nodal curves in Gromov-Witten theory have degeneracy
  index zero, or equivalently all boundary and corner strata are empty.
\end{example}

\section{Strong bundles and Fredholm sections}\label{sec:bundles}

With the notion of scale smoothness and M-polyfolds in place, the
  purpose of this section is to introduce the remaining notions of bundles
  and Fredholm sections that are used in the statement of the polyfold
  regularization theorem.
Recall that this result uses M-polyfolds as ambient spaces and
  associates a unique cobordism class of smooth compact manifolds to each
  suitable Fredholm section.
Here and throughout we will discuss neither isotropy (which requires
  a generalization to groupoids modeled on M-polyfolds with orbifolds as
  perturbed zero sets) nor orientations (which require determinant line
  bundles of the Fredholm sections).
Boundaries and corners are discussed further in
  Remark~\ref{rmk:RegularizationCorners}.
Let us moreover mention that, while we introduce the notion of bundles
  and Fredholm sections in the general framework of retractions, the
  implicit function and regularization theorems are presently published
  only in the more restrictive setting of splicings.
To guide the presentation we begin with the statement and vague
  introduction of the new notions, which will then be made precise in
  a step by step manner in the following sections.

\setcounter{theorem}{0}

\begin{theorem}[\bf Polyfold regularization, \cite{HoferWysockiZehnder2}
  Thm.5.22] \label{thm:PolyfoldRegularization2}
Let ${\rm pr}:\cY\to\cX$ be a strong M-polyfold bundle with corners
  (see Definition \ref{def:sBundle}) modeled on sc-Hilbert spaces,
  and let $s:\cX\to\cY$ be a proper Fredholm section (see Definition
  \ref{def:scFredholmSection}).
Then there exists a class of sc$^+$-sections $\nu:\cX\to\cY$ (see
  Definition \ref{def:sc+}) supported near $s^{-1}(0)$ such that $s+\nu$
  is transverse to the zero section and $(s+\nu)^{-1}(0)$ carries the
  structure of a smooth compact manifold with corners.

Moreover, for any other such perturbation $\nu':\cX\to\cY$ there
  exists a smooth compact cobordism between $(s+\nu')^{-1}(0)$ and
  $(s+\nu)^{-1}(0)$.
\end{theorem}

Some of the notions here can be easily defined by copying the notions
  from classical differential geometry. In particular we introduce a first,
  rather weak, notion of bundle.
  
\begin{definition}\hspace{2mm} \\ \vspace{-5mm}  \label{def:prelim}
\begin{enumerate}
  \item
  A map $f:\cX \to \cY$ between two M-polyfolds is sc$^\infty$ if
    it pulls back to sc$^\infty$ maps $\psi\circ f \circ \phi^{-1}:
    \cO\supset \phi(U\cap V) \to \cR$ in any pair of charts $\phi:\cX
    \supset U\to \cO\subset \E$, $\psi:\cY\supset V \to \cR \subset \F$.
  In particular, a bijection $f:\cX \supset U \to V\subset \cY$ between
    open subsets of M-polyfolds is called {\bf sc-diffeomorphism} if it
    pulls back to sc-diffeomorphisms between open subsets of any pair
    of charts.
  \item
  A {\bf topological M-polyfold bundle} is an sc$^\infty$
    surjection ${\rm pr}:\cY\to\cX$ between two M-polyfolds together
    with a real vector space structure on each fiber $\cY_x:={\rm
    pr}^{-1}(x)\subset\cY$ over $x\in \cX$.
  (That is, each $\cY_x$ is equipped with compatible multiplication by
    $\R$ and addition, and hence a unique zero vector $0_x\in\cY_x$.)
\footnote
  { 
  The analogous classical notion of topological bundle also requires
    local trivializations.
  We avoid this condition here since the polyfold notion of trivializations
    in Definition~\ref{def:Bundle} will not yield identifications of
    the fibers.
  }
  
\item
  A {\bf section} of ${\rm pr}:\cY\to\cX$ is an sc$^\infty$ map
    $s:\cX\to\cY$ such that ${\rm pr}\circ s =\id_\cX$.
  It is called {\bf proper}\footnote
    {
    For applications, for example to Gromov-Witten moduli spaces, one
      should think of $\cX$ as consisting of equivalence classes of
      maps of a fixed homology class.
    The related notion of ``component-properness'' would allow one to
      consider an M-polyfold that contains maps of any homology class,
      where compactness of $s^{-1}(0)$ is only required in each fixed
      connected component.
    } 
  if its zero set $s^{-1}(0)$ is compact in the relative topology
  of $\cX$,
  $$
    s^{-1}(0) := \bigl\{ x\in\cX \,\big|\, s(x) = 0_x \in \cY_x \bigr\}
    \subset\cX .
  $$
  \end{enumerate}
\end{definition}

The notion of an M-polyfold bundle, introduced in Section~\ref{ss:Bundle},
  will be a vast strengthening of this notion of a surjection with linear
  structure on the fiber, in which the local models for the total space
  $\cY$ are generalized splicing cores, and which are given by families
  of projections that are parametrized by the retract; the latter is
  the local model for the base $\cX$.
When it comes to Fredholm theory, the notion of a Fredholm section will
  implicitly require a ``fillability'' property of the local models for the
  bundle  -- namely an even closer relationship between the retractions
  modeling $\cY$ and $\cX$. 
Here the idea is that there is a scale smooth family of isomorphisms
  between the fibers of the complementary splicing modeling $\cY$ and a
  ``normal bundle'' to the retract that models the base $\cX$.
This ensures that the ``virtual vector bundle $\cY_x -\rT_x\cX$''
  has isomorphic fibers so that a nonlinear Fredholm theory is possible.

Furthermore, an M-polyfold bundle is ``strong'' essentially if it allows
  for a dense set of compact sections -- 
  whose linearizations are
  compact operators, which thus can be used to perturb Fredholm sections
  to achieve transversality.
The corresponding sections will be called sc$^+$, and are more formally
  introduced at the end of Section~\ref{ss:Bundle}.
Finally, the notion of a Fredholm section is discussed in
  Section~\ref{ss:fred}, and Section~\ref{ss:reg} gives a more technical
  description of the admissible class of perturbations (which, in particular,
  are required to preserve the compactness of the resulting zero set).

\subsection{M-polyfold bundles} \label{ss:Bundle}

The preliminary notion of a bundle over an M-polyfold in Definition
  \ref{def:prelim}~(ii) is refined by the restriction to the following
  local models. 
These models generalize the classical notion of a local model for a
  Banach bundle, which we recall are trivial bundles over open subsets
  in a Banach space.

\begin{definition} \label{def:BundleRetract}
Let $\cO\subset [0,\infty)^k\times\E$ be an sc-retract with corners in
  the sense of Definition~\ref{def:neat}, and let $\F$ be an sc-Banach
  space.
Then a {\bf sc-bundle retract} over $\cO$ in $\F$ is a family of
  subspaces $(\cR_p \subset \F)_{p\in\cO}$ that are scale smoothly
  parametrized by $p\in\cO$ in the following sense:
There exists a {\bf sc-retraction of bundle type}, 
  \begin{equation}\label{eq:BunRet}
    \cU\times\F \;\longrightarrow\; [0,\infty)^k\times\E\times\F , \qquad
    (v,e,f) \;\longmapsto\; \bigl( r(v,e) , \Pi_{(v,e)} f \bigr) ,
    \end{equation}
  given by a neat sc-retraction $r:\cU\to  [0,\infty)^k\times\E$
  with image $r(\cU)=\cO$ and a family of linear projections
  $\Pi_{(v,e)}:\F\to\F$ that are parametrized by $(v,e)\in\cU$, and whose
  images for $p=(v,e)\in\cO$ are the given subspaces $\Pi_p(\F)=\cR_p$.

To any such retract we associate the {\bf M-polyfold bundle model}
  $$
    \pr_{\cO} \,:\;  \cR = {\textstyle \bigcup_{p\in\cO}} \{p\}\times
    \cR_p \;\longrightarrow\; \cO , \qquad (p,f) \;\longmapsto\; p .
  $$
\end{definition}

Retractions of bundle-type are retractions themselves, and hence support
  sc-calculus as before.
In particular, and also as before, the local model is given by the retract
  and ambient space, whereas the choice of projections $\Pi_{(v,e)}$
  is auxiliary.

\begin{remark}\rm \label{rmk:BunSpli}
Continuing the comparison with the notion of splicings from
  Remark~\ref{rmk:SplicingCorners}, a special case of an sc-bundle retract
  is the splicing core associated to an {\bf sc-bundle splicing}
  $$
    U \times \E' \times \F \;\longrightarrow\; \E'\times\F , \qquad
    (v,e,f) \mapsto (\pi_v e , \Pi_v f)
  $$
  given by two families of projections $\pi_v$ and  $\Pi_v$ on $\E'$
  and $\F$ respectively.
Of interest is the fact that they are parametrized by the same
  open subset $U\subset [0,\infty)^k \times \R^{d-k}$ in a finite
  dimensional sector,  and they are scale smooth in the sense of
  Definitions~\ref{def:splicing} and \ref{def:neat}.
In the notation of \cite{HoferWysockiZehnder1}, these are models for
  M-polyfolds of type 0 in that we do not allow the ``projections in
  the fiber'' $\Pi$ to be parametrized by the splicing core $K^\pi$,
  but just by its gluing parameters $U$.
This appears to be sufficient for applications to Morse theory and
  holomorphic curve moduli spaces.
In this setting, the M-polyfold bundle model 
  $$
    \pr_{K^\pi} \,:\; \bigcup_{v\in U} \{v\}\times \pi_v(\E') \times
    \Pi_{v}(\F) \;\longrightarrow\;  K^\pi = \bigcup_{v\in U} \{v\}\times
    \pi_v(\E')
  $$
  is {\bf fillable} if there exists a family of isomorphisms
  $f^{\scriptscriptstyle C}_v : \ker \pi_v \overset{\cong}{\to} \ker\Pi_v$
  such that $U \times \E' \to \F$, $(v,e) \mapsto f^{\scriptscriptstyle
  C}_v (e - \pi_v e)$ is sc$^\infty$.
\end{remark}

\begin{example}\rm
The construction of a bundle splicing for Morse theory is briefly
  discussed in Section~\ref{ss:ret}.
When the ambient space of the Morse trajectories is $X=\R^n$, then the
  splicing in the fiber is essentially the same as for the base with the
  following modifications: Firstly, the fiber does not require hypersurface
  slicing conditions; secondly, the regularity of functions in the fiber
  is one less than that in the base, so the section $\gamma \mapsto
  (\gamma,\dot\gamma)$ is scale continuous.
Finally, the maps in the fiber converge to $0$ on both ends.
\end{example}

Now we can refine the notion of a topological M-polyfold bundle from
  Definition~\ref{def:prelim}~(ii) by requiring the bundle to be locally
  sc-diffeomorphic to an M-polyfold bundle model.

\begin{definition}\label{def:Bundle}
An {\bf M-polyfold bundle} is an sc$^\infty$ surjection ${\rm
  pr}:\cY\to\cX$ between two M-polyfolds together with a real vector space
  structure on each fiber $\cY_x:={\rm pr}^{-1}(x)\subset\cY$ over $x\in
  \cX$ such that, for a sufficiently small neighborhood $U\subset\cX$
  of any point in $\cX$, there exists a {\bf local sc-trivialization}
  $\Phi: \cY \supset {\rm pr}^{-1}(U) \to \cR$.
The latter is an sc$^\infty$ diffeomorphism to an sc-bundle retract
  $\cR = \bigcup_{p\in\cO} \{p\}\times \cR_p \subset \E\times\F$ that
  covers an M-polyfold chart $\phi:U\to\cO\subset\E$ in the sense that
  $\pr_\cO \circ\Phi = \phi\circ {\rm pr}$, and preserves the linear
  structure in the sense that $\Phi|_{\cY_x} :\cY_x \to \{\phi(x)\}
  \times \cR_{\phi(x)}$ is an isomorphism in every fiber over $x\in U$.
\end{definition}

To obtain a good set of perturbations for Fredholm sections, we refine
  this notion further by requiring the existence of a ``subbundle of higher
  regularity'', analogous to the fibers $W^{1,p}(S^2,u^*\rT M)$ $\subset
  L^p(S^2,u^*\rT M)$ of a bundle over $W^{1,p}$-regular maps $u:S^2 \to M$.
These ``higher regularity fibers'' will be the target spaces for
  ``lower order perturbations'' of the section -- in this case the
  Cauchy-Riemann operator 
$\pbar: W^{1,p}(S^2,M) \to\bigcup_u  \{u\} \times L^p(S^2,\Lambda^{0,1}\otimes u^*\rT M)$.
In \cite{HoferWysockiZehnder1} this is formalized by introducing double
  filtrations and new notions of scale smoothness with respect to
  these scales.
We have chosen a more minimalist, yet equivalent, route.
Note here that in our notation, one should think of the ambient space
  for the base retract $\E$ and the ambient space for the fibers $\F$ as
  sc-Banach spaces such as $\E=\bigl(W^{1+m}(S^2,\C^n)\bigr)_{m\in\N_0}$
  and $\F=\bigl(W^{m}(S^2,\C^n)\bigr)_{m\in\N_0}$ whose scales are shifted
  by the order of the differential operator that we wish to encode as
  section of the bundle.
For that purpose we introduce the notation $\F_1:=(F_{m+1})_{m\in\N_0}$
  for the scale structure induced by $\F$ on its subspace $F_1$ as
  mentioned in Remark~\ref{rmk:sc-Banach}.
   
\begin{definition}\label{def:sBundle}
An M-polyfold bundle ${\rm pr}:\cY\to\cX$ is called {\bf strong} if it
  has trivializations in strong M-polyfold bundle models that are strongly
  compatible in the following sense.
\begin{enumlist}
  \item[(i)]
  A {\bf strong sc-retraction of bundle type} is a retraction $R:
    \cU\times\F \to [0,\infty)^k\times\E\times\F$, $(v,e,f)\mapsto
    \bigl( r(v,e) , \Pi_{(v,e)} f \bigr)$ as in \eqref{eq:BunRet}
    that restricts to an sc$^\infty$ map $\cU\times\F_1 \to
    [0,\infty)^k\times\E\times\F_1$, i.e.\ a retraction in the sc-Banach
    space $\bigl(\R^k\times E_m\times F_{m+1}\bigr)_{m\in\N_0}$.
  \item[(ii)]
  A {\bf strong M-polyfold bundle model} is the projection $\pr_{\cO}
    : \cR = {\textstyle \bigcup_{p\in\cO}} \{p\}\times \cR_p \to
    \cO$ from the total space of a {\bf strong sc-bundle retract}
    $(\cR_p\subset\F)_{p\in\cO}$ to its base retract $\cO$ as in
    Definition~\ref{def:BundleRetract}, where $\cR$ is the image of a
    strong retraction of bundle type.
  \item[(iii)]
Two local sc-trivializations $\Phi:  {\rm pr}^{-1}(U) \to \cR\subset
  [0,\infty)^k\times\E \times \F$, and $\Phi': {\rm pr}^{-1}(U') \to
  \cR'\subset [0,\infty)^{k'}\times\E' \times \F' $ to strong M-polyfold
  bundle models $\cR\to\cO$ and $\cR'\to\cO'$ are {\bf strongly compatible}
  if their transition map restricts to a scale smooth map with respect
  to the ambient sc-sectors $[0,\infty)^{k}\times \E\times \F_1$ and
  $[0,\infty)^{k'}\times \E' \times \F'_1$.
That is, we require sc$^\infty$ regularity of the map between these
  sectors in sc-Banach spaces of
  $$
    \io_{\cR'} \circ \Phi' \circ \Phi^{-1} \circ R  \,:\;
    R^{-1}\bigl( \Phi\bigl( {\rm pr}^{-1}(U\cap U') \bigr) \bigr)
    \cap [0,\infty)^k\times \E\times \F_1  \;\longrightarrow\;
    [0,\infty)^{k'}\times\E' \times \F'_1
  $$
  for any strong sc-retraction of bundle type with $R(\cU\times \F)=\cR$
  (and hence $R(\cU\times F_1)=\cR\cap (\cU\times F_1)$) and the
  inclusion $\io_{\cR'}:\cR'\cap (\cU'\times F'_1) \hookrightarrow
  [0,\infty)^{k'}\times E_0' \times  F'_1$.
  \end{enumlist}

\noindent
For a strong M-polyfold bundle $:\cY\to\cX$ we denote by ${\rm
  pr}|_{\cY^1}:\cY^1\to\cX$ the subbundle of vectors $Y\in \cY$ such
  that for some (and hence any) trivialization $\Phi:  {\rm pr}^{-1}(U)
  \to \cR\subset [0,\infty)^k\times\E \times \F$ to a strong M-polyfold
  bundle model we have $\Phi(Y)\in [0,\infty)^k\times E_0 \times F_1$.
\end{definition}

\begin{remark}\rm
Note that sc-bundle splicings in our simplified version of
  Remark~\ref{rmk:BunSpli} are automatically strong.
Indeed, scale smoothness of a family of projections $U\times\F \to
  \F, (v,f)\mapsto \Pi_v f$ directly implies scale smoothness of the
  restriction $U\times\F_1 \to \F_1$, since the dependence on $f$ is linear
  -- hence smooth once sc$^0$ -- and the scale structure on $U\subset\R^k$
  is trivial, hence it is oblivious to the shift in scales.
\end{remark}

\begin{example}\rm
In the example of Morse theory, the total space of the bundle over a
  space of unbroken trajectories $\widetilde\cB_a^c$ is $\widetilde\cE_a^c
  = \bigl(\cB_a^c\times W^{1,2}(\R,\R^n)\bigr)/\R$, where $\R$ acts by
  simultaneous shift on both factors.
This explains why the construction of charts only requires slicing
  conditions for the base.
The total space of the M-polyfold bundle over the space of broken and
  unbroken trajectories $\cX_a^c$ is then $\cY_a^c = \widetilde\cE_a^c
  \sqcup \widetilde\cE_a^b \times \widetilde\cE_b^c$, with the topology
  given by pregluing similar to Example~\ref{ex:TrajectoryTopology}.

In the bundle over unbroken trajectories, the ``higher
  regularity fibers'' discussed below are $\{\g\} \times
  W^{1+\ell,2}_{\delta_{\ell}}(\R,\g^*X)$ for $\g\in W^{2+k,2}_{\rm
  loc}\cap \widetilde\cB_a^c$ and $\ell=k+1$.
Whereas in case $X=\R^n$ with $\g^*X\cong\R^n$, these fibers are well
  defined for any $\ell>k$, and the general case of a nonlinear ambient
  space $X$ only allows for $\ell=k+1$.
This is because the $W^{1+\ell,2}_{\delta_{\ell}}$-completion of sections
  of $\g^*X$ requires a connection on $\g^*X$, or local trivializations,
  whose parallel transport and transition maps can only be as regular
  as $\g$, so a $W^{1+\ell,2}_{\delta_{\ell}}$-norm is well defined
  only for $1+\ell\leq k+2$.
\end{example}

The restriction to ``higher regularity fibers'' ${\rm
  pr}|_{\cY^1}:\cY^1\to\cX$ of any strong M-polyfold bundle is an
  M-polyfold bundle in its own right, since $\cY^1$ is an M-polyfold
  with local models in strong sc-retractions of bundle type in
  $[0,\infty)^k\times\E \times \F_1$, which are compatible by restriction
  of the strong compatibility requirement for the trivializations of
  $\cY\to\cX$.
The construction of these bundles uses the strongness assumption
  crucially; so, for example, the topological subbundle $\cY^2=\{
  Y\in\cY \,|\,  \Phi(Y)\in [0,\infty)^k\times E_0\times F_2 \}$
  over $\cX$ does not inherit a scale smooth structure, unless, for
  example, one additionally knows that all sc-bundle retracts are given
  by families of projections $\Pi_p:F_2\to F_2$ that are scale smooth
  as a map $(E_m\times F_{2+m})_{m\in\N} \to  (F_{2+m})_{m\in\N}$, which
  has no direct implication to or from regularity as a map $(E_m\times
  F_{1+m})_{m\in\N} \to  (F_{1+m})_{m\in\N}$.

Note that we still obtain more useful M-polyfold bundles from the
regularity stratifications on the M-polyfolds $\cY$ and $\cY^1$ that
are given by Definition~\ref{def:MpolyfoldStrata} (and which induce
different stratifications on $\cY^1\subset\cY$).
The regularity strata of $\cY$ resp.\ $\cY^1$ are
  \begin{align*}
    \cY_m
    &=
      \bigl\{Y\in\cY \,|\, \Phi(Y)\in [0,\infty)^k\times E_m\times F_m
      \;\text{in some chart}\; \Phi \bigr\} , \\
      \cY^1_m
    &=
      \bigl\{Y\in\cY \,|\, \Phi(Y)\in [0,\infty)^k\times E_m\times
      F_{m+1} \;\text{in some chart}\; \Phi \bigr\}.
    \end{align*}
Note that the restriction ${\rm pr}|_{\cY_m}:\cY_m\to\cX_m$ is
  another M-polyfold bundle since ${\rm pr}(\cY_m)\subset \cX_m$ by
  scale continuity, ${\rm pr}|_{\cY_m}$ locally surjects onto $\cX_m$
  in the M-polyfold bundle models, and the local trivializations are
  given by restriction of those for $p$.  Similarly, the restriction
  ${\rm pr}|_{\cY^1_m}:\cY^1_m\to\cX_m$ is another M-polyfold bundle
  for each $m\in\N_0$, so each regularity stratum $\cX_m$ of the
  base supports two bundles $\cY_m$ and $\cY^1_m$.
The fibers of the latter embed compactly and densely into the fibers
  of the former.
In fact, the motivation for introducing strong bundles is the need for
  ``compact perturbations,'' which we can now define rigorously as sections
  of $\cY_1$.
In addition, we introduce an abstract notion that encodes elliptic
  regularity for differential operators.
To begin, we recall the notion of scale smooth section from
  Definition~\ref{def:prelim}~(iii).

\begin{definition} \label{def:sc+}
Let ${\rm pr}:\cY\to\cX$ be a strong M-polyfold bundle. We denote the
  space of sc$^\infty$ sections by
  $$
    \G({\rm pr}) := \bigl\{ s: \cX\to \cY \;\text{sc}^\infty \,\big|\,
    {\rm pr}\circ s = {\rm Id}_{\cX} \bigr\} .
  $$
The subset of {\bf sc$^+$ sections} $\G^+({\rm pr})\subset\G({\rm pr})$
  consists of the sections $s:\cX\to\cY^1$ with values in $\cY_1$, that
  are in fact scale smooth as sections of $\cY_1\to \cX$, or equivalently
  $\G^+({\rm pr})\cong \G({\rm pr}|_{\cY_1})$.

Moreover, we call a section $s\in\G({\rm pr})$ {\bf regularizing}
  if the following implication holds:
  $$
    m\in\N_0, x\in \cX_m, s(x)\in \cY^1_m \;\Longrightarrow\;
    x\in\cX_{m+1} .
  $$
The space of regularizing sections is equivalently defined and denoted by
  $$
    \Gamma^{reg}({\rm pr}) :=  \bigl\{ s\in\G({\rm pr}) \,\big|\,
    \forall m\in\N_0 : s^{-1}(\cY^1_m) \subset \cX_{m+1} \bigr\} .
  $$
\end{definition}

Finally, we can phrase the fact that compact perturbations preserve
  elliptic regularity as some property of the appropriate sections,
  $$
    s \in \Gamma^{reg}({\rm pr}) , \nu \in \Gamma^{+}({\rm pr})
    \;\Longrightarrow\; f+\nu\in \Gamma^{reg}({\rm pr}) .
  $$

\begin{example}\rm
In the case of Morse theory, a change in the metric from $g$ to
  $g'$ corresponds to an sc$^+$ perturbation $\nu(\gamma)=\bigl(\g,
  \nabla^{g} f(\gamma) - \nabla^{g'} f(\gamma)\bigr)$ of the
  section $s(\gamma)=\bigl(\g,  \frac\rd{\rd t}\gamma - \nabla^{g}
  f(\gamma)\bigr)$.
However, in the case of Cauchy-Riemann operators, a perturbation
  of the almost complex structure from $J$ to $J'$ fails to be sc$^+$
  since the principal part of $\nu(u)=\bigl(u, (J'-J)\partial_t u\bigr)$
  is a differential operator of the same order as the principal part $u
  \mapsto \partial_s u + J \partial_t u$ of the section.
\end{example}

\subsection{Fredholm sections in M-polyfold bundles} \label{ss:fred}

Contrary to previous sections, we will work our way up towards the most
  general notion of Fredholm sections, starting with linear Fredholm
  operators and then proceeding via nonlinear Fredholm maps on sc-Banach
  spaces. 
Once this is accomplished, we will introduce the useful alternative
  notion of Fredholm maps with respect to a splitting into (finitely many)
  gluing parameters and an sc-Banach space.
The discussion in these stages is essentially copied from \cite{w:fred}.

We begin with \cite[Def.2.8]{HoferWysockiZehnder1} of an sc-Fredholm
  operator in terms of sc-direct sums $\E=X\oplus_{sc} Y$, which are
  defined in general as the splitting inducing an sc$^0$ isomorphism $\E
  \to (X\cap E_m)_{m\in\N_0} \times (Y\cap E_m)_{m\in\N_0}$.
This includes the nontrivial requirement that each sequence in the
  latter sc-product is in fact a scale structure on $X$ and $Y$
  respectively.
In particular, this implies that finite dimensional factors of an
  sc-direct sum must be contained in $E_\infty$.
We spell out the sc-direct sum requirements in detail below, though
  they will subsequently be simplified.

\begin{definition}  \label{def:FredOp}
Let $\E, \F$ be sc-Banach spaces.
A {\bf sc-Fredholm operator} $L:\E\to\F$ is a linear map $L: E_0\to F_0$
  that satisfies the following.
  \begin{enumerate}
  \item
  The kernel $\ker L$ is finite dimensional and has a $sc$-complement $
    \E = \ker L \oplus_{sc} X$ in the sense that $\ker L\subset E_\infty$
    and $X\subset E_0$ is a subspace on which $X_m:=(X\cap E_m)_{m\in\N_0}$
    induces a scale structure such that $E_m=\ker L \oplus X_m$ is a
    direct sum on every scale $m\in\N_0$.
  \item
  The image $L(E_0)$ has a finite dimensional $sc$-complement $\F =
    L(E_0) \oplus_{sc} C$ in the sense that $(L(E_0)\cap F_m)_{m\in\N_0}$
    induces a scale structure on $L(E_0)$ and $C\subset E_\infty$ is a
    finite dimensional subspace such that $F_m=(L(E_0)\cap F_m) \oplus C$
    is a direct sum on every scale $m\in\N_0$.
  \item
  The operator restricts to a $sc$-isomorphism $L|_X : X \to L(E_0)$
    in the sense that $L|_{X_m} : X_m \to L(E_0)\cap F_m$ is a bounded
    isomorphism on every scale $m\in\N_0$.
  \end{enumerate}
The {\bf Fredholm index} of  $L$ is $\ind(L) := \dim\ker L -
  \dim(\qu{F_0}{\im L})$.
\end{definition}

In practice one can prove the linear Fredholm property by checking the
  following simplified list of properties.

\begin{lemma}[\cite{w:fred} Lemma~3.6] \label{le:scfredlin}
Let $\E, \F$ be sc-Banach spaces.
Then a linear map $L: E_0\to F_0$ is  an sc-Fredholm operator if and
  only if it satisfies the following.
\begin{enumerate}
  \item
  $L$ is sc$^0$; that is, all restrictions $L|_{E_m} : E_m \to F_m$
    for $m\in\N_0$ are bounded linear operators.
  \item
  $L$ is {\bf regularizing}; that is, $e\in E_0$ and $L e\in F_m$ for
    any $m\in\N$ implies $e\in E_m$.
  \item
  $L: E_0 \to F_0$ is a Fredholm operator, that is, it has finite
    dimensional kernel $\ker L$ and cokernel $\qu{F_0}{L(E_0)}$.
  \end{enumerate}
\end{lemma}

Indeed, \cite[\S 3.5]{w:fred} shows that regularizing sc$^0$ operators,
  which are Fredholm on the $0$-scale restrict to Fredholm operators
  $L|_{E_m}: E_m \to F_m$ on every scale, and have isomorphic kernel and
  cokernel.
Then, a little more functional analysis provides the sc-complements
  required by the more complicated notion of sc-Fredholm operator.

\begin{example} \rm
The prototypical examples of sc-Fredholm operators are the following
  elliptic operators:
  \begin{itemlist}
  \item
    $\frac\rd\dt : \cC^1(S^1) \to \cC^0(S^1)$ is an sc-Fredholm
      operator from $\bigl( \cC^{1+k}(S^1) \bigr)_{k\in\N_0}$ to $\bigl(
      \cC^{k}(S^1) \bigr)_{k\in\N_0}$.
  \item
  The Cauchy--Riemann operator $\overline\partial_J : W^{1,p}(S^2,\C^n)
    \to L^p(S^2,\Lambda^{0,1}\otimes_J\C^n)$ with respect to $J=i$
    on $\C^n$ and $j=i$ on $S^2=\CP^1$ is given by $u\mapsto \frac 12(
    J \circ \rd u \circ j + \rd u )$.
  (Its target is the $L^p$-closure of the smooth, $(J,j)$-antilinear
    $\C^n$-valued $1$-forms on $S^2$.)
  It is an sc-Fredholm operator from $\bigl( W^{1+k,p}(S^2,\C^n)
    \bigr)_{k\in\N_0}$ to $\bigl( W^{k,p}(S^2,\C^n) \bigr)_{k\in\N_0}$
    for any  ${1<p<\infty}$.
  \end{itemlist}

\noindent
Indeed, the sc$^0$-property of these operators is a formalization of
  the fact that linear differential operators of degree $d$ are bounded
  as operators between appropriate function spaces (e.g.\ H\"older or
  Sobolev spaces), with a difference of $d$ in the differentiability index.
The regularizing property, in this context, is simply the statement of
  elliptic regularity.
Finally, the elliptic estimates for an operator and its dual generally
  hold on all scales similar to the boundedness above, and this implies
  the Fredholm property on all scales.
\end{example}

Next, we need a notion of a nonlinear Fredholm map on sc-Banach spaces
  that allows for an implicit function theorem for sc$^1$ maps with
  surjective linearization.
This cannot simply be obtained by adding ``sc-'' in appropriate places to
  the classical definition of Fredholm maps since the implicit function
  theorem is usually proven by means of a contraction property in a
  suitable reduction.
Since the contraction will be iterated to obtain convergence, it needs
  to act on a fixed Banach space rather than between different levels of
  an sc-Banach space.
In classical nonlinear Fredholm theory, this contraction form follows
  from the continuity of the differential in the operator norm, and, in
  particular, continuity of the differential in the operator topology is
  indeed necessary to obtain the contraction property which allows one to
  use Banach's fixed point theorem.
However, for the generalized Cauchy-Riemann operators involved in
  the description of holomorphic curve moduli spaces, this stronger
  differentiability will not hold as soon as their domain contains gluing
  parameters which act on functions by reparametrization.
This issue is resolved in \cite[Def.3.6]{HoferWysockiZehnder2}\footnote
  {
  The following definition is not explicitly given in the
    current work of HWZ.
  It is obtained from the definition of a polyfold Fredholm section of
    a strong bundle as the special case of a section in a trivial bundle
    with trivial splicing.
  }
  by making the contraction property a part of the definition of
  Fredholm maps.\footnote
  {
  Note that a classical ``contraction property'' would be an estimate
    such as \eqref{contraction} for some $\theta<1$.
  However, Fredholm stability (preservation of the contraction property
    under appropriate perturbations) turns out to require this kind of
    estimate for arbitrarily small contraction factors $\theta>0$, just
    allowing for $\theta$-dependent domains.
  }

\begin{definition} \label{def:HWZfred}
Let $\Phi:\E\to\F$ be a $sc^\infty$ map between sc-Banach spaces $\E,
  \F$.
Then $\Phi$ is {\bf sc-Fredholm at $\mathbf{0}$} if the following holds:
  \begin{enumerate}
    \item
    $\Phi$ is {\bf regularizing as germ}: For every $m\in\N$ there exists
      $\epsilon_m>0$ such that $\Phi(e)\in F_{m+1}$ and $\|e\|_{E_m}\leq
      \epsilon_m$ implies $e\in E_{m+1}$.
    \item
    There exists an sc-Banach space $\W$ and sc-isomorphisms (i.e.\
      linear $sc^0$ bijections) $h:\E \to \R^k\times\W$ and $g:\F \to
      \R^\ell\times\W$ for some $k,\ell\in\N_0$ such that
      $$
	g \circ \Phi \circ h^{-1} \,: \; (v,w) \; \mapsto\;  g(\Phi(0))
	+ \bigl( A(v,w) , w - B(v,w) \bigr) ,
      $$
      where $A:\R^k\times\W \to \R^\ell$ is any $sc^\infty$ map and
      $B:\R^k\times\W \to \W$ is a {\bf contraction germ}:
    For every $m\in\N_0$ and 
    $\theta>0$
    there exists $\epsilon_m>0$
      such that for all $v\in\R^k$ and $w_1,w_2\in\W$ with $\|v\|_{\R^k},
      \|w_1\|_{W_m}, \|w_2\|_{W_m} \leq \epsilon_m$ we have
      \begin{equation}\label{contraction}
	\bigl\| B(v,w_1) - B(v,w_2) \bigr\|_{W_m} \leq \theta \| w_1 -
	w_2 \|_{W_m} .
	\end{equation}
    \end{enumerate}
\end{definition}

This definition, however, raises the question of how this ``contraction
  germ normal form'' is established in practice.
Recall that in the example of Cauchy-Riemann operators, it
  was the presence of gluing that motivated the development of an
  alternative nonlinear Fredholm notion in \cite{w:fred}, based on the
  observation that the gluing parameters usually are the only source of
  non-differentiability, and after splitting off a finite dimensional
  space of gluing parameters one deals with classical $\cC^1$-maps on
  all scale levels.
The resulting notion of a Fredholm property with respect to a splitting
  $\E\cong\R^d\times\E'$ is just slightly stronger than the definition
  via contraction germs, but should be more intuitive for applications
  to Morse theory as well as holomorphic curve moduli spaces.
In fact, in practice the Fredholm property in \cite[Thm.8.26]{hwzbook}
  and \cite[Prop.4.8]{HWZI_applications} is proven implicitly via this
  stronger differentiability.
We formalize this approach in the following Lemma where we denote open
  balls centered at $0$ in a level $E_m$ of a scale space by
$$
B_r^{E_m} := \bigl\{ e \in E_m \,\big|\, \| e\|_m < r \bigr\} \qquad \text{for}\; r>0.
$$

\begin{lemma}[\cite{w:fred} Thm.4.4]  \label{le:scfred}
Let $\Phi:\E\to\F$ be a $sc^\infty$ map between sc-Banach spaces $\E,
  \F$ such that the following holds.
\begin{enumlist}
  \item[(i)]
    $\Phi$ is regularizing as germ in the sense of
    Definition~\ref{def:HWZfred}~(i).
  \item[(ii)]
    $\E\cong\R^d\times\E'$ is an sc-isomorphism and for every $m\in\N_0$
    there exists $\eps_m>0$ such that $\Phi(r,\cdot) : B_{\eps_m}^{E'_m}
    \to F_m$ is differentiable for all $|r|_{\R^d}<\eps_m$, and its
    differential $\rD_{\E'}\Phi(r_0,e_0) : \E' \to \F$, $e \mapsto
    \frac{\rm d}{{\rm d}t} \Phi(r_0,e_0+ te)|_{t=0}$ in the direction
    of $\E'$ has the following continuity properties:
    \begin{enumerate}
      \item[a)]
      For fixed $m\in\N_0$ and $r\in B_{\eps_m}^{\R^d}$ the differential
	operator $B_{\eps_m}^{E'_m} \to \mathcal{L}(E'_m,F_m)$, $e \mapsto
	\rD_{\E'}\Phi(r,e)$ is continuous, and the continuity is uniform
	in a neighborhood of $(r,e)=(0,0)$.
      That is, for any $\d>0$ there exists $0<\eps_{m,\d}\le\eps_m$
	such that for all $(r,e)\in B_{\eps_{m,\d}}^{\R^d}\times
	B_{\eps_{m,\d}}^{E'_m}$ we have
	$$
	  \qquad\qquad \bigl\| \rD_{\E'}\Phi(r,e) h  -
	  \rD_{\E'}\Phi(r,e') h \bigr\|_{F_m}  \le \d \| h \|_{E'_m}
	  \qquad \forall \|e'-e\|_{E'_m} \le \eps_{m,\d} , h \in E'_m .
	$$
      \item[b)]
      For any sequences $\R^d \ni r^\nu\to 0$ and $e^\nu\in B^{E'_m}_1$
	with $\bigl\| \rD_{\E'}\Phi(r^\nu,0) e^\nu \bigr\|_{F_m}
	\underset{\scriptscriptstyle \nu\to\infty}{\longrightarrow} 0$
	we also have $\bigl\| \rD_{\E'}\Phi(0,0) e^\nu \bigr\|_{F_m}
	\underset{\scriptscriptstyle \nu\to\infty}{\longrightarrow} 0$.
      \end{enumerate}
  \item[(iii)]
    The differential $\rD_{\E'}\Phi(0,0) : \E' \to \F$ is sc-Fredholm.
    Moreover $\rD_{\E'}\Phi(r,0) : E_0 \to F_0$ is Fredholm for all
      $|r|_{\R^d} < \eps_0$, with Fredholm index equal to that for
      $r=0$, and it is weakly regularizing; that is, $\ker\rD_{\E'}\Phi(r,0)
      \subset E_1$. 
  \end{enumlist}
Then $\Phi$ is sc-Fredholm at $0$ in the sense of
  Definition~\ref{def:HWZfred}.
\end{lemma}

Note here that condition (ii-a) requires the differential
  $\rD_{\E'}\Phi(r,e)$ to be continuous in the operator topology under
  variations of $e$, but not of $r$.
(It is only the first continuity that is uniform in $r$.)
This is why the second part of condition (iii) does not simply follow
  from Fredholm stability.

\begin{example} \rm
For unbroken Morse trajectories, the principal part of the section
  roughly takes the form $\Phi(\gamma)= \frac\rd{\rd t}\gamma - \nabla
  f(\gamma)$ in local charts.
It satisfies conditions (ii) and (iii) of the above Lemma
  since it is in fact classically smooth as map $\bigl(\phi_a^c +
  W^{k+2,2}_{\delta_k}(\R,\R^n)\bigr) \to W^{k+1,2}_{\delta_k}(\R,\R^n)$,
  where $\phi_a^c$ is a smooth reference path from $a$ to $c$. 
\end{example}

In order to move on to a Fredholm notion for sections of M-polyfold
  bundles, we need to introduce the notion of a filling.
This is a device that turns the local study of the section, possibly
  defined only as map between nontrivial retracts with tangent bundles
  of locally varying dimensions, into the equivalent local study of a
  ``filled'' sc-Fredholm map from an open set of an sc-Banach space to
  another fixed sc-Banach space.

\begin{definition} \label{def:filling}
Let $s:\cO \to \cR$, $s(p)=(p,f(p))$ be an sc$^\infty$ section
  of an M-polyfold bundle model $\pr_{\cO} : \cR \to \cO$
  as in Definition~\ref{def:BundleRetract}, whose base is an
  sc-retract $\cO\subset [0,\infty)^k\times\E$ containing $0 \in
  [0,\infty)^k\times\E$, and with fibers $\cR_p\subset\F$ for $p\in\cO$.
Then a {\bf Fredholm filling at $0$} for $s$ over $\cO$ consists of
  \begin{itemlist}
    \item
      a neat sc-retraction of bundle type $R: \cU\times\F \to
      \cU\times\F$, $R(p,h) =  \bigl( r(p) , \Pi_{p} h \bigr)$ on an
      open subset $\cU\subset[0,\infty)^k\times\E$ such that $r(\cU)=\cO$
      and $\Pi_{p}\F=\cR_p$ for all $p\in\cO$,
    \item
      an sc$^\infty$ map $\overline{f}:\cU \to \F$ that is sc-Fredholm
      at $0$ in the sense of Definition~\ref{def:HWZfred},
    \end{itemlist}
  with the following properties:
  \begin{enumerate}
    \item  
      $\bar{f}|_\cO=f$;
    \item
      if $p\in\cU$ such that $\bar{f}(p)\in\cR_{r(p)}$ then $p=r(p)$,
      that is $p\in\cO$;
    \item
      The linearization of the map $[0,\infty)^k\times\E \to \F$, $p
      \mapsto (\id_\F -\Pi_{r(p)} )\bar{f}(p)$ at $0$ restricts to an
      isomorphism from $\ker \rD_0 r$ to $\ker \Pi_0$.
    \end{enumerate}
\end{definition}

Note that conditions (i) and (ii) imply equality of the zero sets
  $\bar f^{-1}(0)=f^{-1}(0)$, since the vector $0$ lies in every fiber
  $\cR_{r(p)}$.
Condition (iii) ensures that the Fredholm index of any two fillers is
  the same.
In particular, if $f(p)=0$, then the linearization $\rD_p f:\rT_{p}\cO
  \to \cR_p$ has the same kernel as $\rD_p\bar{f}: \rT_{p}\cU\to \F$,
  and the cokernels of both maps are identified by the inclusion $\cR_p
  \subset \F$.

\begin{definition}\label{def:scFredholmSection}
An sc$^\infty$ section $s:\cX\to\cY$ of an M-polyfold bundle is an
  {\bf sc-Fredholm section} if $s$ is regularizing in the sense of
  Definition~\ref{def:sc+} and for each $x\in\cX_\infty$ there is a
  local sc-trivialization $\Phi:{\rm pr}^{-1}(U)\to \cR$ in the sense of
  Definition~\ref{def:Bundle} over a neighborhood $U\subset\cX$ of $x$
  with $\Phi(x,0)=0$, such that $\Phi_*s$ has a Fredholm filling in the
  sense of Definition~\ref{def:filling}.
 \end{definition}

Now the \emph{Fredholm index} of an sc-Fredholm section $s:\cX\to\cY$
  at a point $p\in s^{-1}(0)$ can be defined as the Fredholm index of
  its linearization $\rD_p s:\rT_p\cX \to \cY_p$.
This linearization is in any local trivialization given by the
  linearization of the fiber part $\rD_p f$, and has the same Fredholm
  index as the linearization $\rD_p\bar{f}$ of any filler $\bar{f}$.
In fact, [\cite{HWZnew}, Prop. 6.2] shows that this index is constant
  on path-connected components of $\cX$.

\begin{example}\rm
In applications to splicings obtained from pregluing constructions as
  in Example~\ref{ex:pregluing_retraction}, a canonical candidate for a
  Fredholm filling is given by applying the linearized operator on the
  image of the antigluing, while the nonlinear operator (the gradient
  flow or Cauchy-Riemann operator) acts only on the image of the gluing.
In the case of Morse theory, this is being worked out in \cite{AW}.
For an analogous simplified case of Hamiltonian Floer theory see
  \cite{w:fred}.
\end{example}

Recall that sc$^{+}$-sections play the role of perturbations.
The following stability result (which was first proven in
  \cite{HoferWysockiZehnder2} under slightly different assumptions)
  is extremely important for the perturbation theory.
It is the polyfold analogue of the classical Fredholm theory fact that
  the sum of a Fredholm operator and a compact operator is again Fredholm.

\begin{theorem}[\cite{HWZnew}, Prop.3.10, Thm.3.15]
Let ${\rm pr}:\cY\to \cX$ be a strong M-polyfold bundle.
Then for any sc-Fredholm section $s:\cX\to\cY$ and sc$^+$ section
  $\nu:\cX\to\cY^1$ the section $s+\nu: \cX\to\cY$ is again sc-Fredholm,
  and has the same Fredholm index as $s$ on each path-connected component
  of $\cX$.
\end{theorem}

\subsection{Transverse perturbations and the implicit function theorem}
  \label{ss:reg}

Finally, with the notions of bundles and Fredholm sections
  in place, we can introduce the polyfold regularization theorem
  \ref{thm:PolyfoldRegularization2} more rigorously, beginning with the
  notion of transversality and an implicit function theorem for transverse
  Fredholm sections.
Here and throughout, we fix an M-polyfold bundle $\pr:\cY\to\cX$, which
  mainly is assumed to have no boundary or corners (i.e.\ $\cX=\cX^{(0)}$
  and $\cX^{(\ell)}=\emptyset$ for $\ell\ge 1$ in the notation of
  Definition~\ref{def:MpolyfoldStrata}).
The case of Fredholm sections over M-polyfolds with boundaries and
  corners is discussed separately.

\begin{definition} \label{def:trans}
A scale smooth section $s:\cX\to\cY$ is called {\bf transverse (to the
  zero section)} if for every $x\in s^{-1}(0)$ the linearization $\rD_x
  s : \rT_x\cX \to \cY_x$ is surjective.
Here the {\bf linearization} $\rD_x s$ is represented by the
  differential $\rD_{\phi(x)} (\Pi\circ f\circ r)|_{\rT_{\phi(x)}\cO} :
  \rT_{\phi(x)}\cO \to \Pi_{\phi(x)}(\F)$ in any local sc-trivialization
  ${\rm pr}^{-1}(U)\overset{\sim}{\to} \bigcup_{p\in\cO}\Pi_p(\F)$ which
  covers $\phi:\cX\supset U \overset{\sim}{\to} \cO=r(\cU)\subset \E$
  and transforms $s$ to $p\mapsto (p,f(p))$.
\end{definition}

\begin{theorem}[\cite{HoferWysockiZehnder2}, Thm.5.14]
  \label{iftnoboundary}
Let $s:\cX\to\cY$ be a transverse sc-Fredholm section. 
Then the solution set $\cM:=s^{-1}(0)$ inherits from its ambient space
  $\cX$ a smooth structure as finite dimensional manifold.
Its dimension is given by the Fredholm index of $s$ and the tangent
  bundle is given by the kernel of the linearized section, $\rT_x
  \cM=\ker\rD_x s$.
\end{theorem}

If $\cX$ has boundary and corners then the charts $\phi:\cX\supset U
  \overset{\sim}{\to} \cO=r(\cU)\subset C$ take values in an sc-sector
  $C=[0,\infty)^k\times\E$ and the implicit function theorem, in addition
  to surjectivity of the linearization $\rD_x s$, also requires some type
  of transversality of the kernel $K_x:= \ker \rD_{\phi(x)} (\Pi\circ
  f\circ r)|_{\rT_{\phi(x)}\cO} \subset \R^k\times \E$ at any solution
  $x\in s^{-1}(0)$ to the boundary strata.
Since by the regularization property the solution set
  $s^{-1}(0)\subset\cX_\infty$ consists of smooth points, we can choose
  the chart so that $\phi(x)=0\in C$ is the point with highest degeneracy
  index in the sc-sector.
Then the classical transversality notion is the following.\footnote
  {
    \cite[Def.4.10]{HoferWysockiZehnder2} requires neatness with respect
      to the partial quadrant $\rD_0 r([0,\infty)^k\times \E)\subset \rT_0
      \cO \cong \rT_x\cX$.
    This is equivalent to our simplified notion by linear algebra using
      the fact that $K_x$ is finite dimensional by the Fredholm property
      of the section (hence under the above neatness condition one finds an
      sc-complement of $K_x\subset\R^k\times\E$ in $\{0\}\times\E$) and that
      $\im\rD_0 r=\rT_0\cO$ projects onto $\R_k$ by the neatness condition
      on the sc-retraction~$r$.
  }

\begin{definition}
A subset $K\subset \R^k\times \E$ is {\bf neat with respect to the
  sector} $[0,\infty)^k\times\E$ if the projection $\Pr_{\R^k}:K \to \R^k$
  is surjective.

A section $s:\cX\to\cY$ over an M-polyfold $\cX$ with nonempty
  boundary $\partial\cX=\bigcup_{\ell\ge 1} \cX^{(\ell)}$ is called
  {\bf neatly transverse} if it is transverse in the sense of
  Definition~\ref{def:trans} and each kernel $K_x$ of the linearized
  operators at solutions $x\in s^{-1}(0)$ is neat with respect to an
  M-polyfold chart with maximally degenerate sc-sector.
\end{definition}

In particular, neatness requires sufficiently high dimension $\dim
  K_x\ge k$, so that solution sets of transverse sections with neat
  kernels cannot intersect boundary strata of degeneracy index higher
  than the Fredholm index.
The corresponding implicit function theorem is the following.

\begin{theorem}[\cite{HoferWysockiZehnder2},
  Thm.5.22]\label{iftwithboundary}
Let $s:\cX\to\cY$ be a neatly transverse sc-Fredholm section over an
  M-polyfold $\cX$ with nonempty boundary.
Then the solution set $\cM:=s^{-1}(0)$ inherits from its ambient space
  $\cX$ a smooth structure as a finite dimensional manifold with boundary
  and corner stratification $\cM^{(\ell)} = s^{-1}(0) \cap \cX^{(\ell)}$.
\end{theorem}

\begin{remark} \label{rmk:diag} {\rm For purposes beyond the scope of
  this exposition\footnote
  {
  The construction of coherent perturbations does not always allow one
    to achieve neatness by perturbations.
  Roughly speaking, if a moduli problem can be glued to itself, then
    the negative index solutions in a family occur in arbitrarily high
    degeneracy indices.
  In the operations formalism of HWZ, this is reflected in the occurrence
    of ``diagonal relators''; it also appears in geometric regularizations
    such as \cite[\S 10e]{Seidel}, where transversality is achieved by a
    ``delay function method''.
  }
  HWZ also introduce the following weaker notion\footnote
  {
  \cite[Def.4.10]{HoferWysockiZehnder2} again works in the partial quadrant
    $\rD_0 r([0,\infty)^k\times \E)\subset \rT_0 \cO \cong \rT_x\cX$,
    but we may simplify this to a condition in $C=[0,\infty)^k\times\E$
    since $\rD_0 r|_{K_x}=\id_{K_x}$ by the retraction property of $r$. With
    that in mind, we rephrased the conditions of $K_x\cap C\subset C$ having
    open interior and an sc-direct sum $\R^k\times\E=K_x \oplus N$ such that
    $k+n\in C \Leftrightarrow k\in C$ for $\|n\|/\|k\|$ sufficiently small.
  Indeed, in the case of $\pr_{\R_k}(K_x)\neq \R^k$ our simplified notion
    clearly implies these conditions.
  On the other hand, the first condition implies that $K_x$ has a basis
    of vectors in $C$, and in fact in $(0,\infty)^k\times\E$, unless $K_x$
    is entirely contained in a boundary face of $C$.
  The latter is excluded by the second condition which must hold for some
    vectors $n$ transverse to that face.
  }  
  of transversality to the boundary strata:}

\medskip

The subset $K_x\subset \R^k\times \E$ is {\bf in good position to the
  sector} $[0,\infty)^k\times\E$ if either the projection $\Pr_{\R^k}:K_x
  \to \R^k$ is surjective or if $\Pr_{\R^k}:K_x \to \R^k$ is injective
  and $K_x$ is spanned by vectors in $(0,\infty)^k\times\E$.
A section $s:\cX\to\cY$ over an M-polyfold $\cX$ with nonempty boundary
  $\partial\cX=\bigcup_{\ell\ge 1} \cX^{(\ell)}$ is said to have {\bf
  kernels in good position} if each kernel $K_x$ of the linearized
  operators at solutions $x\in s^{-1}(0)$ is in good position w.r.t.\
  an M-polyfold chart with maximally degenerate sc-sector.

\medskip

{\rm This notion of boundary transversality still provides an implicit
  function theorem, in which just the control of boundary strata is less
  precise, see \cite[Thm.5.22]{HoferWysockiZehnder2}.  }

\medskip 

Let $s:\cX\to\cY$ be a transverse sc-Fredholm section over an M-polyfold
  $\cX$ with nonempty boundary, and suppose that it has kernels in good
  position.
Then the solution set $\cM:=s^{-1}(0)$ inherits from its ambient
  space $\cX$ a smooth structure as finite dimensional manifold with
  boundary and corner stratification $\cM^{(\ell)} \subset s^{-1}(0)
  \cap \bigcup_{k\ge\ell}\cX^{(k)}$.
\end{remark}

As in the classical situation, an sc-Fredholm section need not generally be transverse,
in which case the above implicit function theorems do not apply.
However, one can achieve transversality by perturbation with
  sc$^{+}$-sections, which are essentially compact perturbations of the
  Fredholm section and were introduced in Definition~\ref{def:sc+}; they
  exist if $\cY\to\cX$ is a {\bf strong} M-polyfold bundle in the sense
  of Definition~\ref{def:sBundle}.
In order to construct appropriate perturbations from these, one
  moreover needs to work with smooth cutoff functions, which will
  be provided by assuming one works with ambient {\bf sc-Hilbert
  structures}, rather than sc-Banach structures, as introduced in
  Definition~\ref{def:scBanachSpace}.
(See the discussion there for a possible extension to sc-Banach
  structures with scale smooth cutoff functions.)

Additionally, we now need to be concerned with preserving the compactness
  of the unperturbed solution set $s^{-1}(0)$.
Recall from Definition~\ref{def:prelim}~(iii) that a section
  $s:\cX\to\cY$ is called {\bf proper} if $s^{-1}(0)$ is compact.
In order to preserve compactness one can make use of the compactness
  of the embedding $F_1\hookrightarrow F_0$ in the scale structure of
  the ambient space of the fibers of the bundle $\pr:\cY\to\cX$.
More precisely, recall that the fibers $\cY_x$ for $x\in\cX$ are
  locally isomorphic to subspaces $\bigl(\cR_p\subset\F\bigr)_{p\in\cO}$
  parametrized by an sc-retract $\cO$, and the transition maps preserve
  the fibers $\cR_p\cap F_1$, so they form another M-polyfold bundle
  $\cY_1\to\cX$.
By restricting the $F_1$-norm to the fibers and patching these
  local fiber-wise norms together with smooth cutoff functions on $\cX$,
  one now obtains an auxiliary norm on the dense subset $\cY_1\subset
  \cY$ in the following sense.

\begin{definition}
An {\bf auxiliary norm} $N$ for the strong M-polyfold bundle
  $\pr:\cY\to\cX$ is a continuous map $N:\cY_1\to [0,\infty)$ such that
  the restriction to each fiber $\pr^{-1}(x)\cap \cY_1$ for $x\in\cX$
  is a complete norm.

Moreover, if $s:\cX\to\cY$ is a proper section, then a pair of
  an auxiliary norm $N$ and an open neighborhood $\cU\subset\cX$
  of $s^{-1}(0)$ is said to {\bf control compactness} if for any
  sc$^+$-section $\nu:\cX\to\cY_1$ with $\supp \nu \subset \cU$
  and $\sup_{x\in\cX} N(\nu(x)) \le 1$ the perturbed solution set
  $(s+\nu)^{-1}(0)\subset\cX$ is compact.
\end{definition}

Any two auxiliary norms are equivalent in a neighborhood of the compact
  solution set $s^{-1}(0)$ by \cite[Lemma~5.8]{HoferWysockiZehnder2}.
Moreover, \cite[Thm.5.12]{HoferWysockiZehnder2} proves that
  neighborhoods controlling compactness exist for any given auxiliary
  norm.
Here the compactness holds with respect to the basic $\cX_0$ topology,
  but by \cite[Thm.5.11]{HoferWysockiZehnder2} can be strengthened to
  the topology on $\cX_\infty$ (given by simultaneous convergence in all
  topologies on $\cX_\infty\subset\cX_m$) if the section $s:\cX\to\cY$
  (and hence also $s+\nu$) is assumed to be sc-Fredholm.
With these notions in place we can finally state a
  technically complete version of the M-polyfold regularization
  theorem~\ref{thm:PolyfoldRegularization2}, which -- in the case without
  boundary -- simultaneously achieves compactness and
  transversality of the perturbed solution space, as well as a uniqueness
  up to cobordism.

\begin{theorem} (\cite{HoferWysockiZehnder2},Theorem~5.22)
  \label{thm:PolyfoldRegularization3}
Let $\pr:\cY\to\cX$ be a strong M-polyfold bundle modeled on sc-Hilbert
  spaces, and let $s:\cX\to\cY$ be a proper Fredholm section.
\begin{enumlist}
  \item[(i)]
    For any auxiliary norm $N:\cY_1\to[0,\infty)$ and neighborhood
      $s^{-1}(0)\subset\cU\subset\cX$ controlling compactness, there exists
      an sc$^+$-section $\nu:\cX\to\cY_1$ with $\supp\nu\subset\cU$ and
      $\sup_{x\in\cX} N(\nu(x)) < 1$, and such that $s+\nu$ is transverse
      to the zero section.
    In particular, $(s+\nu)^{-1}(0)$ carries the structure of a smooth
      compact manifold.
  \item[(ii)]
    Given two transverse perturbations $\nu_i:\cX\to\cY_1$ for
      $i=0,1$ as in (i), controlled by auxiliary norms and neighborhoods
      $(N_i,\cU_i)$ controlling compactness, there exists an sc$^+$-section
      $\Ti\nu:\cX\times[0,1]\to\cY_1$ such that $\{(x,t)\in\cX\times[0,1]
      \,|\, s(x)+\Ti\nu(x,t) \}$ is a smooth compact cobordism from
      $(s+\nu_0)^{-1}(0)$ to $(s+\nu_1)^{-1}(0)$.
    For details, see Remark 5.16 of \cite{HWZnew}.
  \end{enumlist}
\end{theorem}

Note here that one can choose the perturbations in part (i) ``small''
  in the following ways: Given a pair $(N,\cU)$ that controls
  compactness, we can apply Theorem~\ref{thm:PolyfoldRegularization3}
  with the auxiliary norm $\delta^{-1} N$ scaled by any $\delta>0$ and
  any neighbourhood $\cU'\subset \cU$ of the zero set $s^{-1}(0)$. 
In fact, it suffices to have $\cU'$ contain the part of the zero set
  where $s$ is not transverse.
As a result, we obtain a perturbation $\nu$ of small norm
  $\sup_{x\in\cX}N(\nu(x))< \delta$ and -- more importantly -- small
  support near the nontransverse part of $s^{-1}(0)$.
The latter -- very much unlike any geometric perturbation scheme --
  allows us to preserve parts of the solution space that are already cut
  out transversely.
Moreover, the second smallness control on perturbations is useful when
  solutions in $s^{-1}(0)$ satisfy a desirable property (e.g.\ positivity
  of intersections).
If this property is open with respect to the topology on $\cX$ (e.g.\ the
  $H^3$-topology, which is stronger than $\cC^1$), then the perturbation
  $\nu$ can be chosen with support sufficiently close to $s^{-1}(0)$
  so that the perturbed solutions in $(s+\nu)^{-1}(0)\subset s^{-1}(0)
  \cup \supp\nu$ still have the same property.

\begin{remark} [\bf Regularization with boundary and corners]
  \label{rmk:RegularizationCorners} \rm
The regularization theorem~\ref{thm:PolyfoldRegularization3} generalizes
  directly to strong bundles $\cY\to\cX$ over M-polyfolds with boundary
  and corners in two versions corresponding to the notion of transversality
  to the boundary strata.

Firstly, (i) holds with $s+\nu$ neatly transverse, and hence
  $(s+\nu)^{-1}(0)$ a compact manifold with boundary and corners, whose
  corner strata are given by its intersection with the corresponding
  boundary strata of $\cX$.
Moreover, (ii) provides a cobordism with boundary and corners in the
  sense that its intersection with each stratum $\cX^{(\ell)}\times[0,1]$
  is a cobordism between $(s+\nu_0)^{-1}(0)\cap \cX^{(\ell)}$ and
  $(s+\nu_1)^{-1}(0)\cap \cX^{(\ell)}$.

Secondly, under additional conditions on the perturbations discussed
  in Remark~\ref{rmk:diag}, the transverse perturbations $s+\nu$ in (i)
  can still be constructed to have kernels in good position, and hence
  $(s+\nu)^{-1}(0)$ is a compact manifold with boundary and corners.
Then (ii) provides a cobordism with boundary and corners in the sense
  that its corner strata are cobordisms between the corner strata of
  $(s+\nu_0)^{-1}(0)$ and $(s+\nu_1)^{-1}(0)$.
\end{remark}


\begin{thebibliography}{99}

\bibitem[A]{Adams} R.A.Adams,
  {\it Sobolev Spaces}, Academic Press, 1978.


\bibitem[AFFW]{AFW}
  P.~Albers, 
  B.~Filippenko, 
  J.~Fish, and K.~Wehrheim,  {\it A proof of the
  Arnold conjecture by polyfold techniques}, in preparation. 

\bibitem[AW]{AW}
  P.~Albers, K.~Wysocki, {M-Polyfolds in Morse Theory}, in preparation.

\bibitem[AD]{AD2010}
 M.~Audin, M.~Damian, {\it Th\'{e}orie de Morse et homologie de Floer}, Savoirs Actuels (Les Ulis), EDP Sciences (2010).

\bibitem[B]{Bourgeois}
  F.~Bourgeois, \emph{A Morse-Bott approach to contact homology},
  Ph.D. thesis, 2002.

\bibitem[{BEHWZ}]{BEHWZ}
  F.~Bourgeois, Y.~Eliashberg, H.~Hofer, K Wysocki and E Zehnder,
  \emph{Compactness results in symplectic field theory}, Geom. Topol.
  {\bf 7} (2003), 799--888.

\bibitem[Bo]{Nate}
  N.~Bottman {\it Pseudoholomorphic quilts with figure eight
  singularity}, arXiv:1410.3834.

\bibitem[C]{Chekanov}
  Yu.~Chekanov, {\it Differential algebra of Legendrian links},
  Invent. Math. {\bf 150} (2002), 441--483.

\bibitem[CL]{CL} O.~Cornea, F.~Lalonde, Cluster homology: an overview
  of the construction and results, {\it Electron.\ Res.\ Announc.\
  Amer.\ Math.\ Soc.} 12 (2006), 1--12.

\bibitem[CM]{CM} K.~Cieliebak, K.~Mohnke, \emph{Symplectic hypersurfaces
  and transversality for Gromov-Witten theory} J. Symp. Geom. {\bf 5}
  (2007), no. 3, 281--356.

\bibitem[CMS]{cms}
  K.~Cieliebak, I.~Mundet i Riera, D.A.~Salamon,
  \emph{Equivariant moduli problems, branched manifolds, and the Euler
  class}, Topology 42 (2003), 641--700. 

\bibitem[EES]{EES}
  T.~Ekholm, J.~Etnyre, M.~Sullivan. \emph{Legendrian contact homology
  in $P\times\R$}, Transactions of the American Mathematical Society 359
  (2007), no.7, 3301--3335.

\bibitem[EGH]{EGH}
  Y. Eliashberg, A.B. Givental, and H. Hofer, Introduction to symplectic
  field theory, Geom.Funct. Anal. 10 (2000), 560--673.

\bibitem[Fa]{fabert} O.~Fabert, \emph{Contact homology of Hamiltonian
  mapping tori}, Comm. Math. Helv. {\bf 85} (2010), 203--241.

\bibitem[F1]{floer} A.~Floer,
  \newblock The unregularized gradient flow of the symplectic action.
  \newblock {\em Comm.\ Pure Appl.\ Math.}
  {\bf 41} (1988), no. 6, 775--813.

\bibitem[F2]{floer:Lagrangian} A.~Floer,
  Morse theory for Lagrangian intersections,
  {\it J. Diff. Geom.} {28} (1988), 513--547.

\bibitem[F3]{floer:Arnold} A.~Floer,
  \newblock Symplectic fixed points and holomorphic spheres.
  \newblock {\em Comm.\ Math.\ Phys.}
  {\bf 120} (1989), 575--611.


\bibitem[FHS]{floer-hofer-salamon}
  A.~Floer, H.~Hofer, D.A.~Salamon.
  \newblock Transversality in elliptic {M}orse theory for the symplectic
    action.
  \newblock {\em Duke Math. J.}, 80(1):251--292, 1995.

\bibitem[Fu]{F} K.\ Fukaya, Morse homotopy, $A_\infty$-category, and
  Floer homologies, {\it Proceedings of GARC Workshop on Geometry and
  Topology '93}, Lecture Notes Ser.\ 18, 1--102.

\bibitem[FO]{FO}  K.\ Fukaya and K.\ Ono, Arnold conjecture and
  Gromov--Witten invariants, {\it Topology}  {\bf 38} (1999), 933--1048.

\bibitem[FOh]{FOh} K.\ Fukaya, Y.-G.\ Oh, 
  Zero-loop open strings in the cotangent bundle and Morse homotopy. {\it
  Asian J.Math.} 1 (1997), no.1, 96--180.

\bibitem[FOOO]{FOOO} K.\ Fukaya, Y.-G.\ Oh, H.\ Ohta, and K.\ Ono,
  {\it Lagrangian Intersection Theory, Anomaly and Obstruction, Parts I
  and II}, AMS/IP Studies in Advanced Mathematics, Amer.\ Math.\ Soc.\
  and Internat.\ Press.

\bibitem[Ge]{Gerst}
  A.~Gerstenberger \emph{Universal Moduli Spaces in Gromov-Witten
  Theory}, Ph.D. thesis, 2012.

\bibitem[Gr]{G}
  M.~Gromov, \emph{Pseudoholomorphic curves in symplectic manifolds},
  Invent. Math {\bf 82} (1985), no. 2, 307-347.

\bibitem[H1]{Hofer0}
  H.~Hofer, \emph{A General Fredholm Theory and Applications}, Current
  Developments in Mathematics, edited by D.~Jerison, B.~Mazur, T.~Mrowka,
  W.~Schmid, R.~Stanley, and S.T.~Yau, International Press 2006

\bibitem[H2]{Hofer}
  H.~Hofer, \emph{Polyfolds and a general Fredholm theory},
  arXiv:~0809.3753.

%\bibitem[H3]{Hofer2}
%  H.~Hofer, \emph{Lecture on Polyfolds and Applications I: Basic Concepts
%  and Illustrations}, in preparation.

\bibitem[HLS]{HLS97}
  H.~Hofer, V.~Lizan, J.-C.~Sikorav, \emph{On genericity for holmorphic
  curves in four-dimensional almost-complex manifolds}, J. Geom. Anal. {\bf
  7} (1997), no. 1, 149-159.

\bibitem[HS]{Hofer-Salamon} H.~Hofer, D.~Salamon,
  Floer homology and Novikov rings,
  {\it The Floer Memorial Volume}, 
  Birkh\"auser 1995, pp 483--524.

\bibitem[HWZIII]{HWZ99}
  H.~Hofer, K.~Wysocki, E.~Zehnder, \emph{Properties of pseudo-holomorphic
  curves in symplectizations III. Fredholm theory}, Topics in non-linear
  analysis, 1999, pp. 381-475.

\bibitem[HWZ0]{HWZ_lectures}
  H.~Hofer, K.~Wysocki, E.~Zehnder, \emph{Polyfolds and Fredholm Theory},
  lecture notes from the first Symplectic Field Theory workshop the
  Mathematisches Institut University of Leipzig, 2005.

\bibitem[HWZ1]{HoferWysockiZehnder1} H.~Hofer, K.~Wysocki, E.~Zehnder,
  \emph{A General Fredholm Theory I: A Splicing-Based Differential
  Geometry}, {\it JEMS} 9:4 (2007), 841--876.

\bibitem[HWZ2]{HoferWysockiZehnder2}
  H.~Hofer, K.~Wysocki, E.~Zehnder, \emph{A general Fredholm theory
  II: Implicit function theorems}, GAFA {\bf 19} (2009), no. 1,
  206--293.

\bibitem[HWZ3]{HoferWysockiZehnder3}
  H.~Hofer, K.~Wysocki, E.~Zehnder, \emph{A general Fredholm theory
  III: Fredholm functors and polyfolds}, Geom. Topol. {\bf 13} (2009),
  Issue 4, 2279--2387.

\bibitem[HWZ4]{HoferWysockiZehnder4}
  H.~Hofer, K.~Wysocki, and E.~Zehnder, \emph{A general Fredholm theory
  IV: Operations and Orientations}, in preparation.

\bibitem[HWZ5]{HWZscSmooth} H.~Hofer, K.~Wysocki, and  E.~Zehnder,
  \emph{Sc-Smoothness, Retractions and New Models for Smooth Spaces},
  Discrete and Continuous Dynamical Systems, Vol 28 (No 2), October
  2010, 665-788.

\bibitem[HWZ6]{HWZ_Integration}
  H.~Hofer, K.~Wysocki, and E.~Zehnder, \emph{Integration Theory on the
  Zero Set of Polyfold Fredholm Sections}, Math. Ann. Vol 336, Issue 1
  (2010), 139-198.

\bibitem[HWZ7]{HWZ_DM}
  H.~Hofer, K.~Wysocki,  and E.~Zehnder, \emph{Deligne-Mumford type
  spaces with a View Towards Symplectic Field Theory}, lecture note
  in preparation.

\bibitem[HWZ8]{HWZI_applications}
  H.~Hofer, K.~Wysocki,  and E.~Zehnder, \emph{Applications of Polyfold
  Theory I: Gromov-Witten Theory}, arXiv: 1107.2097.

\bibitem[HWZ9]{HWZII_applications}
  H.~Hofer, K.~Wysocki,  and E.~Zehnder, \emph{Applications of Polyfold
  Theory II: The Polyfolds of Symplectic Field Theory}, in
  preparation.

\bibitem[HWZ10]{HWZnew}
  H.~Hofer, K.~Wysocki, E.~Zehnder, \emph{Polyfold and Fredholm Theory I:
  Basic Theory in M-Polyfolds}, arXiv:1407.3185.


\bibitem[HWZ11]{hwzbook} H.~Hofer, K.~Wysocki,  and E.~Zehnder,
{\it Lectures on Polyfolds and Applications I: Basic Concepts and Illustrations}, 2009 draft.



\bibitem[HWZ12]{HWZ_DetBundles}
  H.~Hofer, K.~Wysocki,  and E.~Zehnder, \emph{Connections and
  Determininant Bundles for Polyfold Fredholm Operators}, in preparation.

\bibitem[I]{Io}  E.~Ionel,  {\it GW invariants relative normal crossings
  divisors}, arXiv:1103.3977

\bibitem[J]{joyce} D.\ Joyce, {\it D-manifolds, d-orbifolds and derived
  differential geometry: a detailed summary},  arXiv:1208.4948.

\bibitem[LiT]{LiT}  J.~Li and G.~Tian, 
  Virtual moduli cycles and Gromov--Witten invariants for general
  symplectic manifolds, {\it Topics in Symplectic $4$-manifolds (Irvine
  CA 1996)}, Internat.\ Press, Cambridge, MA (1998), 47--83.

\bibitem[LiuT]{LiuT}  G.~Liu and G.~Tian, Floer homology and Arnold
  conjecture, {\it Journ.\ Diff.\ Geom.}, {\bf 49} (1998), 1--74.

\bibitem[Li]{jiayong} J.~Li, \emph{A polyfold set up for moduli spaces
  of Morse trees with holomorphic disks}, Ph.D.\ Thesis, Massachusetts
  Institute of Technology, work in progress.

\bibitem[LW]{jiayong-w}
  J.~Li and K.~Wehrheim, {\it $A_\infty$-structures from Morse trees with pseudoholomorphic disks}, preprint available at \href{http://math.berkeley.edu/~katrin/papers/disktrees.pdf}{math.berkeley.edu/~katrin/papers/disktrees.pdf}.


\bibitem[Mc]{McD} D.~McDuff, The local behaviour of holomorphic
  curves in almost complex 4-manifolds, \emph{J.\ Diff.\ Geom.} {\bf 34}
  (1991), 143--164.

\bibitem[Mo]{moer} I.~Moerdijk, \emph{Orbifolds as Groupoids: an Introduction}, 
arXiv:math/0203100. 

\bibitem[MS]{MDSa} D.~McDuff and  D.~Salamon, \emph{$J$-holomorphic
  curves and symplectic topology}, AMS Colloquium Publications {\bf
  52}, 2004.

\bibitem[Mu]{Munkres} J. Munkres,  \emph{Topology} (2nd Edition),
  Prentice Hall, 2000.

\bibitem[MW]{mcduff-wehrheim} D.~McDuff and K.~Wehrheim,
  \emph{Smooth Kuranishi atlases with trivial isotropy},
  arXiv:1208.1340.

\bibitem[Oh]{oh}
  Y.-G.~Oh.
  \newblock Floer cohomology of {L}agrangian intersections and
  pseudo-holomorphic disks. {I}.
  \newblock {\em Comm. Pure Appl. Math.}, 46(7):949--993, 1993.
  and Addendum,
  \newblock
  {\em Comm. Pure Appl. Math.} 48 (1995), no. 11, 1299--1302.

\bibitem[PSS]{PSS}
  S.~Piunikhin, D.A.~Salamon, and M.~Schwarz, Symplectic
  {F}loer-{D}onaldson theory and quantum cohomology, in {\em Contact and
  symplectic geometry}, {\em Publ. Newton Inst.}  8 (1996), 171--200.

\bibitem[Sa]{Salamon} D.~Salamon, \emph{Lectures on Floer homology},
  Symplectic Geometry and Topology,  Symplectic Geometry and Topology,
  edited by Y. Eliashberg and L. Traynor, IAS/Park City Mathematics
  series, Vol 7, 1999, 143--230.

\bibitem[Sc1]{schwarz:Morse} M.~Schwarz,
  {\it Morse Homology}, Birkh\"auser, 1993.

\bibitem[Sc2]{Sch} M.~Schwarz, \emph{Cohomology operations from
  $S^1$-cobordisms in Floer homology}, Ph.D. thesis, Swiss Federal Inst. of
  Techn. Zurich, Diss. ETH No. 11182, 1995.

\bibitem[Se]{Seidel} P.~Seidel,
  {\it Fukaya Categories and Picard--Lefschetz theory},
  Zurich Lectures in Advanced Mathematics,
  European Math. Soc. (EMS), Zurich, 2008.

\bibitem[Si]{Siebert} B. Siebert, Symplectic Gromov--Witten invariants,
  in {\it New Trends in Algebraic Geometry, (Warwick 1996)}, 375--424,
  London Math Soc. Lecture Notes Ser {\bf 264},
  Cambridge Univ. Press, Cambridge, 1999.
        
\bibitem[T]{Tr} H.~Triebel, \emph{Interpolation Theory, Function Spaces,
  Differential Operators}, North-Holland, Amsterdam (1978) Zbl 0387.46033
  MR 0500580

\bibitem[W0]{w:msri-talk} K.~Wehrheim, \emph{Analytic foundations
  and holomorphic disks}, slides from the Introductory Workshop:
  Symplectic and Contact Geometry and Topology, MSRI 2009, available at
  {\tt www-math.mit.edu/$\sim$katrin}.

\bibitem[W1]{w:Morse} K.~Wehrheim, Smooth structures on Morse trajectory
  spaces, featuring finite ends and associative gluing,
  \emph{Geom. Topol. Monogr.} 18 (2012), 369--450.

\bibitem[W2]{w:fred} K.~Wehrheim, Fredholm notions in scale calculus
  and Hamiltonian Floer theory, to appear in \emph{J. Symp. Geom.}

\bibitem[WW]{WW} K.~Wehrheim and C.~Woodward, \emph{Pseudoholomorphic
  Quilts}, to appear in \emph{J. Symp. Geom.}

\bibitem[We1]{Wen05} C.~Wendl, \emph{Finite energy foliations and
  surgery on transverse links}, Ph.D. Thesis, New York University, 2005.

\bibitem[We2]{W} C.~Wendl, \emph{Automatic transversality and
  orbifolds of punctured holomorphic curves in dimension four},
  Comment. Math. Helv. 85, no. 2, 347-407 (2010).

\bibitem[Y]{Dingyu} D.~Yang, \emph{The polyfold -- Kuranishi
  correspondence}, work in progress.

\end{thebibliography}
\end{document}